\documentclass[11pt]{amsart}
\usepackage[utf8]{inputenc}
\usepackage{amsfonts, mathtools, nccmath, amssymb, graphicx, multicol, commath, bbm, mathrsfs, bm, subfig}
\usepackage{amsthm, mathabx}
\usepackage[foot]{amsaddr}
\usepackage{esint}
\usepackage{braket}
\usepackage{color,dsfont}
\usepackage[top=1in, bottom=1in, left=1in, right=1in]{geometry}
\usepackage{tikz, graphicx, subfig, booktabs}
\usepackage{hyperref}
\hypersetup{colorlinks,breaklinks,
             linkcolor=blue,urlcolor=blue,
             anchorcolor=blue,citecolor=blue}

\usepackage{cite}

\newcommand{\restr}{%
  \,\raisebox{-.127ex}{\reflectbox{\rotatebox[origin=br]{-90}{$\lnot$}}}\,%
}
\newcommand{\R}{\mathbb{R}}
\renewcommand{\S}{\mathbb{S}}
\newcommand{\Pd}[1][d]{\mathbb{P}_{#1}}

\newcommand{\N}{\mathbb{N}}
\renewcommand{\P}{\mathscr{P}}

\newcommand{\M}{\mathcal{M}}
\newcommand{\V}{\mathcal{V}}
\newcommand{\Vcal}{\mathcal{V}}

\newcommand{\Sch}{\mathcal{S}}

\newcommand{\Ex}{\mathbb{E}}
\newcommand{\prob}{\mathbb{P}}
\renewcommand{\L}[1][]{\mathscr{L}}

\newcommand{\E}{\mathcal{E}}

\renewcommand{\H}{\mathcal{H}}
\newcommand{\A}{\mathcal A}
\newcommand{\F}{\mathcal{F}}

\newcommand{\B}{\mathscr{B}}

\newcommand{\CE}{\mathcal{CE}}

\newcommand{\one}{\mathds{1}}
\newcommand{\eps}{\varepsilon}

\newcommand{\rv}{\mathfrak{v}}
\newcommand{\ru}{\mathfrak{u}}
\newcommand{\rf}{\mathfrak{f}}
\newcommand{\rg}{\mathfrak{g}}
\newcommand{\rj}{\mathfrak{J}}

\DeclareMathOperator{\dist}{dist}
\DeclareMathOperator{\supp}{supp}

\DeclareMathOperator{\id}{Id}

\DeclareMathOperator{\spanset}{Span}
\DeclareMathOperator{\tanspace}{Tan}

\DeclareMathOperator{\Lip}{Lip}

\DeclareMathOperator{\diff}{diff}

\DeclareMathOperator{\re}{Re}

\DeclareMathOperator*{\essinf}{essinf}
\DeclareMathOperator*{\esssup}{esssup}
\DeclareMathOperator{\vol}{Vol}
\DeclareMathOperator{\med}{med}

\newtheorem{theorem}{Theorem}[section]

\newtheorem{lemma}[theorem]{Lemma}
\newtheorem{proposition}[theorem]{Proposition}
\newtheorem{corollary}[theorem]{Corollary}
\theoremstyle{definition}
\newtheorem{definition}[theorem]{Definition}
\newtheorem{example}[theorem]{Example}

\newcommand{\red}{\normalcolor} 
\newcommand{\blue}{\normalcolor} 
\newcommand{\vio}{\normalcolor} 
\newcommand{\purp}{\normalcolor} 
\newcommand{\grn}{\normalcolor} 
\newcommand{\nc}{\normalcolor}

\definecolor{byzantine}{rgb}{0.74, 0.2, 0.64}
\definecolor{darkgreen}{rgb}{0.1,0.6,0.1}
\definecolor{darkred}{rgb}{0.6,0,0}
\definecolor{lightgray}{rgb}{0.5,0.5,0.5}

\newcommand{\te}{\textrm}

\newenvironment{listi}  
  {\begin{list} 
 {(\roman{broj})}
{ \usecounter{broj}}
     \setlength{\labelwidth}{25pt}
  }
{   \end{list} }
\newcounter{broj}
\renewcommand{\hat}{\widehat}
\renewcommand{\check}{\widecheck}
\def\Xint#1{\mathchoice
   {\XXint\displaystyle\textstyle{#1}}%
   {\XXint\textstyle\scriptstyle{#1}}%
   {\XXint\scriptstyle\scriptscriptstyle{#1}}%
   {\XXint\scriptscriptstyle\scriptscriptstyle{#1}}%
   \!\int}
\def\XXint#1#2#3{{\setbox0=\hbox{$#1{#2#3}{\int}$}
     \vcenter{\hbox{$#2#3$}}\kern-.5\wd0}}

\def\dashint{\Xint-}

\newtheorem{remarkx}[theorem]{Remark}
\newenvironment{remark}{\pushQED{\qed}\remarkx}{\popQED\endremarkx}

\numberwithin{equation}{section}
\newenvironment{proofsketch}{%
  \proof}{\endproof}

\title{\vspace{-5mm}Geometry and analytic properties of the sliced Wasserstein space}
\author{Sangmin Park} 
\author{Dejan Slep\v{c}ev}
\address{S. Park, D. Slep\v{c}ev: Department of Mathematical Sciences, Carnegie Mellon University, 5000 Forbes ave., Pittsburgh, PA 15213}
\email{sangminp@andrew.cmu.edu, slepcev@math.cmu.edu}
\begin{document}

\begin{abstract}
The sliced Wasserstein metric compares probability measures on $\R^d$ by taking averages of the Wasserstein distances between projections of the measures to lines. The distance has found a range of applications in statistics and machine learning, as it is easier to approximate and compute in high dimensions than the Wasserstein distance. While the geometry of the Wasserstein metric is quite well understood, and has led to important advances, very little is known about the geometry and metric properties of the  sliced Wasserstein (SW) metric.  
Here we show that when the measures considered are ``nice'' (e.g. \vio bounded above and below by positive multiples of the Lebesgue measure \nc) then the SW metric  is comparable to the (homogeneous) negative Sobolev norm $\dot H^{-(d+1)/2}$. On the other hand when the measures considered are close in the infinity transportation metric to a discrete measure, then the SW metric between them is close to a multiple of the Wasserstein metric. We characterize the tangent space of the SW space, and show that the speed of curves in the space can be described by a quadratic form, but that the SW space is not a length space. We establish a number of properties of the metric given by the minimal length of curves between measures -- the SW length. Finally we highlight the consequences of these properties on the gradient flows in the SW metric.
\end{abstract}

\maketitle



\bigskip
\setcounter{tocdepth}{1} 
\tableofcontents

\newpage

\subsection*{Notation}
\begin{itemize} \addtolength{\itemsep}{3pt}
\addtolength{\itemindent}{-35pt}
\item[ ] $\Pd$ -- the set we call the Radon domain is defined as $\Pd=(\S^{d-1}\times\R)/\!\sim$, the quotient space for the equivalence relation $(-\theta,-r)\sim(\theta,r)$; see \eqref{def:Pd}
\item[] $Rf=\hat f$ -- the Radon transform of function $f:\R^d\rightarrow\R$; see \eqref{def:RadonTrans}
\item[] $R^\ast \rg=\widecheck{\rg}$ -- the dual Radon transform of function $\rg:\Pd\rightarrow\R$; see \eqref{def:dualRadon}
\item[] $\Sch(\Omega)$ with $\Omega=\R^d$ or $\Pd$ -- the Schwartz class of functions; see \eqref{def:Sch_radon}
\item[] $\Sch'(\Omega)$ with $\Omega=\R^d$ or $\Pd$ -- the set of tempered distributions on $\Omega$
\item[] $\M(\Omega)$ with $\Omega=\R^d$ or $\Pd$ -- the set of locally finite Borel measures on $\Omega$; see \eqref{def:MPd}
\item[] $\M_b(\Omega)$ with $\Omega=\R^d$ or $\Pd$ -- the set of bounded Borel measures on $\Omega$
\item[] $\P_p(\R^d)$, $\P_p(\Pd)$ -- sets of probability measures with bounded $p$-th moments; see \eqref{def:PpPd}
\item[] $W_p$ -- the $p$-transportation distance; see \eqref{def:Wp}. We write $W$ for the Wasserstein distance $W_2$. 
\item[] $SW_p$ --  the $p$-sliced Wasserstein distance; see \eqref{def:SWp}. We write $SW$ for $SW_2$.
\item[] $\Gamma(\mu,\nu)$ -- set of transport plans between probability measures $\mu,\nu$; see \eqref{def:Wp}
\item[] $\widehat\Gamma(\hat\mu,\hat\nu)$ -- set of slice-wise transport plans between probability measures $\hat\mu,\hat\nu\in\P(\Pd)$; see \eqref{def:hatGamma}
\item[] $\Gamma_o(\mu,\nu)$ --  set of optimal transport plans between $\mu,\nu \in \P_2(\R^d)$ for the quadratic cost; see \eqref{def:Gammao}
\item[] $T_\mu^\nu$ -- optimal transport map for quadratic cost between $\mu\in\P_2(\R^d)$ and $\nu\in\P_2(\R^d)$
\item[] $\widehat\Gamma_o(\hat\mu,\hat\nu)$ -- set of slice-wise optimal transport plans between $\hat\mu,\hat\nu \in \P_2(\Pd)$; see \eqref{def:hatGammao}
\item[] $\widehat T_{\hat\mu}^{\hat\nu}$ -- slice-wise optimal transport map from $\hat\mu\in\P_2(\Pd)$ to $\hat\nu\in\P_2(\Pd)$
\item[] $\F_d f = (2\pi)^{-d/2}\int_{\R^d} e^{-ix}f(x)\,dx$ -- the $d$-dimensional Fourier transform of a function $f:\R^d\rightarrow\R$
\item[] $\F_1 \rg= (2\pi)^{-1/2}\int_{\R}e^{-ir}\rg(\theta,r)\,dr$ --  the slice-wise $1$-dimensional Fourier transform of $\rg:\Pd\rightarrow\R$
\item[] $\Lambda_d=(-\partial_r^2)^{\tfrac{d-1}{2}}$ -- slice-wise fractional derivative  applied to functions on $\Pd$; see \eqref{eq:Lambdad-laplace} and \eqref{def:hatLambda}
\item[] $H_t^s(\Omega)$ with $\Omega=\R^d$ or $\Pd$ and $s\in\R$, $t>-\frac d2$ --  Sobolev space with attenuated/amplified low frequencies; see \eqref{def:Hts_norm} and \eqref{def:Hts_norm;radon}. We write $H_t^{s-}(\Omega)=\bigcap_{\eps>0}H_t^{s-\eps}$; see \eqref{def:Hscr}
\item[] \grn $\mathscr{H}(\R^d;\R^d)$ -- space of admissible fluxes for the continuity equation; see \eqref{def:Hscr} \nc
\item[] $\CE$ --  set of suitable distributional solutions of the continuity equation; see Definition~\ref{def:CE}
\item[] \grn$\langle f,g\rangle_{\Omega}$ with $\Omega=\R^d$ or $\Pd$ -- action of a distribution $f$ on $\Omega$ on a function $g$ on $\Omega$; see \eqref{def:bracket_Omega}\nc
\end{itemize}

\section{Introduction}

The sliced Wasserstein distance, introduced by Rabin, Peyr\'e, Delon, and Bernot~\cite{RabPeyDelBer12}, compares probability measures on $\R^d$ by taking averages of the Wasserstein distances between projections of the measures to each 1-dimensional subspaces of $\R^d$. Thanks to its lower sample and computational complexity relative to the Wasserstein distance, the sliced Wasserstein distance and its variants~\cite{DeshMax19,KNSUBRR19,Bonet23,NW_Rig22,PatyCuturi19,BSTK23,NguyenHoPhamBui21,NgueynHo22} have recently expanded its applications in statistics~\cite{NGSK22,Lin_etal21,Nadjahi21,ManBalWas22, NDCKSS20} and machine learning~\cite{KSP16,Bonet22,LSMDS19,DaiSel20,DZS18,DeshMax19,KPMR19,LBBHU19} as a tool to compare measures and construct paths in spaces of measures. \vio For $p\in[1,+\infty)$, we write $SW_p$ (or the $SW_p$ distance) to refer to the $p$-sliced Wasserstein distance, and refer to the corresponding space as the $SW_p$ space. When $p=2$, we drop the subscript and simply refer to them as the $SW$ distance and the $SW$ space.\nc

Despite the multitude of uses of the sliced Wasserstein distance there are only a few works dealing with its metric and geometric properties: Bonnotte~\cite{Bon13} established that $SW_p$ is indeed a metric, and is equivalent to $W_p$ for measures supported on a common compact set for all $p\geq 1$; more recently, Bayraktar and Guo~\cite{BayGuo21} showed that $SW_p$ and $W_p$ induce the same topology on $\P_p(\R^d)$ for $p\geq 1$. This is in stark contrast with the Wasserstein metric, whose geometry has been a subject of intense study and has led to important advances; see~\cite{AGS, San15, Vil09}. 

Here we take steps towards a better understanding of the sliced Wasserstein distance and its geometry. In particular we show that for measures that
are absolutely continuous with respect to the Lebesgue measure, have bounded density, and differ only within a set compactly contained in the interior of their support, the SW metric is comparable to the (homogeneous) negative Sobolev norm $\dot H^{-(d+1)/2}$; see Theorem~\ref{thm:lsw-sw-H;comparison}. On the other hand when the measures considered are close in the infinity-transportation metric to a discrete measure, the SW metric between them is close to a multiple of the Wasserstein metric (Theorem~\ref{thm:SWnear-discrete}). We show that, unlike the Wasserstein space, the SW space is not a length space. Nevertheless, it still has a tangential structure that resembles the one of the Wasserstein space. We also show that geodesics (considered as length minimizing curves) in SW exist and study the intrinsic metric $\ell_{SW}$, defined as the length of the minimizing geodesics between measures. In particular we show that $\ell_{SW}$ satisfies some of the similar comparison and approximation properties as the SW metric. Finally we discuss the consequences of these properties to gradient flows with respect to the SW metric.

\subsection{Setting}\label{ssec:setting}
{\red For $\Omega=\R^{d},\S^{d-1}\times\R$ we denote by $\P(\Omega)$ the space of all Borel probability measures on $\Omega$.\nc}
Let $1\leq p<\infty$. For probability measures  \grn $\mu,\nu\in \P_p(\R^d) := \{ \mu \in \P(\R^d) \::\: \int_{\R^d} |x|^p d\mu(x) < \infty \}$\nc, the $p$-Wasserstein distance, $W_p$, is defined as follows: 
\begin{equation}\label{def:Wp}
    \begin{split}
        W_p(\mu,\nu) := & \inf_{\gamma\in\Gamma(\mu,\nu)}\left(\int_{\R^d\times \R^d} |x-y|^p \,d\gamma(x,y)\right)^{1/p} \\ 
        &\text{ where }
        \Gamma(\mu,\nu)=\left\{\gamma\in\P(\R^d \times \R^d):\,\pi^1_\#\gamma=\mu,\,\pi^2_\#\gamma=\nu\right\}.
    \end{split}
\end{equation}

To define the sliced Wasserstein distance, we introduce the following notation: for each $\theta\in\S^{d-1}$, let $\R\theta:=\spanset{\{\theta\}}\subset\R^d$, and define the projection $\pi^\theta:\R^d\rightarrow\R\theta$ by
\begin{equation}\label{def:pitheta}
\pi^\theta(x)=(\theta\cdot x)\theta.
\end{equation}
The $p$-sliced Wasserstein distance  $SW_p$ is defined by
\begin{equation}\label{def:SWp}
    SW_p(\mu,\sigma)=\left(\dashint_{\S^{d-1}}W_p^p(\pi^\theta_\# \mu,\pi^\theta_\#\sigma)\,d\theta\right)^{\frac1p}.
\end{equation}

The Radon transform provides a natural language to describe objects relating to the sliced Wasserstein distance. Consider an integrable function $f:\R^d\rightarrow \R$, for  $d\geq 2$. We use both $Rf$ and $\hat f$ to  denote its Radon transform: For $\theta \in \S^{d-1}$ and $r \in\R$
\begin{equation}\label{def:RadonTrans}
R_\theta f(r)=Rf(\theta, r)=\widehat f(\theta,r):=\int_{\theta^\perp} f(r\theta+y^\theta)\,dy^\theta,
\end{equation}
where $\theta^\perp:=\{y\in\R^d:\;y\cdot\theta=0\}$ and $dy^\theta$ is the $(d-1)$-dimensional Lebesgue measure on $\theta^\perp$. By Fubini's theorem, $R_\theta f(r)$ exists for $\L^1$ a.e. $r\in\R$ for each $\theta\in\S^{d-1}$ when $f\in L^1(\R^d)$.

Note that $Rf$ is even on $\S^{d-1}\times\R$,  meaning that  $Rf(-\theta,-r)=Rf(\theta,r)$. This motivates defining the $d$-dimensional ``Radon domain'' $\Pd$ by
\begin{equation}\label{def:Pd}
    \Pd:=(\S^{d-1}\times\R)/\sim \text{, where the equivalence relation is given by } (-\theta,-r)\sim(\theta,r).
\end{equation}

The Radon transform can be extended to distributions (see~\cite[Chapter 1.5]{Hel10} and~\cite{Sha21}); in particular, when $\mu$ is a bounded measure, the distributional extension is consistent with the definition of $R_\theta \mu$ as a pushforward of $\mu$ by the projection map $x\mapsto x\cdot\theta$ (see Remark~\ref{rmk:radon_on_measure}).
Thus, for $\mu\in\P_p(\R^d)$ we have $\hat\mu^\theta = R_\theta \mu\in\P_p(\R)$, and
\[SW_p(\mu,\sigma)=\left(\dashint_{\S^{d-1}}W_p^p(\hat\mu^\theta,\hat\sigma^\theta)\,d\theta\right)^{\frac1p}.\]
Note that $\hat\mu^\theta\in\P_p(\R)$ whereas $\pi^\theta_\#\mu\in\P_p(\R\theta)$; we will sometimes use the latter when it is more convenient to consider measures on subspaces of $\R^d$ than on $\R$.

Henceforth we will focus on the case $p=2$, and write $SW=SW_2$ and $W=W_2$.

\subsection{Summary of results}\label{ssec:summary}
We obtain a number of geometric and analytic properties of the $SW$ distance and  of the associated length space, and investigate their implications on the sliced Wasserstein gradient flows and statistical estimation rates in the metrics. \vio Throughout this paper, we use $X\lesssim_{\alpha_1,\cdots,\alpha_k}Y$ as a shorthand for the inequality $X\leq C(\alpha_1,\cdots,\alpha_k) Y$, where $C(\alpha_1,\cdots,\alpha_k)>0$ is a finite positive constant depending on $(\alpha_{i})_{i=1}^k$. When summarizing results or making remarks, we sometimes omit the dependence on the parameters and write $\lesssim$ for the sake of simplicity; however, all rigorous statements contain clear characterizations of the constants. \nc
\smallskip

\emph{Basic properties.}
In Section~\ref{sec:SW;basic} we establish some basic  properties of the $SW$ metric. In particular in Proposition~\ref{prop:SW2_complete} we show that $(\P_2(\R^d),SW)$ is a complete metric space, which we  refer to as the $SW$ space.
We then turn to the intrinsic geometry of the $SW$ space. Example~\ref{ex:SW2_notgeo} shows that, unlike the Wasserstein space, the $SW$ space is not a geodesic space. That is, one cannot in general find a continuous curve connecting two measures in the SW space with its length equal to the distance between the measures.
\smallskip

\emph{Tangential structure of sliced Wasserstein space.}
\grn We show in Section~\ref{sec:three} that the $SW$ space has a tangent structure which resembles the tangent structure of the Wasserstein space. Recall that in the Wasserstein space, each absolutely continuous curve $(\mu_t)_{t\in I}$ defined on an interval $I\subset\R$ corresponds to a measure-valued distributional solution of the continuity equation 
\[\partial_t \mu_t+\nabla\cdot(v_t\mu_t)=0,\; \text{ with } \|v_t\|_{L^2(\mu_t)} = |\mu'|_{W}(t) \text{ for a.e. } t\in I,\]
where $|\mu'|_{W}$ is the metric derivative w.r.t $W$; see~\cite[Theorem 8.3.1]{AGS}\nc. Theorem~\ref{thm:SW2_ac-curves} establishes an analogous result for the sliced Wasserstein space --  for each absolutely continuous curve $(\mu_t)_{t\in I}$ in the sliced Wasserstein space, there exists a vector-valued flux $(J_t)_{t\in I}$ such that $J_t$ is in a suitable subspace of $\Sch'(\R^d;\R^d)$ and
\[\partial_t \mu_t+\nabla\cdot J_t=0,\; \text{ with } \norm{\tfrac{d\widehat J_t}{d\hat\mu_t}}_{L^2(\hat\mu_t)} = |\mu'|_{SW}(t) \text{ for a.e. } t\in I.\]
We note two key differences: the metric derivative is characterized by the weighted $L^2$ norm in the Radon domain, finiteness of which does not imply $J_t\ll\mu_t$ in general (see Remark~\ref{rmk:SW2_ac-curves}). Moreover, Sharafutdinov's results of Radon transform on Sobolev spaces~\cite{Sha21} imply that $\|R f\|_{L^2(\Pd)}=\|f\|_{\dot H^{-(d-1)/2}(\R^d)}$, thus we can formally understand $|\mu'|_{SW}(t)$ as corresponding to a weighted high order negative Sobolev norm of the flux $J_t$, in contrast to the weighted $L^2$-norm in the Wasserstein case. Hence, at least for absolutely continuous measures,  the formal Riemannian metric measuring infinitesimal length in the sliced Wasserstein space corresponds to a weaker space than the one for the Wasserstein metric. 
Furthermore in Section~\ref{ssec:sw-tanspace}, 
we characterize its tangent space which has the following key property analogous to its Wasserstein counterpart: Each absolutely continuous curve $(\mu_t)_{t\in I}$ is associated to a unique (up to a $\L^1\restr_I$-null set) family of tangent vectors $(J_t)_{t\in I}$, which moreover attain the metric derivative through the quadratic form $J_t \mapsto \|d\widehat J_t/d\hat\mu_t\|_{L^2(\hat\mu_t)}$.
\smallskip

\emph{Intrinsic sliced Wasserstein length space.}
In Section~\ref{sec:ellSW}, motivated by the general of lack geodesics in $(\P_2(\R^d),SW)$, we introduce the sliced Wasserstein length metric $\ell_{SW}$ defined as the infimum of the lengths of curves between measures in the SW space. We establish the basic properties of the metric space $(\P_2(\R^d),\ell_{SW})$; in particular we prove in Proposition~\ref{prop:lsw-geodesic} that geodesics exist, which further implies that in general $\ell_{SW}\neq SW$.
\smallskip

\emph{Comparison of sliced Wasserstein metric with negative Sobolev norms and Wasserstein metric.}
In Section~\ref{sec:comparison} we establish some of the key results of this paper, namely the comparison theorems of $SW$ metric with negative Sobolev norms near absolutely continuous measures and comparisons of $SW$ with the Wasserstein metric near  discrete measures. In particular, consider  an absolutely continuous measure $\mu$ bounded away from zero and infinity on some bounded open convex domain $\Omega$.  Theorem~\ref{thm:lsw-sw-H;comparison} establishes that 
\[\|\mu-\nu\|_{\dot H^{-(d+1)/2}(\R^d)}\lesssim SW(\mu,\nu)\leq \ell_{SW}(\mu,\nu)\lesssim SW(\mu,\nu) \lesssim \|\mu-\nu\|_{\dot H^{-(d+1)/2}(\R^d)},\]
for all measures $\nu$ which are \vio bounded above and below by constant multiples of $\mu$ \nc and coincide with $\mu$ near the boundary of $\Omega$.
In other words we show that near $\mu$, $SW$ is  equivalent to $\dot H^{-(d+1)/2}$.
 
 On the other hand, Theorem~\ref{thm:SWnear-discrete} states 
\[SW(\mu^n,\nu)\leq \ell_{SW}(\mu^n,\nu)\leq \frac1d W(\mu^n,\nu)\leq (1+o(1))SW(\mu^n,\nu)\]
for $\nu$ near discrete measures of the form $\mu^n=\sum_{i=1}^n m_i\delta_{x_i}$.

These two results provide interesting insights about the $SW$ \vio metric\nc. Near smooth measures it behaves like a highly negative Sobolev space, in contrast to the Wasserstein metric which for such measures behaves like the $\dot H^{-1}$ norm as noted by Peyre~\cite{RPeyre18}, while near discrete measures $SW$ behaves like the Wasserstein distance.
\smallskip

\emph{Approximation by discrete measures in sliced Wasserstein length.}
Manole, Balakrishnan, and Wasserman~\cite[Proposition 4]{ManBalWas22} have shown that a finite random sample (i.e. the empirical measure of the set of $n$ random points) of a probability measure on $\R^d$ estimates the measure in the sliced Wasserstein distance at a parametric rate $O(n^{-1/2})$ for a large class of measures; \purp see also~\cite{NDCKSS20}\nc. This is in stark contrast with the Wasserstein distance where the approximation error is poor in high dimensions and scales like $n^{-1/d}$. 
We start by pointing out a connection between the results on the parametric finite-sample estimation in the sliced Wasserstein distance and the results in statistical literature, that our results in Section~\ref{sec:comparison} identify. Namely it is known that finite-sample estimation of measures with respect to  maximum mean discrepancy (MMD) also enjoys parametric rate ~\cite[Theorem 3.3]{Sri16}. MMD distance is nothing but the norm in the dual of a reproducing kernel Hilbert space (RKHS). In particular the results of~\cite{Sri16} apply to the dual of the Sobolev space $H^s$ with $s>\frac d2$ (when the spaces embeds in the spaces of H\"older continuous functions and are RKHS). Our Theorem~\ref{thm:lsw-sw-H;comparison} says that near absolutely continuous measures, SW behaves like the $\dot H^{-(d+1)/2}$-norm; as the associated norm $\|\cdot\|_{H^{-(d+1)/2}(\R^d)}$ is an MMD, we can formally understand $SW$ to exhibit behaviors like an MMD. Thus the MMD parametric estimation can be seen as a tangential or a linearized analogue of the finite sample estimation rates in SW distance. 

Here we investigate the finite sample estimation rates in the SW intrinsic length metric $\ell_{SW}$. The goal is to gain a better understanding of the extent to which $SW$ and $\ell_{SW}$ share properties. In Theorem~\ref{thm:lsw;para_rate}, we establish that the finite sample approximation in $\ell_{SW}$ happens at the parametric rate up to a logarithmic correction, namely that 
\[\ell_{SW}(\mu,\mu^n)\lesssim \sqrt{\frac{\log n}{n}} \; \text{ with high probability, where } \mu^n=\frac1n\sum_{i=1}^n \delta_{X_i} \text{ with } X_i\overset{i.i.d.}{\sim}\mu.\]
While this  is consistent with the geometric view of $(\P_2(\R^d),\ell_{SW})$ as a curved or nonlinear dual of Reproducing Kernel Hilbert Space (see beginning of Section~\ref{sec:lsw;stat} for discussion), the statement and the proof requires dealing with discrete measures where such heuristic view does not hold. 
\smallskip

\emph{Implications on gradient flows.}
Section~\ref{sec:SWGF} applies the comparison results on $\ell_{SW}$, $SW$ to obtain comparisons for the metric slopes. Given a metric space $(X,m)$, recall that metric slope $|\partial\E|_m$ of a functional $\E:X\rightarrow \R$ is defined by
    \begin{equation}\label{def:met_slope}
        |\partial\E|_m(u)=\limsup_{v\xrightarrow[]{d}u}\frac{[\E(u)-\E(v)]_+}{m(u,v)}.
    \end{equation}
Let $V:\R^d\rightarrow \R$ and consider the potential energy $\Vcal(\mu):=\int_{\R^d} V(x)\,d\mu(x)$. Proposition~\ref{prop:metslope;potential-ac} states that when $V$ is smooth and compactly supported, for suitable absolutely continuous $\mu\in\P_2(\R^d)$ it holds that
\[|\partial\Vcal|_{\dot H^{(d+1)/2}(\R^d)}(\mu)\lesssim |\partial\Vcal|_{\ell_{SW}}(\mu)\leq |\partial\Vcal|_{SW}(\mu)\lesssim |\partial\Vcal|_{\dot H^{(d+1)/2}(\R^d)}(\mu) \]
whereas Proposition~\ref{prop:swslope;potential} shows that the slope behaves quite differently at discrete measures, $\mu^n=\sum_{i=1}^n m_i\delta_{x_i}$, namely that
\[|\partial\Vcal|_{SW}(\mu^n)=|\partial\Vcal|_{\ell_{SW}}(\mu^n)=\sqrt{d}\,|\partial\Vcal|_{W}(\mu^n).\]
\vio By considering a sequence of discrete measures $\mu^n$ converging to an absolutely continuous measure $\mu$, we may deduce that \nc  $|\partial\Vcal|_{SW}$ (resp. $|\partial\Vcal|_{\ell_{SW}}$) is not lower semicontinuous in $SW$ (resp. $\ell_{SW}$) in general, even when $V\in C_c^\infty(\R^d)$; see Corollary~\ref{cor:relaxed_slope;pot}. 
This implies that the potential energy is not $\lambda$-geodesically convex in $(\P_2(\R^d),\ell_{SW})$. Furthermore we observe in Remark~\ref{rmk:SWGF;instability} that starting from discrete measures with finite number of particles, the curves of maximal slope in the Wasserstein space, after a constant rescaling of time, are the curves of maximal slope in the SW space. 

On the other hand, for smooth measures, the curves of maximal slope with respect to the  Wasserstein metric are not curves of maximal slope in the SW space. We formally show that SW gradient flows of the potential energy \vio satisfy \nc a higher order equation given by a pseudodifferential operator of order $d$, which is consistent with the rigorous results of Proposition~\ref{prop:metslope;potential-ac}. We conclude that the framework of gradient flows in metric spaces would not be the right tool to study such equations; PDE based approaches may provide an avenue for creating a well-posedness theory, which remains an open problem. 

\subsection{Related works}\label{ssec:relatedworks}

Since the introduction of the sliced Wasserstein distance~\cite{RabPeyDelBer12}, numerous variants have been considered. Deshpande et al.~\cite{DeshMax19} proposed the max-sliced Wasserstein distance (max-SW distance), which is the maximum of the 1D Wasserstein distances, instead of the average as in the $SW$ case. Niles-Weed and Rigollet~\cite{NW_Rig22} and Paty and Cuturi~\cite{PatyCuturi19} independently proposed the $k$-dimensional generalization (max-$k$-SW distance) for $1\leq k\leq d$. Generalizations to spherical~\cite{Bonet23} and other nonlinear projections~\cite{KNSUBRR19} have also been considered to more effectively capture the geometric structure of data. Based on the ideas of partial optimal transportation~\cite{Figalli09}, Bai, Schmitzer, Thorpe, and Kolouri~\cite{BSTK23} introduced sliced optimal partial transport to compare of measures with different masses. Further projection-based transport metrics include the distributional sliced Wasserstein distance introduced by Nguyen, Ho, Pham, and Bui~\cite{NguyenHoPhamBui21} and the convolution sliced Wasserstein distance proposed by Nguyen and Ho~\cite{NgueynHo22}. 

The sliced Wasserstein distances have found numerous applications in image processing. In fact, utility of the sliced Wasserstein barycenter for tasks such as image synthesis, color transfer, and texture mixing served as a motivation behind the introduction of sliced Wasserstein distance~\cite{RabPeyDelBer12}. Bonneel, Rabin, Peyr\'e, and Pfister~\cite{BoRaPePf14} further studied efficient numerical methods to compute sliced Wasserstein and related barycenters, and their applications. Kolouri, Park, and Rhode proposed the Radon cumulative distribution transport (Radon CDT)~\cite{KSP16} for image classification; Radon CDT effectively computes the sliced Wasserstein `geodesic', by taking the Radon inverse of the displacement interpolation between the Radon transform of the measures. However, we note that such inverse will in general fail to be a curve in the space of probability measures, as the Radon inverse of nonnegative functions need not be nonnegative. 

Gradient flows related to the sliced Wasserstein distance have been applied to various machine learning and image processing tasks. Bonnotte noticed~\cite{Bon13} that the continuous analogue of the the isotropic Iterative Distribution Transfer (IDT) algorithm, introduced by~\cite{PKD07} to transfer the color palette of a reference picture to a target picture, is the Wasserstein gradient flow of $\mu\mapsto \frac12 SW^2(\mu,\sigma)$. Liutkus et al.~\cite{LSMDS19} utilizes the Wasserstein gradient flows of the entropy-regularized version of the same energy functional for generative modelling. Gradient flow of the same energy in the sliced Wasserstein space have been considered by Bonet et al. 
\cite{Bonet22} also for generative modelling; \purp they also study the JKO scheme with respect to $SW$, and establish existence and uniqueness of minimizers of the scheme when the optimization is restricted to probability measures supported on a common compact set~\cite[Section 3.2]{Bonet22}\nc. Sliced Iterative Normalizing Flows (SINF)~\cite{DaiSel20}, useful for sampling and density evaluation, can be seen as a max-SW variant of the isotropic IDT algorithm.

Other applications in machine learning include: sliced Wasserstein generative adversarial nets by Deshpande, Zhang, Schwing~\cite{DZS18}; max-SW generative adversarial nets~\cite{DeshMax19} for generative modelling; sliced Wasserstein autoencoder by~\cite{KPMR19}; and use of $SW_1$ distance for unsupervised domain adaptation~\cite{LBBHU19}.

On the statistical side, Manole, Balakrishnan, and Wasserman~\cite{ManBalWas22} established, based on the 1-dimensional results by Bobkov and Ledoux~\cite{BobLed19}, the parametric estimation rate $\Ex SW(\mu^n,\mu) \lesssim n^{-1/2}$ for the empirical measure $\mu^n$ of $n$ i.i.d samples of $\mu$, and further investigated statistical properties of the trimmed sliced Wasserstein distances. Nietert, Goldfeld, Sadhu, and Kato~\cite{NGSK22} established empirical estimation rate in $SW$ and max-$SW$ for log-concave distributions with explicit constants dependent on the intrinsic dimension, and showed robustness to data contamination and explored efficient computational methods. Lin, Zheng, Chen, Cuturi, and Jordan~\cite{Lin_etal21} investigated the max-$k$-sliced distances and their corresponding integral variants integral projection robust Wasserstein (IPRW) distance, also known as the $k$-sliced Wasserstein distances ($k$-SW distances), and established several statistical properties including sample complexity $O(n^{-1/k})$. More recently, Olea, Rush, Velez, and Wiesel~\cite{ORVW23} explored the connection between a certain linear predictor problems and distributionally robust optimization based on a modified max-SW. For applications in Approximate Bayesian Computation, we refer the readers to~\cite[Chapter 4]{Nadjahi21} and the references therein.

Regarding analytic and topological properties, Bonnotte~\cite{Bon13} showed that $SW_p$ is indeed a distance on $\P_p(\R^d)$ for $1\leq p<\infty$, and established that, for measures supported on $\Omega\subset\subset\R^d$ \grn-- i.e. $\Omega$ compactly contained in $\R^d$ -- \nc we have
\[SW_p\leq W_p\lesssim_{\Omega} SW_p^{\frac{1}{2p(d+1)}}.\]
More recently, Bayraktar and Guo~\cite{BayGuo21} showed that $SW_p$, $W_p$, and the $p$-max-sliced Wasserstein distance induce the same topology on $\P_p(\R^d)$ for $p\geq 1$. We note here that this does not directly imply completeness of $(\P_2(\R^d),SW)$, as not all Cauchy sequences in $SW$ need be Cauchy in $W$.

Bonnotte also showed the existence of the Wasserstein gradient flow of the energy functional $\mu\mapsto \frac12 SW^2(\mu,\sigma)$ for the target measure $\sigma$, despite the lack of geodesic convexity of the energy functional, and derived the corresponding PDE~\cite[Chapter 5]{Bon13}, the continuous-time version of the previously mentioned isotropic IDT algorithm. Due to the lack of convexity of the energy functional, even the asymptotic convergence of the gradient flow remains open. Nevertheless, Li and Moosm\"{u}ller~\cite{LiMo23} recently established almost sure convergence of the discrete isotropic IDT algorithm with step-sizes satisfying certain summability conditions. \grn More recently, Cozzi and Santambrogio~\cite{CozSan24} established the convergence rate $t\mapsto SW^2(\mu_t,\sigma)=O(t^{-1})$ when the target measure $\sigma$ is any isotropic Gaussian. \nc

As our work was nearing completion, we became aware of the  independent work by Kitagawa and Takatsu~\cite{KitagawaTakatsu23} on the sliced Wasserstein spaces. In their work, Kitagawa and Takatsu establish in  the metric completeness of sliced optimal-transportation-based spaces, which generalizes our Proposition~\ref{prop:SW2_complete}, and also demonstrate that the SW spaces are not geodesic spaces, generalizing our Example~\ref{ex:SW2_notgeo}.
Their work focuses on isometrically embedding the SW type spaces into larger spaces and the barycenter problem. \purp See also~\cite{kitagawaTakatsu2024} for their more recent work on disintegrated optimal transport for metric fiber bundles.\nc

\section{Basic properties of the sliced Wasserstein space}\label{sec:SW;basic}

In this section, we examine the basic properties of the sliced Wasserstein space $(\P_2(\R^d),SW)$. We start by reviewing a few properties of the Radon transform that we use (Section~\ref{ssec:prelim}). In Section~\ref{ssec:SW2;basicprop} we establish basic metric properties including lower semicontinuity of $SW$ and precompactness of balls in $SW$ with respect to the narrow topology, from which completeness follows. We conclude the section by noting that, unlike the Wasserstein space, the sliced Wasserstein space is not a geodesic space.

\subsection{Preliminaries on the Radon transform}\label{ssec:prelim}
Here we provide a brief overview of the key properties of the Radon transform. We refer the readers to Appendix~\ref{app:radon} for precise statements and to the book by Helgason~\cite{Hel10} for a more thorough introduction. 

\smallskip

\emph{The dual Radon transform.} Given an integrable function $\rg:\Pd\rightarrow\R$ we define its dual Radon transform, which we write $R^\ast \rg$ or $\check{\rg}$, by
\begin{equation}\label{def:dualRadon}
    R^\ast \rg(x)=\check{\rg}(x)=\dashint_{\S^{d-1}} \rg(\theta, x \cdot\theta)\,d \theta.
\end{equation}
As
\[\dashint_{\S^{d-1}}\int_{\R} g(\theta,r)\,d\L^1(r)\,d\theta
=\dashint_{\S^{d-1}}\int_{\R^d} g(\theta,x\cdot\theta) \,d\L^d(x)\,d\theta = \int_{\R^d} R^\ast g(x)\,d\L^d(x),
\]
by Fubini's theorem $R^\ast\rg$ is well-defined for $\L^d$-a.e. $x\in\R^d$ whenever $\rg\in L^1(\Pd)$. Furthermore, the dual transform $R^\ast$ satisfies
\begin{equation}\label{eq:duality;radon}
    \dashint_{\S^{d-1}}\int_{\R} Rf(\theta,r)\rg(\theta,r)\,dr\,d\theta = \int_{\R^{d}} f(x)R^\ast \rg(x)dx
\end{equation}
whenever either $Rf \rg$ or $f R^\ast \rg$ are absolutely integrable; see~\cite[Lemma 5.1]{Hel10} for further details. In particular, the extension of the Radon transform to finite measures $\mu$ as the pushforward of $\mu$ under the map $x\mapsto x\cdot\theta$ is consistent with \eqref{eq:duality;radon} (see Remark~\ref{rmk:radon_on_measure}).
Consequently, we will often use the duality formula for bounded measures in the form
\begin{equation}\label{eq:duality;radon;measure}
    \int_{\R^d} R^\ast \rg\,d\mu = \dashint_{\S^{d-1}}\int_{\R} \rg\,d\hat\mu \;\;\text{ for } \mu\in\M_b(\R^d) \text{ and } \rg\in C_0(\Pd).
\end{equation}
\smallskip

\emph{Spaces related to the Radon transform.}
To add clarity, we denote the functions \vio defined on $\Pd$ \nc by a different set of symbols -- e.g. $\rf,\rg,\ru,\rv,\rj$. We denote by $\M(\Omega)$ the space of locally finite signed Borel measures on $\Omega$. 
We note that $\M(\Pd)$ can be identified with
\begin{equation}\label{def:MPd}
    \{\rj\in\M(\S^{d-1}\times\R): d\rj(-\theta,-r)=d\rj(\theta,r) \}.
\end{equation}
We write $\M_b(\Omega)$ for the space of bounded Borel measures on $\Omega$.
Any $\Omega=\R^d,\Pd,I\times\R^d,I\times\Pd$ is a Polish space, hence $\M(\Omega)$ can be equivalently understood as a space of signed Radon measures. 
Finally, we denote by $\M(\Omega;\R^d):=\M(\Omega)^d$ the space of vector valued Radon measures.

We will mostly treat $\theta\in\S^{d-1}$ as a parameter and $r\in\R$ as the variable, which is reflected in our notation.  For instance, for a function $\rg:\Pd\rightarrow\R$ we write $\rg^\theta(r)=\rg(\theta,r)$. Denoting by $\vol_{\S^{d-1}}$ 
the normalized volume measure on $\S^{d-1}$ satisfying $\int_{\S^{d-1}}d\vol_{\S^{d-1}}=1$, for each $\rj\in\M_b(\Pd)$ we write $\rj^\theta\in\M_b(\R)$ for its disintegration with respect to $\vol_{\S^{d-1}}$ -- i.e. \vio$\rj=\rj^\theta\,d\vol_{\S^{d-1}}(\theta)$; for precise statement of the disintegration theorem, see~\cite[III-70]{DelMey78} or~\cite[Theorem 5.3.1]{AGS}\nc. \grn We will always consider $\rj\in\M_b(\Pd)$ with its first marginal equal to $\vol_{\S^{d-1}}$\nc.

We denote by $\Sch(\Omega)$ with $\Omega=\R^d,\Pd$ the Schwartz-Bruhat space of smooth rapidly decreasing functions~\cite{Bruhat61,Osborne75}. We note that $\Sch(\R^d)$ is the usual Schwartz class, whereas $\Sch(\Pd)$ can be identified with the subspace of $\Sch(\S^{d-1}\times\R)$ of even functions, namely the set
\begin{equation}\label{def:Sch_radon}
    \{\rg\in\Sch(\S^{d-1}\times\R):\,\rg(-\theta,-r)=\rg(\theta,r)\}
\end{equation}
 We write $\Sch'(\Omega)$ with $\Omega=\R^d,\Pd$ to denote the space of continuous linear functionals on $\Sch(\Omega)$ --i.e. the space of tempered distributions on $\Omega$.

The $L^2$-Sobolev theory of Radon transforms will be crucial in understanding the differential structure of the sliced Wasserstein space. For this purpose, we use Sobolev spaces $H_t^s$ with attenuated ($t>0$) or amplified ($t<0$) low frequencies, introduced by Sharafutdinov~\cite{Sha21}. \grn For each $1\leq k\leq d$ let us denote by $\F_k$ the $k$-dimensional Fourier transform
\[\F_k f(\xi)=(2\pi)^{-k/2}\int_{\R^k} f(x)e^{-ix\cdot\xi}\,dx.\]
\nc For $s\in\R$ and $t>-\frac d2$, the Hilbert space $H_t^s(\R^d)$ is defined as the completion of $\Sch(\R^d)$ under the norm
\begin{equation}\label{def:Hts_norm}
    \|f\|_{H_t^s(\R^d)}^2=\int_{\R^d} |\xi|^{2t}(1+|\xi|^2)^{s-t} |\F_d f(\xi)|^2\,d\xi.
\end{equation}
Similarly we define the analogous space $H_{t}^s(\Pd)$ for $s\in\R$ and $t>-\frac 12$ in the Radon domain as the completion of Schwartz functions $\Sch(\Pd)$ on the Radon domain under the norm
\begin{equation}\label{def:Hts_norm;radon}
    \|\rg\|_{H_t^s(\Pd)}^2 = \frac{1}{2(2\pi)^{d-1}} \int_{\S^{d-1}} \int_{\R} |\zeta|^{2t}(1+\zeta^2)^{s-t}|\F_1\rg(\theta,\zeta)|^2\,d\zeta\,d\theta.
\end{equation}
\grn Here and in the sequel the one-dimensional Fourier transform $\F_1$ always applies to the scalar variable when applied to functions defined on $\Pd$. \nc
We only use the norms $H_t^s$ or $H_{-t}^{-s}$ for $0\leq t\leq s$. In the first case, we can view the $H_t^s$ norm as counting derivatives of order between $t$ and $s$, as
\[ |\xi|^{2t}+|\xi|^{2s} \lesssim_{t-s} |\xi|^{2t}(1+|\xi|^{2})^{t-s}\lesssim_{t-s} |\xi|^{2t}+|\xi|^{2s}.\]
Thus when $0=t\leq s$ we see $H_t^s$ coincides with the standard Sobolev space of order $s$. On the other hand, the space $H_{-t}^{-s}$ can be understood as the dual of $H_t^s$. Indeed, for any $f,g\in\Sch(\R^d)$
\begin{equation}\label{eq:Hts_dual}
      \left|\int_{\R^d}f(x)g(x)\,dx\right| \leq \|f\|_{H_t^s(\R^d)}\|g\|_{H_{-t}^{-s}(\R^d)};
\end{equation}
for  details, see~\cite[Theorem 5.3]{Sha21} and its proof. We provide further \grn information \nc on the relationship between the Radon transform and Sobolev spaces in  Appendix~\ref{app:radon}.

Outside of Section~\ref{sec:SW;basic} we will mostly be interested in the case $t=s$, where $\dot H^s(\Omega):=H_s^s(\Omega)$ with $\Omega=\R^d$ or $\Pd$ is equivalent to the more familiar homogeneous Sobolev space.

We only consider the space $H_t^s(\R^d)$ where $-\frac d2 <t< \frac d2$, and $H_t^s(\Pd)$ for $-\frac12<t<\frac12$ which ensures that the identity map continuously embeds $H_t^s(\R^d)$ to $\Sch'(\R^d)$ and the same holds for $H_t^s(\Pd)$ and $\Sch'(\Pd)$; see~\cite[Theorem 5.3]{Sha21}, which we have included in the appendix (Theorem~\ref{thm:Hts;tempered_dist}) for completeness. Thus, for $\Omega=\R^d,\Pd$, the spaces $H_t^s(\Omega)$ can be seen as a complete normed subspace of $\Sch'(\Omega)$. We stress that, while we use the norm $\|\cdot\|_{\dot H^{(d+1)/2}(\R^d)}$ in comparison to $SW$, generic elements of spaces $H_t^s(\R^d)$ with $|t|>\frac d2$ are not considered in this paper.

Furthermore, when $0<s<\frac d2$, $\dot H^s(\R^d)$ continuously embeds to $L^{\frac{2d}{d-2s}}(\R^d)$ by Gagliardo-Nirenberg-Sobolev inequality for fractional Sobolev spaces; see for instance~\cite[Theorem 11.31]{Leoni23}. Thus, we can consider $\dot H^s(\R^d)$ as a space of functions in this case.

We note that in some works the definition of homogeneous Sobolev spaces for $s<d/2$ differs from the one we use. Namely $\dot H^s(\Omega)$ is defined as the subset of $\Sch'(\Omega)$ for which the seminorm \eqref{def:Hts_norm} for $t=s$ is bounded; in this case, elements in $\dot H^s$ are uniquely defined in $\Sch'$ modulo polynomials; see for instance~\cite[Remark 3, Section 5.1]{Triebel10}.

Sharafutdinov showed~\cite[Theorem 2.1]{Sha21} that the Radon transform can be extended as a bijective isometry between $H_t^s(\R^d)$ and $H^{s+(d-1)/2}_{t+(d-1)/2}(\Pd)$ -- i.e. when $t>-\frac d2$
\begin{equation}\label{eq:Radon_isometry;Hts}
    \|f\|_{H_t^s(\R^d)}=\|Rf\|_{H^{s+(d-1)/2}_{t+(d-1)/2}(\Pd)}.
\end{equation}
The special case $t=s=0$ was observed by Reshetnyak, recorded in~\cite[Section 1.1.5]{GGV66} and also in~\cite[Chapter 1, Theorem 4.1]{Hel10}. \purp Whenever $t\in(-d/2,-d/2+1)$, the $H_{s+(d-1)/2}^{t+(d-1)/2}(\Pd)$-norm is stronger than the topology of $\Sch'(\Pd)$, thus the continuous extension of the Radon transform applied to any function $f\in H_t^s(\R^d)$ is unambiguously defined as an element of $\Sch'(\Pd)$ independently of $t\in(-d/2,-d/2+1)$ and $s\in\R$. Therefore \nc in the remainder of this paper we refer to this extension simply as the Radon transform. 

\blue
In Sections~\ref{sec:comparison} and~\ref{sec:lsw;stat} we will make use of the weighted homogeneous Sobolev norm of order $-1$ in dimension 1. Given $\sigma,\mu,\nu\in\P_2(\R)$, the $\dot H^{-1}(\sigma)$-norm of $\mu-\nu$ is defined by
\begin{equation}\label{def:weightedH1neg;1D}
    \|\mu-\nu\|_{\dot H^{-1}(\sigma)}:=\sup\left\{\int_{\R} \varphi\,d(\mu-\nu):\;\varphi\in\Sch(\R) \text{ and }\|\varphi\|_{\dot H^1(\sigma)}\leq 1 \right\}.
\end{equation}
\nc
\smallskip

\emph{Operators related to the Radon transform.}
Calculus using the Radon transform often involves $\Lambda_d$, which is defined via
\begin{equation}\label{eq:Lambdad-laplace}
    \Lambda_d=(-\partial_{r}^2)^{\frac{d-1}{2}}.
\end{equation}
When $d$ is even, the fractional power of the 1-dimensional Laplace operator is defined using the Hilbert transform; for precise definitions, see Definition~\ref{def:Lambda} in the appendix. Observe that $\Lambda_d$ is well-defined as an operator from $\Sch(\Pd)$ to itself, and can be extended as a bounded operator from $H^r_t(\Pd)$ to $H^{r-(d-1)}_{t-(d-1)}(\Pd)$ for $t>d-3/2$; see Remark~\ref{rmk:Lambda}.

The operators $\Lambda_d$ can be understood by their interaction with the Fourier transform, namely
\begin{equation}\label{eq:Lambdas}
    \begin{split}
    (\F_1\Lambda_d \rg)(\theta,\zeta)&=|\zeta|^{d-1}\F_1 \rg(\theta,\zeta).
    \end{split}
\end{equation}
We also use fractional powers of the Laplace operator in $d$-dimensions, which can be defined via the Fourier transform by
\[(\F_d(-\Delta)^{s} f)(\xi)=|\xi|^{2s}\F_d f(\xi).\]
Again, observe that $(-\Delta)^s$ is well-defined as an operator from $\Sch(\R^d)$ to itself, and can be extended as a bounded operator from $H^r_t(\R^d)$ to $H^{r-2s}_{t-2s}(\R^d)$ when $t-2s>-\frac d2$.
Rigorous definition of $(-\Delta)^s$ for fractional powers $s$ without relying on the Fourier transform can be found in~\cite[Chapter 7.6]{Hel10}, which we have included in Proposition~\ref{prop:riesz_trnsfrm;properties} in the appendix for completeness. 
The inversion formulae are expressed using these operators: Setting $c_d=(4\pi)^{(d-1)/2}\Gamma(d/2)/\Gamma(1/2)$, for each $f\in\Sch(\R^d)$ and $\rg\in\Sch(\Pd)$ we have
\[
c_d f=R^\ast \Lambda_d R f=(-\Delta)^{(d-1)/2} R^\ast Rf \quad \text{ and } \quad 
c_d \rg(\theta,r)= R R^\ast(\Lambda_d \rg).
\] 
See Proposition~\ref{prop:radon_inversion} for further details. 

Whenever $t\in(-\frac d2,-\frac d2+1)$, we have $H_t^s(\R^d)\subset\Sch'(\R^d)$ and $H_{t+(d-1)/2}^{s+(d-1)/2}(\Pd)\subset\Sch'(\Pd)$. In this case, straightforward calculations using the inversion formula and the Fourier transform imply
\begin{equation}\label{eq:duality;radon;dist}
    J(c_d\varphi) = RJ(\Lambda_d R\varphi) \text{ for } J\in H_t^s(\R^d) \text{ and }\varphi\in\Sch(\R^d).
\end{equation}
\purp For each of the domains  $\Omega=\R^d,\Pd,\R$, we will write
\begin{equation}\label{def:bracket_Omega}
    \langle J,\varphi\rangle_{\Omega}:=J(\varphi) \text{ for } J\in\Sch'(\Omega) \text{ and } \varphi\in\Sch(\Omega).
\end{equation}
\nc
\smallskip

\subsection{Basic properties of sliced Wasserstein metric}\label{ssec:SW2;basicprop}
In this section we establish some basic properties of the SW distance and the SW space.

Let $\P_2(\Pd)$ be the \grn set of Borel probability measures \nc in the Radon domain with bounded second moment
\begin{equation}\label{def:PpPd}
\begin{split}
    \P_2(\Pd):=&\left\{\hat\mu\in\P_2(\S^{d-1}\times\R): \pi^1_\# \hat\mu=\vol_{\S^{d-1}},\;\;\hat\mu^\theta\in\P_2(\R), \phantom{\int_\R} \right.\\
    &\left.\;\; d\hat\mu^{-\theta}(-r)=d\hat\mu^{\theta}(r),\quad\dashint_{\S^{d-1}} \int_{\R} r^2\,d\hat\mu^\theta(r)\,d\theta<\infty\right\}
\end{split}
\end{equation}
where $\pi^1$ is the projection in the first variable. In other words, $(\hat\mu^\theta)_{\theta\in\S^{d-1}}$ is a family of measures in $\P_2(\R)$ parametrized by $\theta\in\S^{d-1}$ additionally satisfying the evenness condition.
\vio Observe that for each $\theta\in\S^{d-1}$, we can choose an orthonormal frame $\{\theta^1,\cdots,\theta^d\}$ with $\theta^1=\theta$ and $\theta^i\in\S^{d-1}$ for $i=1,\cdots,d$. As $|x|^2=\sum_{i=1}^d |x\cdot\theta^i|^2$, we have
\[\int_{\S^{d-1}}|x|^2\,d\vol_{\S^{d-1}}(\theta)=\sum_{i=1}^d\int_{\S^{d-1}}|x\cdot\theta^i|^2\,d\vol_{\S^{d-1}}(\theta^i)=d\int_{\S^{d-1}} |x\cdot\theta|^2\,d\vol_{\S^{d-1}}(\theta).\]
\nc Thus
\[\dashint_{\S^{d-1}} \int_{\R} r^2 d\hat\mu^\theta(r)\,d\theta = \dashint_{\S^{d-1}} \int_{\R^d} |x\cdot\theta|^2\,d\mu(x)\,d\theta = \frac1d \int_{\R^d} |x|^2\,d\mu(x).\]
Equivalently,
\begin{equation}\label{eq:dSW=W;delta}
    SW^2(\mu,\delta_0)=\frac1d W^2(\mu,\delta_0) \text{ for any } \mu\in\P_2(\R^d).
\end{equation}
Thus the finite second moment condition in the Euclidean and Radon domain coincide. Hence, $\mu\in\P_2(\R^d)$ if and only if $\hat\mu\in\P_2(\Pd)$.        

Given $\mu,\nu\in\P_2(\R^d)$, let $\Gamma_o(\mu,\nu)$ be the set of optimal transport plans for the quadratic cost:
\begin{equation}\label{def:Gammao}
    \Gamma_o(\mu,\nu)=\left\{\gamma\in\Gamma(\mu,\nu):\; \grn \int_{\R^d\times\R^d} |x-y|^2\,d\gamma(x,y)\nc =W^2(\mu,\nu)\right\},
\end{equation}
where $\Gamma(\mu,\nu)$ is as defined in \eqref{def:Wp}. On the other hand, given $\hat\mu,\hat\nu\in\P_2(\Pd)$, write
\begin{equation}\label{def:hatGamma}
    \widehat\Gamma(\hat\mu,\hat\nu)=\{\hat\gamma\in\P(\S^{d-1}\times\R\times\R):\;\pi^1_\#\hat\gamma=\vol_{\S^{d-1}},\;\hat\gamma^\theta\in\Gamma(\hat\mu^\theta,\hat\nu^\theta) \text{ for a.e. }\theta\in\S^{d-1}\}
\end{equation}
where $(\hat\gamma^\theta)_{\theta\in\S^{d-1}}$ is the disintegration of $\hat\gamma$ with respect to $\vol_{\S^{d-1}}$.
Then we define the set $\widehat\Gamma_o$ of slice-wise optimal transport plans by
\begin{equation}\label{def:hatGammao}
    \widehat\Gamma_o(\hat\mu,\hat\nu)=\{\hat\gamma\in\widehat\Gamma(\hat\mu,\hat\nu):\;\hat\gamma^\theta\in\Gamma_o(\hat\mu^\theta,\hat\nu^\theta) \text{ for a.e. }\theta\in\S^{d-1}\}.    
\end{equation}
As noticed by Bonnotte~\cite{Bon13}, given $\gamma\in\Gamma(\mu,\nu)$, $(\pi^\theta\times\pi^\theta)_\# \gamma\in\Gamma(\pi^\theta_\#\mu,\pi^\theta_\#\nu)$, and thus
\[SW(\mu,\nu)\leq \frac{1}{\sqrt{d}} W(\mu,\nu) \text{ for all } \mu,\nu\in\P_2(\R^d).\]
                       
\grn We denote the optimal transport map between $\mu,\nu\in\P_2(\R^d)$, if it exists, by $T_{\mu}^\nu$ -- i.e. $(\id\times T_\mu^\nu)\in\Gamma_o(\mu,\nu)$, where $\id(x)=x$\nc. Similarly, we denote by $\widehat T_{\hat\mu}^{\hat\nu}$ the family of optimal transport maps in the Radon domain -- i.e. $\widehat T_{\hat\mu}^{\hat\nu}:\S^{d-1}\times\R\rightarrow \R$ such that
\[\|\widehat T_{\hat\mu}^{\hat\nu}-\widehat\id\|_{L^2(\hat\mu)}=SW(\mu,\nu) \text{ where } \widehat\id(\theta,r)=r.\]
\smallskip

\grn Recall that the sequence $(\mu_n)_{n\in\N}$ in $\P(\R^d)$ converges narrowly to $\mu\in\P(\R^d)$ if for each continuously bounded function $\varphi\in C_b(\R^d)$
\begin{equation}\label{def:narrow_conv}
    \lim_{n\rightarrow\infty}\int_{\R^{d}} \varphi(x)\,d\mu_n(x) = \int_{\R^d} \varphi(x)\,d\mu(x).
\end{equation}
We begin by establishing lower semicontinuity of $SW$ with respect to the narrow convergence.
\nc
\begin{proposition}[$SW$ is lower semicontinuous with respect to the narrow topology]\label{prop:SW2_lsc}
    The map $(\mu,\nu)\mapsto SW(\mu,\nu)$ from $\P_2(\R^d)\times\P_2(\R^d)$ to $[0,+\infty)$ is lower semicontinuous with respect to the narrow topology.
\end{proposition}

\begin{proof}
Note that the analogous statement for $W_p$ for $p\in[1,+\infty)$ is a classical result; see~\cite[Remark 6.12]{Vil09} for instance (where they refer to the narrow convergence as weak convergence). Clearly $\sigma^k\rightharpoonup \sigma$ implies $\pi^\theta_\# \sigma^k\rightharpoonup\pi^\theta_\#\sigma$. Thus, if $(\mu^k,\nu^k)\rightharpoonup(\mu,\nu)$ narrowly, then by Fatou's lemma
\begin{align*}
    \liminf_k SW^2(\mu^k,\nu^k) &= \liminf_k \dashint_{\S^{d-1}} W^2(\pi^\theta_\#\mu^k,\pi^\theta_\#\nu^k)\,d\theta  \\
    &\geq \dashint_{\S^{d-1}}\liminf_k W^2(\pi^\theta_\#\mu^k,\pi^\theta_\#\nu^k) \,d\theta
    \geq \dashint_{\S^{d-1}} W^2(\pi^\theta_\#\mu,\pi^\theta_\#\nu)\,d\theta = SW^2(\mu,\nu).
    \end{align*}
    Thus we deduce that $SW$ is lower semicontinuous with respect to the narrow convergence.
\end{proof}

\begin{remark}[Lack of compactness in $SW$]\label{rmk:SWball=notswcpct}
We note here that the closed unit ball $\overline B^{SW}(\nu,1)$ in $SW$ is not compact in the $SW$ topology. The argument is analogous to one that shows that $\overline B^{W}(\nu,1)$ is not compact with respect to the topology of the Wasserstein metric. The argument is as follows: Consider
\[\nu_k=(1-\eps_k)\nu+\eps_k\delta_{x_k}\in \overline B_{W}(\nu,1)\subset \overline B^{SW}(\nu,1)\]
where $|x_k|\nearrow \infty$ and choose $\eps_k\searrow 0$ such that $\eps_k|x_k-x_0|^2=1$. Quick calculations show that while $\nu_k\rightharpoonup \nu$, the second moments do not converge, as
\[\liminf_{k\to \infty} \int_{\R^d} |x|^2\,d\nu_k = \int_{\R^d} |x|^2\,\vio d\nu\nc+ 1.\]
Thus $W(\nu_k,\nu)\not\rightarrow 0$. As $W$ and $SW$ induce the same topology, we deduce that $\overline B^{SW}(\nu,1)$ is indeed not compact with respect to the $SW$ topology.
\end{remark}

On the other hand we note that balls in sliced Wasserstein space are compact with respect to the narrow convergence of measures:
\begin{proposition}[Narrow compactness of the sliced Wasserstein unit ball]\label{prop:SWball=wcpct}
    Let $\nu\in\P_2(\R^d)$ be fixed. Then the closed unit ball
    \[\overline B^{SW}(\nu,1):=\left\{\sigma\in\P_2(\R^d):\;SW(\sigma,\nu)\leq 1\right\}\]
    is compact with respect to the topology of narrow convergence. 
\end{proposition}

\begin{proof}
\grn Recall from \eqref{eq:dSW=W;delta} that
\begin{align*}
    \int_{\R^d}|x|^2\,d\sigma(x) = W^2(\delta_0,\sigma)
    = d SW^2(\delta_0,\sigma).
\end{align*}
\nc Thus, for all $\sigma\in\overline B:=\overline B^{SW}(\nu,1)$, we have
\[\int_{\R^d}|x|^2\,d\sigma(x) \leq d SW^2(\delta_0,\sigma)\lesssim_d SW^2(\delta_0,\nu)+SW^2(\nu,\sigma)\leq SW^2(\delta_0,\nu)+1.\]
Hence the second moment of probability measures in $\overline B$ is uniformly bounded, hence $\overline B$ is tight, as
\begin{align*}
    \sigma(\R^d\setminus B(0,R))&=\sigma(\{x\in\R^d: |x|\geq R\}) \leq \frac{1}{R^2}\int_{\R^d\setminus B(0,R)} |x|^2\,d\sigma(x) \\
    &\leq \frac{1}{R^2}\int_{\R^d} |x|^2\,d\sigma(x) \lesssim_d \frac{1}{R^2} (SW^2(\delta_0,\nu)+1).
\end{align*}
        
Thus by Prokhorov's theorem, for any sequence $(\sigma^k)_k$ in $\overline B$ we can find a subsequence narrowly converging to $\sigma^0\in\P(\R^d)$. Fix such a subsequence without relabling. Moreover $\sigma^0\in\overline B$, as 
\[SW(\nu,\sigma^0)\leq \liminf_{k} SW(\nu,\sigma^k)\leq 1\]
by lower semicontinuity of $SW$ (Proposition~\ref{prop:SW2_lsc}).
\end{proof}

We can deduce completeness from weak compactness, lower semicontinuity, and the topological equivalence. While the authors believe this is known, we record the proof here as we could not locate the statement of completeness in the literature. 

\begin{proposition}[Completeness]\label{prop:SW2_complete}
    $(\P_2(\R^d),SW)$ is a complete metric space.
\end{proposition}
\begin{proof}
    Suppose $(\mu^k)_k$ is a Cauchy sequence with respect to $SW$. Then we can find a closed unit ball $\overline B$ in $SW$ that contains all $\mu^k$ for sufficiently large $k$, which is relatively compact by Proposition~\ref{prop:SWball=wcpct} hence has a subsequential narrow limit $\mu^0\in\overline B\subset\P_2(\R^d)$. Fix such a subsequence without relabeling. Then, by the lower semicontinuity established in Proposition~\ref{prop:SW2_lsc},
    \[SW(\mu^k,\mu^0)\leq \liminf_l SW(\mu^k,\mu^l) \xrightarrow[]{k\rightarrow\infty} 0.\]
    By the triangle inequality we deduce $\mu^k\rightarrow\mu^0$ for the original sequence in $SW$.
\end{proof}

We conclude this section by noting that the sliced Wasserstein space is not a geodesic space. Indeed, let $\mu_0,\mu_1\in\P_2(\R^d)$, and suppose $(\mu_t)_{t}:[0,1]\rightarrow\P_2(\R^d)$ is a \vio constant-speed $SW$-geodesic from $\mu_0$ to $\mu_1$ -- i.e. for any $0\leq s<t\leq 1$, $SW(\mu_t,\mu_s)=(t-s)SW(\mu_0,\mu_1)$. Then for any $N\in\N$ and $0=t_0<t_1<\cdots<t_N=1$,
\begin{align*}
    \norm{\sum_{i=0}^{N-1}W(\hat\mu_{t_i}^\cdot,\hat\mu_{t_{i+1}}^\cdot)}_{L^2(\vol_{\S^{d-1}})}\leq \sum_{i=0}^{N-1} \|W(\hat\mu_{t_i}^\cdot,\hat\mu_{t_{i+1}}^\cdot)\|_{L^2(\vol_{\S^{d-1}})}=\sum_{i=0}^{N-1} SW(\mu_{t_i},\mu_{t_{i+1}})=SW(\mu_0,\mu_1).
\end{align*}
Using this and the triangle inequality $W(\hat\mu_0^\theta,\hat\mu_1^\theta)\leq\sum_{i=0}^{N-1}W(\hat\mu_{t_i}^\theta,\hat\mu_{t_{i+1}}^\theta)$,
\begin{align*}
0\leq\dashint_{\S^{d-1}}\left(\sum_{i=0}^{N-1}W(\hat\mu_{t_i}^\theta,\hat\mu_{t_{i+1}}^\theta)\right)^2-W^2(\hat\mu_0^\theta,\hat\mu_1^\theta)\,d\theta
\leq SW^2(\mu_0,\mu_1)-SW^2(\mu_0,\mu_1)=0.
\end{align*}
As the integrand above is nonnegative, it is zero for a.e. $\theta\in\S^{d-1}$. By for instance considering all rational sequences $0=t_0<\cdots<t_N=1$ and arguing by density, we may deduce that for a.e. $\theta\in\S^{d-1}$ the curve $(\hat\mu^\theta_t)_{t\in[0,1]}$ must be a $2$-Wasserstein geodesic between $\hat\mu_0^\theta$ and $\hat\mu_1^\theta$, hence characterized by the displacement interpolation based on the 1D transport plan. \nc
    
Thus the problem of identifying the geodesic comes down to invertibility of displacement interpolant $(\hat\mu_t)_{t\in[0,1]}$. In principle, sufficient regularity of $\hat\mu_t$ guarantees the existence of a function $\mu_t$ such that $R\mu_t=\hat\mu_t$ for $t\in[0,1]$. However, we additionally require $R^{-1}\hat\mu_t \geq 0$. 
    
While the Radon transform preserves nonnegativity, the inverse Radon transform does not. For example, consider
\[g_\eps= a\one_{B(0,1)}-b\one_{B(0,\eps)} \text{ with } 0<a<b \text{ and } \eps\ll 1.\]
For sufficiently small $\eps>0$ we see $\hat g_\eps \geq 0$, whereas $g_\eps<0$ near the origin. Mollifying $g_\eps$ within radius $\tilde\eps\ll \eps$, we see that the additional regularity does not resolve the issue. In general, it is difficult to determine when the Radon inversion is nonnegative, as the inversion formula involves high order derivatives and $R^\ast$. Indeed, in many cases the geodesic $(\mu_t)_{t}$ cannot exist, as we can see in the following example.

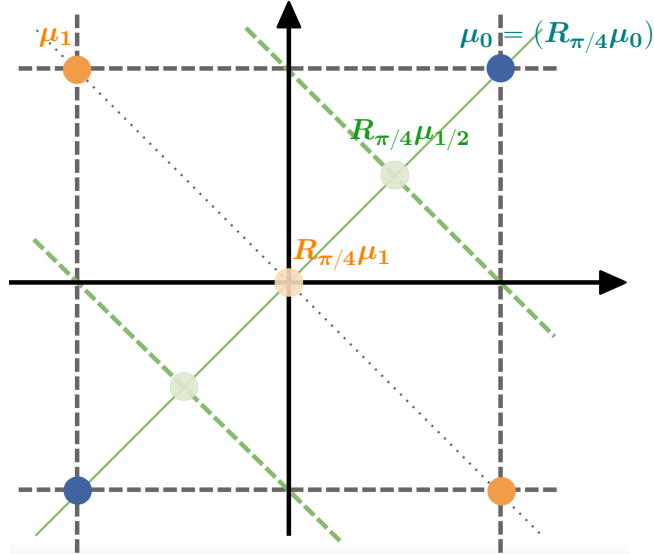
\begin{figure}[htbp]
\centering
\begin{tikzpicture}
	    \node at (0,0){\includegraphics[width=0.5\linewidth]{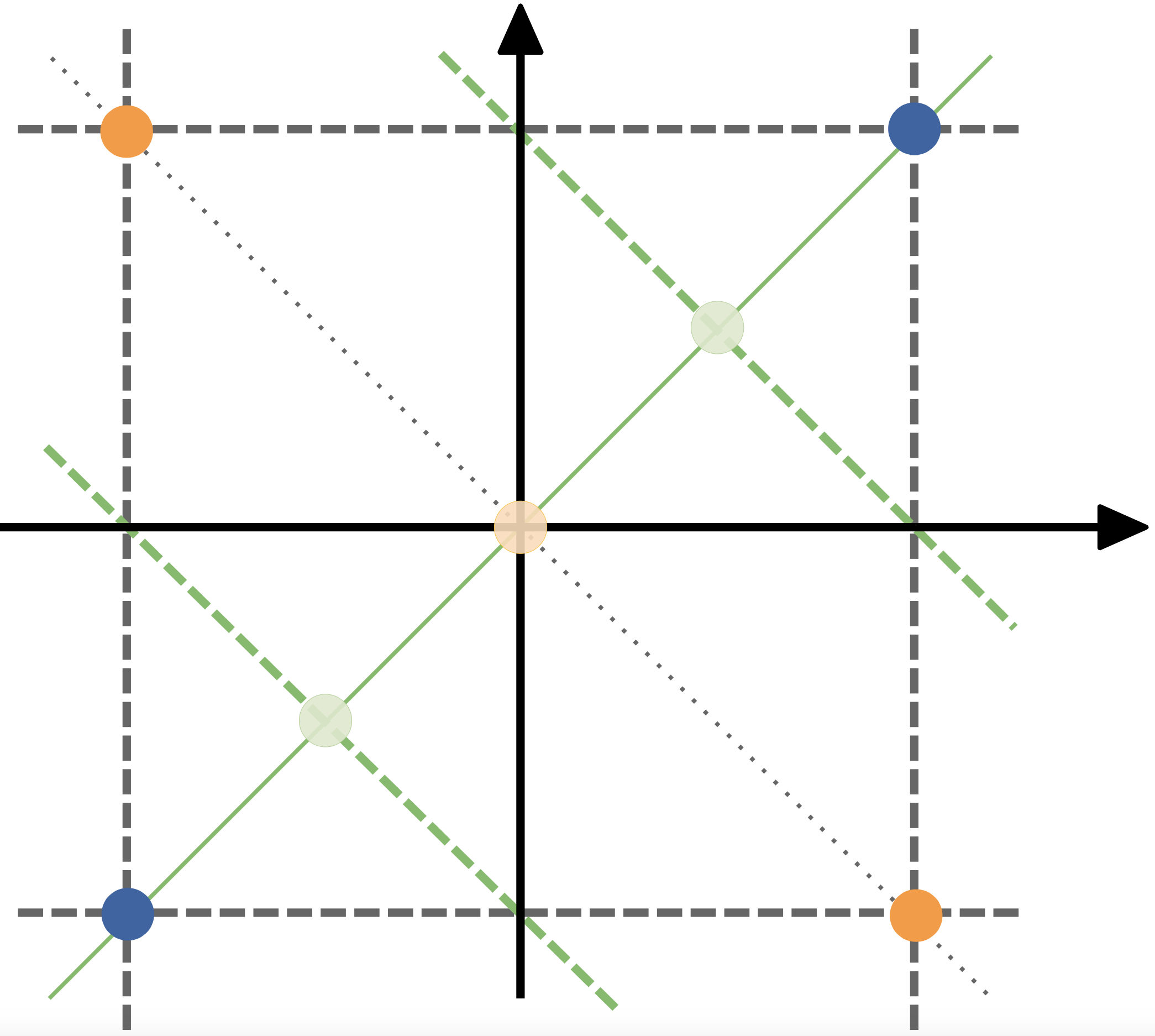}};
	    \node[color=teal] at (3.16,3.2){$\boldsymbol{\mu_0=(R_{\pi/4}\mu_0)}$};
	    \node[color=orange]  at (-3.5,3.2){$\boldsymbol{\mu_1}$};
	    \node[color=orange] at (0.3,0.3){$\boldsymbol{R_{\pi/4}\mu_1}$};
	    \node[color=darkgreen] at (1.2,1.9){$\boldsymbol{R_{\pi/4}\mu_{1/2}}$};
\end{tikzpicture}
\caption{Illustration of the counterexample provided in Example~\ref{ex:SW2_notgeo}.
Measures $\mu_0$ and $\mu_1$ each consist of a pairs of delta masses at the opposite vertices of the unit square, indicated by blue and orange discs respectively. 
The transparent disc in the center is $R_{\pi/4}\mu_1$ shown on the line at angle $\theta=\pi/4$. Green discs on the dashed anti-diagonal lines mark represent the displacement interpolation at $t=1/2$ between $R_{\pi/4}\mu_0$ and $R_{\pi/4}\mu_1$.}
\label{fig:SW_notgeo}
\end{figure} 

\begin{example}[$(\P_2(\R^d),SW)$ is not a geodesic space]\label{ex:SW2_notgeo}
Consider $\mu_0=\frac12(\delta_{(-1-1)}+\delta_{(1,1)})$ and \grn $\mu_1=\frac12(\delta_{(-1,1)}+\delta_{(1,-1)})$ \nc. Suppose $(\mu_t)_{t\in[0,1]}$ is the constant speed geodesic from $\mu_0$ to $\mu_1$. As the quadratic cost is strictly convex, for each $\theta\in\S^1$, $R_\theta\mu_t$ should be the displacement interpolation between $R_\theta\mu_0$ and $R_\theta\mu_1$. \vio Considering $t=\frac12$ in particular, $R_\theta\mu_{1/2}$ should satisfy \nc
\[
    R_\theta\mu_{1/2}=
    \begin{cases}
    \frac12(\delta_{-1}+\delta_{1}) & \text{ for } \theta=0,\\
    \frac12\left(\delta_{-\frac{1}{\sqrt{2}}}+\delta_{\frac{1}{\sqrt{2}}}\right) & \text{ for } \theta=\frac\pi4, \text{ and }\\
    \frac12(\delta_{-1}+\delta_{1}) & \text{ for } \theta=\frac\pi2.
    \end{cases}
\]
\vio However, this is impossible: \nc The first and the third line imply that $\mu_{1/2}$ must be a convex combination of $\delta_{(-1,-1)}, \delta_{(-1,1)},\delta_{(1,-1)}$, and $\delta_{(1,1)}$, which contradicts the second requirement. \grn By continuity of $\theta\mapsto R_\theta\mu_{1/2}$, similar properties hold for $R_\theta\mu_{1/2}$ close to $\theta=0,\pi/4,\pi/2$, which shows that the geodesic cannot exist. \nc

Note that regularity is not the only issue; by convolving $\mu,\sigma$ with a smooth kernel with a sufficiently small radius, we can argue similarly that there cannot be a geodesic. 
\end{example}
In fact, we will see in Corollary~\ref{cor:sw2-neq-lsw} that $SW$ is not even a length metric on $\P_2(\R^d)$ -- i.e. in general it cannot be approximated by the sliced Wasserstein length of absolutely continuous curves. This motivates us to investigate the length (and geodesic) metric induced by $SW$ on $\P_2(\R^d)$.

\section{Curves in the sliced Wasserstein space and the tangential structure}\label{sec:three}
In this section we study absolutely continuous curves in the SW space. In Section~\ref{ssec:SW2_ac-curves} we investigate the sliced Wasserstein metric derivative and prove the main result of this section, Theorem~\ref{thm:SW2_ac-curves}, characterizing absolutely continuous curves by corresponding distributional solutions of the continuity equation in the flux form
\[\partial_t\mu_t+\nabla\cdot J_t=0.\]
    
Section~\ref{ssec:sw-tanspace} characterizes the tangent space $\tanspace_\mu(\P_2(\R^d),SW)$. The main result regarding the tangent space, Proposition~\ref{prop:sw-ac;Jt_unique}, states that if $(\mu_t,J_t)_{t\in I}$ is a solution of the continuity equation corresponding to an absolutely continuous curve, $\|d\widehat J_t/d\hat\mu_t\|_{L^2(\hat\mu_t)}$ attains $|\mu'|_{SW}(t)$ if and only if $J_t\in\tanspace_{\mu_t}(\P_2(\R^d),SW)$, up to a $\L^1\restr_I$-null set.

While these results are  reminiscent of the analogous statements for the Wasserstein space~\cite[Proposition 8.4.5]{AGS}, we note here a few differences. Firstly, from Theorem~\ref{thm:SW2_ac-curves} we know that a generic absolutely continuous curve $(\mu_t)_{t\in I}$ admits the representation $\partial_t\mu_t+\nabla\cdot J_t =0$ for some $J_t \in \mathscr{H}(\R^d;\R^d)$, where \grn
\begin{equation}\label{def:Hscr}
\begin{split}
    \mathscr{H}(\R^d;\R^d)&:=R^{-1}\M_b(\Pd;\R^d).
\end{split}
\end{equation}
We will see in Lemma~\ref{lem:Hts-leq-Bsw} that the above space is well-defined as a subspace of $\Sch'(\R^d;\R^d)$, and that 
\[\mathscr{H}(\R^d;\R^d)\subset H_{-(d-1)/2}^{-d/2-}(\R^d;\R^d):=\bigcap_{\eps>0} H_{-(d-1)/2}^{-d/2-\eps}(\R^d;\R^d)\subset\Sch'(\R^d;\R^d).\]
\nc
In general $J_t$ need not be measures, in which case it does not even make sense to consider the Radon-Nikodym derivative $\frac{dJ_t}{d\mu_t}$. Thus the distributional fluxes $J$ will be the main object on which the tangential structure is based. \grn Moreover, not all fluxes in $\tanspace_\mu(\P_2(\R^d),SW)$ preserve the nonnegativity of $\mu$. Therefore the tangent vectors attainable by curves in the $SW$ space forms a (convex) cone, whereas the tangent space of the 2-Wasserstein space is a vector space. Thus, we can formally consider the $SW$ space as a manifold with corners; see Remark~\ref{rmk:Vmu} for further details. \nc
    
\subsection{Absolutely continuous curves in the sliced Wasserstein space}\label{ssec:SW2_ac-curves}
Let $(X,m)$ be a complete metric space. We say a curve $v:(a,b)\rightarrow X$ belongs to $AC((a,b); X,m)$ if there exists $m\in L^1(a,b)$ such that
\begin{equation}\label{def:ACp_curve}
    m(v(s),v(t))\leq \int_s^t m(r)\,dr \quad \forall a<s\leq t<b.
\end{equation}
Furthermore, for any $v\in AC((a,b); X,d)$, the metric derivative
\begin{equation}\label{def:met_der}
    |v'|_m(t):=\lim_{s\rightarrow t}\frac{m(v(s),v(t))}{|s-t|}
\end{equation}
exists for $\L^1$-a.e. $t\in(a,b)$, and $|v'|_m\in L^1(a,b)$. We will often write $I$ to denote an interval, which is assumed to be open but not necessarily bounded, unless otherwise stated. When there is no room for confusion regarding the interval, we simply write $(v_t)_{t\in I}\in AC(X,m)=AC(I;X,m)$
    
Prior to studying the length associated to $SW$, let us first investigate the metric derivative $|\mu'|_{SW}$. As $SW(\mu,\nu)\leq W(\mu,\nu)$ for any $\mu,\nu\in\P_2(\R^d)$, it immediately follows that
\[|\mu'|_{SW}(t)\leq |\mu'|_{W}(t)\]
for all $t$ at which both metric derivatives are well-defined.
    
Recall~\cite[Theorem 8.3.1]{AGS} that each absolutely continuous curve $(\mu_t)_{t\in I}$ in the Wasserstein space has a corresponding distributional solution $(\mu_t,v_t \mu_t)_{t\in I}$ of the continuity equation such that
\[|\mu'|_{W}(t)=\|v_t\|_{L^2(\mu_t)} \text{ a.e. } t\in I.\]
We want to establish an analogous result for the sliced Wasserstein space. Suppose $|\hat\mu'(\theta,\cdot)|_{W}(t)$ is well-defined at all $t\in I$. Then
\begin{align*}
    \frac{SW^2(\mu_t,\mu_{t+h})}{h^2}=\dashint_{\S^{d-1}}\frac{W^2(\hat\mu_t^\theta,\hat\mu_{t+h}^\theta)}{h^2}\,d\theta\xrightarrow[]{h\rightarrow 0} \dashint_{\S^{d-1}} |\hat\mu'(\theta,\cdot)|_{W}^2\,d\theta.
\end{align*}
Assume that $(\mu_t,v_t)_{t\in I}$ satisfy the continuity equation
\[\partial_t \mu_t + \grn\nabla\cdot\nc (v_t\mu_t)=0.\]
\grn From direct calculations one can readily verify (see Proposition~\ref{prop:radon_intertwine} for the proof) 
\begin{equation}\label{eq:Radon-derivative}
    R(\partial_i f)(\theta,r)=\theta_i\partial_r Rf(\theta,r) \text{ for } f\in\Sch(\R^d).
\end{equation}
Formally applying the Radon transform to the continuity equation and using \eqref{eq:Radon-derivative}, \nc we obtain
\[\partial_t\hat\mu_t + \sum_{i=1}^d R [\partial_i(v_t^i\mu_t)] = \partial_t \hat\mu_t + \sum_{i=1}^d \theta_i\partial_r R(v_t^i \mu_t) =0.\]    
Rewriting in the velocity formulation,
\[\partial_t \hat\mu_t + \partial_r \left(\left(\theta\cdot\frac{d\widehat{v_t\mu_t}^\theta}{d\hat\mu_t^\theta}\right)\hat\mu_t^\theta\right)=0.\]
Thus, by applying~\cite[Theorem 8.3.1]{AGS} for each $\theta\in\S^{d-1}$, we deduce that formally
\begin{equation}\label{eq:sw_metder;ubd}
    |\mu'|_{SW}^2(t) = \dashint_{\S^{d-1}} |\hat\mu'(\theta,\cdot)|_{W}^2(t)\,d\theta \leq \dashint_{\S^{d-1}} \int_{\R} \left|\theta\cdot\frac{d\widehat{v_t\mu_t}^\theta}{d\hat\mu_t^\theta}\right|^2\,d\hat\mu^\theta \,d\theta
    =\norm{\theta\left(\theta\cdot\frac{d\widehat {v_t\mu_t}}{d\hat\mu_t}\right)}_{L^2(\hat\mu_t;\R^d)}^2.
\end{equation}
Observe that $\theta(\theta\cdot d\widehat {v_t\mu_t}/d\hat\mu_t)$ is even in $\S^{d-1}\times\R$, hence is a function on $\Pd$. For simplicity we will often write $\norm{\theta\cdot d\widehat {v_t\mu_t}/d\hat\mu_t}_{L^2(\hat\mu_t)}$ instead.
\vio Furthermore, note that the existence of the velocity that saturates the inequality is nontrivial\nc; for a.e. $\theta\in\S^{d-1}$ the projection of the velocity must saturate the corresponding 1D inequality.

Before we begin investigating the metric derivative in detail, let us first consider examples that compare $|\mu'|_{SW}$ and $|\mu'|_{W}$; Example~\ref{ex:metder;particle} demonstrates that they coincide for paths of discrete measures, whereas Example~\ref{ex:2lines} shows that the ratio $\frac{|\mu'|_{SW}}{|\mu'|_{W}}$ is in general not bounded from below.
    
\begin{example}[Discrete measures]\label{ex:metder;particle}
Let $\mu_t:=\frac1n\sum_{i=1}^n \delta_{x_i(t)}$ where $x_i:I\rightarrow\R^d$ is continuously differentiable for each $i=1,\cdots,n$. Then
\[|\mu'|_{SW}(t) = \frac{1}{\sqrt{d}}|\mu'|_{W}(t) \text{ for all } t\in I.\]
Indeed,
\[|\mu'|^2_{SW}(t) = \dashint_{\S^{d-1}} |\hat\mu'(\theta,\cdot)|^2_{W}(t)\,d\theta = \frac1n\sum_{i=1}^n \dashint_{\S^{d-1}} |\theta\cdot x_i'(t)|^2\,d\theta = \frac{1}{dn}\sum_{i=1}^n |x_i'(t)|^2 = \frac1d |\mu'|_{W}^2(t).\]
\end{example}

\begin{example}[Two sliding lines]\label{ex:2lines}
Consider two parallel line segments close to each other moving in opposite directions. A significant portion of the shearing velocity is cancelled out after projection, causing a significant gap between the metric derivatives with respect to $W$ and $SW$.

More precisely, let $\delta>0$ and let $\nu\in\P(\R^2)$ uniformly distributed on a segment with endpoints $(\pm 1,\delta)$. Similarly, $\sigma$ is the analogous measure slightly below, with endpoints $(\pm 1,-\delta)$. Then define $\nu_t=\nu(\cdot-(t,0))$ and $\sigma_t=\sigma(\cdot+(t,0))$ --i.e. $\nu$ is translated to the right and $\sigma$ to the left.
    
Defining $\mu_t=\frac{1}{2}(\nu_t+\sigma_t)$, one can check by direct calculations that
\[|\mu'|_{SW}(t)\lesssim t+\delta \ll 1= |\mu'|_W(t) \text{ for } t\geq 0.\]
\end{example}

\smallskip
We begin the rigorous study of the metric derivative $|\mu'|_{SW}$ by considering Benamou-Brenier functional for the sliced Wasserstein distance. 
Consider the 1D Benamou-Brenier functional for $\R^d$ valued flux, namely $\B_2:\M(\R)\times\M(\R;\R^d)$ defined by \vio
\begin{equation}\label{def:Bp;1d}
\begin{split}
    \B_{2}(\mu,E) :=&\sup\left\{\int_{\R} a(x)\,d\mu(x) + \int_{\R} b(x)\cdot dE(x):\,a,b\in C_b(\R;K_2) \right\}, \\
    &\text{ where }K_2:=\{(a,b)\in\R\times\R^d:\,a+\frac12|b|^2\leq 0\}.
\end{split}
\end{equation}
\nc $\B_2$ enjoys several desirable properties such as joint convexity in the arguments and the lower semicontinuity with respect to the narrow convergence. See~\cite[Section 5.3.1]{San15} for a more complete list of properties (note they refer to narrow convergence as weak convergence).
\vio By \eqref{eq:sw_metder;ubd} we expect  \nc
\[|\mu'|_{SW}^2(t)=\dashint_{\S^{d-1}}|(\hat\mu^\theta)'|_W^2(t)\,d\theta\leq \dashint_{\S^{d-1}}\B_2(\hat\mu_t^\theta,\theta\cdot\widehat{v_t\mu_t}^\theta)\,d\theta.\]
Thus it is natural to define the Benamou-Brenier functional \grn$B_{SW}:\P_2(\R^d)\times \mathscr{H}(\R^d;\R^d)\rightarrow(-\infty,+\infty]$ \nc for the sliced Wasserstein distance by
\begin{equation}\label{def:Bsw}
    B_{SW}(\mu,E)=\dashint_{\S^{d-1}}\B_2(\hat\mu^\theta,\theta\cdot\widehat E^\theta)\,d\theta= \left\|\,\theta\cdot\frac{d\widehat E}{d\hat\mu}\right\|_{L^2(\hat\mu)}^2.
\end{equation}
Note that $B_{SW}$ only depends on the Radon transforms \vio $\hat\mu,\hat E$ \nc of the inputs\grn. By definition \eqref{def:Hscr} of $\mathscr{H}(\R^d;\R^d)$, $\widehat E\in\M_b(\P_d;\R^d)$ hence its disintegration $\widehat E^\theta\in\M(\R;\R^d)$ with respect to $\vol_{\S^{d-1}}$ is well-defined for a.e. $\theta\in\S^{d-1}$, so is the integral in \eqref{def:Bsw}\nc. In general 
\[E\in \mathscr{H}(\R^d;\R^d)\subset\Sch'(\R^d;\R^d),\]
and $E$ is not necessarily a (vector-valued) measure. However, the Radon transform is well-defined for such $E$, unlike for general tempered distributions; see Remark~\ref{rmk:radon_on_measure} for further details. We will see later in Lemma~\ref{lem:Hts-leq-Bsw} that fluxes of interest lie in $\mathscr{H}(\R^d;\R^d)$.
    
We  record some basic properties that $B_{SW}$ inherits from the Benamou-Brenier functional $\B_2$.
\begin{proposition}[Properties of the $B_{SW}$]\label{prop:B_SWp}
The functional $B_{SW}$ is convex, and satisfies the following properties:
\begin{listi}
    \item Let $\mu,\mu_n\in\P_2(\R^d)$ and $E,E_n\in \mathscr{H}(\R^d;\R^d)$ be such that $\widehat E,\widehat E_n\in\M(\Pd;\R^d)$ and $(\hat\mu_n,\widehat E_n)$ narrowly converge to $(\hat\mu,\widehat E)$ in $\M(\Pd)\times\M(\Pd;\R^d)$. Then
    \begin{equation}\label{eq:Bsw;lsc}
        B_{SW}(\mu,E)\leq \liminf_n B_{SW}(\mu_n,E_n).
    \end{equation}
        
    \item $B_{SW}(\mu,E)\geq 0$
            
    \item $B_{SW}(\mu,E)<\infty$ only if $R_\theta \mu\geq 0$ and $R_\theta E\ll R_\theta\mu$ for a.e. $\theta\in\S^{d-1}$            
        
    \item Let $\mu\geq 0$ and suppose $R_\theta E\ll R_\theta \mu$ for a.e. $\theta\in\S^{d-1}$. Then we can write
    \[B_{SW}(\mu,E)=\frac12\dashint_{\S^{d-1}} \int_{\R} \left|\frac{\theta\cdot\widehat{E}(r,\theta)}{\widehat \mu(s,\theta)}\right|^2\,d\widehat\mu(r,\theta)\,d\theta.\]
\end{listi}
\end{proposition}
    
\begin{proof}
Items (i)-(iv) follow directly from the analogous property for $\B_2$ (see~\cite[Proposition 5.18]{San15}, and also~\cite[Lemma 8.1.10]{AGS} for item (iv)).
\end{proof}

The following lemma relates attenuated Sobolev norms and $B_{SW}$, and shows that the domain of interest is a subset of $\mathscr{H}(\R^d;\R^d)\subset\Sch'(\R^d;\R^d)$. \grn Given $v,\theta\in\R^d$, we henceforth write $v||\theta$ to say that $v$ and $\theta$ are parallel. \nc
\begin{lemma}[$B_{SW}$ and $H^{-d/2-}_{-(d-1)/2}$]\label{lem:Hts-leq-Bsw}
    Let $\rj\in\M(\Pd;\R^d)$ be such that $|\rj|$ is a finite measure. Then $\rj\in \bigcap_{\eps>0} H^{-1/2-\eps}(\Pd)$ and there exists $J\in \mathscr{H}(\R^d;\R^d)\subset H^{-d/2^-}_{-(d-1)/2}(\R^d;\R^d)\subset\Sch'(\R^d;\R^d)$ such that $RJ=\rj$ \grn and for all $\eps>0$ \nc
    \begin{equation}\label{eq:Hts-measure}
        \|J\|_{H^{-\frac d2-\eps}_{-(d-1)/2}(\R^d)}=\grn \|\rj\|_{H^{-\frac12-\eps}(\Pd)} \nc.
    \end{equation}

If instead $\rj\in\M(\Pd;\R^d)$ and $\mu\in\P_2(\R^d)$ satisfy
    \begin{listi}
        \item ($\B_{2}^\R$-upper bound) $\dashint_{\S^{d-1}}\B_2(\hat\mu^\theta,\theta\cdot\rj^\theta)\,d\theta<\infty$
        \item (parallel to $\theta$) $\frac{d\rj}{d\hat\mu}(\theta,r)||\theta$ for $\hat\mu$-a.e. $\theta,r$
    \end{listi}
then $\rj$ is a finite measure thus there exists $J\in \mathscr{H}(\R^d;\R^d)\subset\Sch'(\R^d;\R^d)$ such that $RJ=\rj$. Moreover, for each $\eps>0$ there exists $C(\eps)>0$ such that
    \begin{equation}\label{eq:Hts-leq-Bsw}
        \|J\|_{H^{-\frac d2-\eps}_{-(d-1)/2}(\R^d)}=\|\rj\|_{H^{-\frac12-\eps}(\Pd)}=\|\theta\cdot \rj\|_{H^{-\frac12-\eps}(\Pd)} \leq C(\eps) B_{SW}^{1/2}(\mu,J).
    \end{equation}
\end{lemma}

\begin{remark}\label{rmk:Hts-leq-Bsw}
\grn As we see in the proof, assumption (i) along with Proposition~\ref{prop:B_SWp} (iii) implies that $\frac{d\rj}{d\hat\mu}\in L^2(\hat\mu;\R^d)$\nc. In Theorem~\ref{thm:SW2_ac-curves} we will see that the flux $J_t$ such that $B_{SW}(\mu_t,J_t)\leq |\mu'|_{SW}(t)$ satisfies $\frac{d\widehat J_t}{d\hat\mu_t}(\theta,r)\parallel \theta$ for $\hat\mu_t$- a.e. $\theta,r$ thus the condition \grn(ii) \nc is automatically satisfied.
\end{remark}

\begin{proof}
Let $\rj\in\M(\Pd;\R^d)$ and first consider the case $|\rj|$ is a finite measure. Then for each test function $\varphi\in C_c^\infty(\Pd)$ and for any $\eps>0$,
\begin{align*}
    \int_{\Pd}\varphi\,d\rj \leq \|\varphi\|_{L^\infty}\int_{\Pd}\,d|\rj| \leq C(\eps) \|\varphi\|_{H^{\frac12+\eps}(\Pd)}\int_{\Pd}\,d|\rj|
\end{align*}
where we have used the Sobolev embedding theorem in one dimension in the last line.
Thus $\rj\in H^{-1/2-}(\Pd)$ and \eqref{eq:Radon_isometry;Hts} implies the existence of $J^\eps\in H^{-d/2-\eps}_{-(d-1)/2}(\R^d)\in\Sch'(\R^d)$ such that
\[\grn RJ^\eps=\rj \nc  \text{ and } \|J^\eps\|_{H^{-d/2-\eps}_{-(d-1)/2}(\R^d)}\grn=\|\rj\|_{H^{-1/2-\eps}(\Pd)}\nc.\]
\grn To conclude \eqref{eq:Hts-measure} it suffices to note that $J^\eps\in \Sch'(\R^d;\R^d)$ is independent of the choice of $\eps>0$. Indeed, suppose $\eps>\tilde\eps>0$, then \eqref{eq:duality;radon;dist} implies that for all $f\in \Sch(\R^d)$ \nc
\begin{align*}
    \langle J^\eps - J^{\tilde\eps}, f\rangle = c_d^{-1} \langle J^\eps - J^{\tilde\eps}, R^\ast \Lambda_d Rf\rangle = c_d^{-1} \langle \widehat J^\eps - \widehat J^{\tilde\eps} , \Lambda_d Rf\rangle = c_d^{-1} \langle \ru\hat\mu-\ru\hat\mu,\Lambda_d\hat f\rangle =0.
\end{align*}

On the other hand, suppose (i) and (ii) hold. By (ii) $\B_2(\hat\mu^\theta,\theta\cdot\rj)=\B_2(\hat\mu^\theta,|\rj|)$, which by (i) is finite a.e. $\theta\in\S^{d-1}$. Thus, by analogous property of $\B_2$ to Proposition~\ref{prop:B_SWp} (iii), we may write $\rj = \ru d\hat\mu$ for some $\ru\in\Pd\rightarrow\R^d$ satisfying $\ru\parallel\theta$ and $\|\ru\|_{L^2(\hat\mu)}=\|\theta\cdot\ru\|_{L^2(\hat\mu)}<\infty$. By Jensen's inequality,
\[\int_{\Pd}d|\rj|=\int_{\Pd} |\ru|\,d\hat\mu \leq \int_{\Pd} |\ru|^2\,d\hat\mu = \dashint_{\S^{d-1}} \B_2(\hat\mu^\theta,\theta\cdot\rj^\theta)\,d\theta<\infty\]
thus by the previous part we can find $J\in H^{-d/2-\eps}_{-(d-1)/2}(\R^d)$ with $RJ=\rj$. Finally, \eqref{eq:Hts-leq-Bsw} follows directly from property (ii) and \eqref{eq:Hts-measure}.
\end{proof}

We characterize absolutely continuous curves in the sliced Wasserstein space by identifying them with solutions $\CE_I$ of the continuity equation, defined in the following way.
\begin{definition}\label{def:CE}
Let $I\subset\R$ be an open interval. We denote by $\CE_I$ the set of all pairs $(\mu_t,J_t)_{t\in I}$ satisfying the following conditions:
\begin{listi}
    \item The curve $(\mu_t)_{t\in I}$ is narrowly continuous in $\P_2(\R^d)$ with respect to $t\in I$; 
    
    \item $J=R^{-1}\rj$ for some vector-valued Borel measure $\rj\in\M(I\times\Pd;\R^d)$ where the inverse Radon transform applied in the $\Pd$-variable 
    and
    \begin{equation}\label{cond:CE;integrable}
         \int_{\tilde I}\dashint_{\S^{d-1}}\int_{\R} d|\rj(r,\theta,t)| = \int_{\tilde I}\dashint_{\S^{d-1}}\int_{\R} d|\rj_t(r,\theta)|\,dt<\infty \text{ for all } \tilde I\subset\subset I,
    \end{equation}
    where $\rj\restr_{\tilde I\times\Pd}$ admits the disintegration $\rj=\int_{\tilde I} \rj_t\,d\L^1\restr_{\tilde I}(t)$;
    
    \item \purp $(\mu_t,J_t)_{t\in I}$ is a distributional solution of the continuity equation -- i.e.
    \[\int_I \langle\mu_t,\partial_t\varphi(t,\cdot)\rangle_{\R^d}+ \langle J_t,\nabla\varphi(t,\cdot)\rangle_{\R^d}\,dt =0 \text{ for all } \varphi\in C_c^\infty(I\times\R^d).\]\nc
\end{listi}
Moreover, we write $\CE_I(\bar\mu_0,\bar\mu_1)$ the set of pairs $(\mu,J)\in\CE_I$ satisfying $\mu_a=\bar\mu_0$ and $\mu_b=\bar\mu_1$ where $a=\inf I$ and $b=\sup I$.
\end{definition}

\begin{remark}\label{rmk:CE}
Given $\rj\in\M(I\times\Pd;\R^d)$ satisfying \eqref{cond:CE;integrable}, for each $\tilde I\subset\subset I$ the disintegratation theorem with respect to $\L^1\restr_{\tilde I}$ implies that $\rj_t\in\M(\Pd;\R^d)$ is well-defined and is bounded for a.e. $t\in\L^1\restr_{\tilde I}$. By considering an increasing countable sequence of $\tilde I_n\subset\subset I$ such that $\bigcup_n\tilde I_n= I$, we may define $\rj_t\in\M(\Pd;\R^d)$ for $\L^1$-a.e. $t\in I$. Then Lemma~\ref{lem:Hts-leq-Bsw} allows us to define the Radon inverse $J_t=R^{-1}\rj_t\in \mathscr{H}(\R^d;\R^d)$ for a.e. $t \in\ I$. 

Furthermore, condition (ii) is not restrictive whenever $\rj_t\parallel \theta$, as
\begin{align*}
    \int_{\tilde I}\dashint_{\S^{d-1}} \int_{\R} |RJ_t(d\theta,dr)|\,dtd\theta
    &\leq \int_{\tilde I} \, \norm{\frac{d\widehat J_t}{d\hat\mu_t}}_{L^1(\hat\mu_t)}\!\! dt
    \lesssim \int_{\tilde I} \, \norm{\frac{d\widehat J_t}{d\hat\mu_t}}_{L^2(\hat\mu_t)}\!\! dt \\
    &\leq \int_{\tilde I} B_{SW}^{1/2}(\mu_t,J_t)\,dt \leq \sqrt{|\tilde I|}\left(\int_I B_{SW}(\mu_t,J_t)\,dt\right)^{1/2}
\end{align*}
and we are only interested in absolutely continuous curves, for which we will see that the right-hand side is finite. 
\end{remark}

\purp
We state a useful technical lemma, the proof of which we delay to Appendix~\ref{app:CE_proj}.
\begin{lemma}\label{lem:CE_proj}
Let $I\subset \R$ be an open interval. Let $(\mu_t,J_t)_{t\in I}\in\CE_I$ as defined in Definition~\ref{def:CE}. 
Then for a.e. $\theta\in\S^{d-1}$, $t\mapsto(\hat\mu_t^\theta,\widehat J_t^\theta)$ is a distributional solution of the continuity equation on $I\times\R$ -- i.e.
\begin{equation}\label{eq:CE_dist;proj}
\int_I \langle\hat\mu_t^\theta,\partial_t\psi(t,\cdot)\rangle_{\R}+ \langle\theta\cdot\widehat J_t,\partial_r\psi(t,\cdot)\rangle_{\R}\,dt =0 \text{ for all } \psi\in C_c^\infty(I\times\R).
\end{equation}
\end{lemma}
In case $J=\int_I J_t\,dt\in\M_b(I\times\R^d)$, then by approximation $t\mapsto (\mu_t,J_t)$ solve the continuity equation against test functions $C_b^1(I\times\R^d)$. Thus for each $\theta\in\S^{d-1}$, we may take test functions of the form $\varphi(t,x)=\psi(t,x\cdot\theta)$ for any $\psi\in C_c^\infty(I\times\R)$ and deduce that \eqref{eq:CE_dist;proj} holds. However, $(\mu_t,J_t)_{t\in I}\in \CE_I$ in general enjoys less regularity, thus we need to instead rely on an approximation argument involving fine properties of the Radon transform. As the proof relies on technical tools that are irrelevant to other materials of this paper, we delay the proof to the appendix.
\nc

We are now ready to establish the main theorem of this section.
\begin{theorem}[AC curves in the sliced Wasserstein metric space]\label{thm:SW2_ac-curves}
Let $B_{SW}$ be as defined in \eqref{def:Bsw}, and $I\subset\R$ an open interval. 
        
\begin{listi}
    \item Suppose $(\mu_t, J_t)_{t\in I}\in\CE_I$ satisfying $t\mapsto B_{SW}(\mu_t,J_t)\in L^1(I)$. Then, for $\L^1$-a.e. $t\in I$
    \begin{equation}\label{eq:ddt_SW2}
        |\mu'|_{SW}^2(t)\leq B_{SW}(\mu_t,J_t),
    \end{equation}
    and $(\mu_t)_{t\in I}\in AC(I;\P_2(\R^d),SW)$. 
    
    \item Conversely, let $(\mu_t)_{t\in I}\in AC(I;\P_2(\R^d),SW)$.
    Then there exists 
    \begin{equation}\label{eq:Jt;space}
        J_t\in
        \overline{\left\{ \Lambda_d R^\ast(\widehat\mu_t\Lambda_d(\widehat{\nabla\varphi}))\;:\;\varphi\in C_c^\infty(\R^d)\right\}}^{B_{SW}(\mu,\cdot)}
        \subset \mathscr{H}(\R^d;\R^d) \text{ for a.e. } t\in I.
    \end{equation}
    such that $(\mu_t,J_t)_{t\in I}\in \CE_I$ and
    \begin{equation}\label{eq:BleqMetder}
        B_{SW}(\mu_t,J_t) \leq |\mu'|_{SW}^2(t) \text{ for a.e. } t\in I.
    \end{equation} 
    \end{listi}        
\end{theorem}

\begin{proof}
We adapt the proof of the analogous result by Ambrosio, Gigli, and Savar\'e~\cite[Theorem 8.3.1]{AGS}; we leave the structure parallel to their proof, so readers can compare the objects arising in the sliced Wasserstein setting to the Wasserstein counterparts.

\vspace{3mm}
\noindent\emph{Step 1$^o$.}
By assumption $(J_t,\mu_t)_{t\in I}\in\CE_I$ and thus we know $RJ_t\in\M(\Pd;\R^d)$ is well-defined for a.e. $t\in I$. We may deduce by Proposition~\ref{prop:B_SWp} (iii) that $R_\theta J_t\ll\hat\mu_t^\theta$ for a.e. $t\in I$ and a.e. $\theta\in\S^{d-1}$; hence write $RJ_t= \rv_t\hat\mu_t$ and note
\[\int_I B_{SW}(\mu_t,J_t)=\frac12\int_I \dashint_{\S^{d-1}} \|\theta\cdot\rv_t^\theta\|_{L^2(\hat\mu_t^\theta)}^2\,d\theta\,dt<\infty,\]
By Lemma~\ref{lem:CE_proj}, we have, in the sense of distributions,
\[\partial_t\hat\mu_t^\theta+ \grn\partial_r\nc (\theta\cdot\rv_t^\theta \hat \mu_t^\theta)=0 \text{ for a.e. } \theta\in\S^{d-1}.\]
Thus, by~\cite[Theorem 8.3.1]{AGS} we see that for a.e. $\theta\in\S^{d-1}$, each $(\hat\mu_t^\theta)_{t\in I}\in AC(\P_2(\R),W)$, and further \grn
\[|(\hat\mu^\theta)'|_{W}(t)\leq \int_{\R} |\theta\cdot\rv_t(\theta,r)|^2\,d\hat\mu^\theta(r) \text{ for a.e. } \theta\in\S^{d-1},\;t\in I.\]
Combining, we have
\[\int_I |\mu'|_{SW}^2(t)\,dt = \dashint_{\S^{d-1}} |(\hat\mu^\theta)'|_{W}(t)^2\,dt \leq \int_I \dashint_{\S^{d-1}} \int_{\R} |\theta\cdot\rv_t(\theta,r)|^2\,d\hat\mu^\theta(r)\,d\theta \,dt = \int_I B_{SW}(\mu_t,J_t)\,dt.\]
\nc
\smallskip

\noindent\emph{Step 2$^o$.} 
It remains to prove the converse. Let $\varphi\in C_c^\infty(\R^d)$. Then by the Radon inversion formula \eqref{eq:radon_inversion_f2}
\[\varphi=c_d^{-1}R^\ast \Lambda_d R \varphi,\]
and thus for each $\mu\in\P_2(\R^d)$, duality formula \eqref{eq:duality;radon;measure} implies
\begin{align*}
    \mu(\varphi)&:=\int_{\R^d}\varphi(x)\,d\mu(x)
    =c_d^{-1}\int_{\R^d} \dashint_{\S^{d-1}} \Lambda_d R_\theta\varphi (x\cdot\theta)\,d\theta\,d\mu(x)   \\
    &=c_d^{-1}\dashint_{\S^{d-1}} \int_{\R^d} \Lambda_d R_\theta \varphi (x\cdot\theta)\,d\mu(x)\,d\theta
    =c_d^{-1}\dashint_{\S^{d-1}} \int_{\R} \Lambda_d R_\theta \varphi(r)\,d\widehat\mu^\theta(r)\,d\theta.
\end{align*}
Let $(\mu_t)_{t\in I}\in AC(\P_2(\R^d);SW)$. 
Denoting by $\hat\gamma_{s,t}\in\widehat\Gamma_o(\widehat\mu_t,\widehat\mu_s)$ for each $s,t\in I$,
\begin{align*}
    |\mu_t(\varphi)-\mu_s(\varphi)|
    &=c_d^{-1}\left|\dashint_{\S^{d-1}}\int_{\R^d\times\R^d} \Lambda_d R_\theta\varphi(q)-\Lambda_d     R_\theta \varphi(r)\,d\hat\gamma^\theta_{s,t}(q,r)\,d\theta\right|\\
    &\leq c_d^{-1} \sup_{\theta\in\S^{d-1}}\Lip(\Lambda_d R_\theta\varphi)SW(\mu_t,\mu_s).
\end{align*}
Furthermore, define $H^\theta=H^\theta(\varphi)$ by \grn
\begin{align*}
    H_\theta(r,q):=
    \begin{cases}
        |\Lambda_d \partial_r R_\theta\varphi(r)| \text{ when } r=q, \text{ and } \\
        \frac{|\Lambda_d R_\theta\varphi(r)-\Lambda_d R_\theta\varphi(q)|}{|r-q|} \text{ when } r\neq q.
    \end{cases}
\end{align*}
\nc Then
\begin{align*}
\frac{\mu_{s+h}(\varphi)-\mu_s(\varphi)}{|h|}
& \leq \frac{1}{c_d |h|} \dashint_{\S^{d-1}} \int_{\R^d\times\R^d} |r-q|H_\theta(r,q)\,d\gamma_{s+h,h}^\theta(r,q)\,d\theta    \\
& \leq  \frac{SW(\mu_{s+h},\mu_s)}{ c_d |h|}\left(\dashint_{\S^{d-1}}\int_{\R^d\times\R^d} H_\theta^2(r,q) \,d\gamma_{s+h,h}^\theta(r,q)\,d\theta\right)^{1/2}.
\end{align*}
Let $\varphi\in C_c^\infty(I\times \R^d)$ and $(\mu_t)_{t\in I}\in AC(\P_2(\R^d);SW)$. Then
\begin{equation}\label{eq:dsphi_est}
\begin{split}
    \left|\int_{I} \int_{\R^d} \partial_s\varphi(x,s)\,d\mu_s(x)\right|
    &\leq c_d^{-1}\left(\int_{\tilde I} |\mu'|_{SW}^2(s)\,ds\right)^{\frac12}\left(\int_{\tilde I} \dashint_{\S^{d-1}} \int_{\R}|\Lambda_d \partial_r R_\theta\varphi(r)|^2\,d\widehat\mu_t^\theta(r)\,d\theta\right)^{\frac12}\\
    &\leq c_d^{-1}\left(\int_{\tilde I} |\mu'|_{SW}^2(s)\,ds\right)^{\frac12} \left(\int_{\tilde I} \|\Lambda_d R(\nabla\varphi)\|_{L^2(\widehat\mu_t)}^2\right)^{\frac12}
\end{split}
\end{equation}
where the Radon transform is applied in the spatial variable, and $\tilde I\subset I$ is any interval such that $\supp\varphi\subset \tilde I\times \R^d$.  
Let
\[V:=\left\{\Lambda_d\widehat{\nabla\varphi}:\;\varphi\in C_c^\infty(I\times \R^d)\right\}\]
and let $\overline V$ be the closure of $V$ under the norm 
\[\|\cdot\|_{L^2((\hat\mu_t)_{t\in I})}:=\left(\int_{I} \|\cdot\|_{L^2(\hat\mu_t)}^2\,dt\right)^{1/2}.\]
As $\Lambda_d\widehat{\nabla\varphi_t}(\theta,r)=\theta\Lambda_d\partial_r\hat\varphi_t^\theta(r)$, any vector $\ru\in \overline V$ evaluated at each time is parallel to $\theta$ -- i.e. $\ru_s(\theta,r)\parallel \theta$ for all $(s,\theta,r)\in I\times\S^{d-1}\times\R$.
By the estimate \eqref{eq:dsphi_est}, we can extend the functional $L:V\rightarrow\R$
\[L(\Lambda_d\widehat{\nabla\varphi}):=-\int_I \int_{\R^d} \partial_s \varphi_s(x)\,d\mu_s(x)\]
uniquely to a bounded linear functional on $\overline V$, such that
\begin{equation}\label{eq:Lu-bd}
    |L(\ru)|\leq c_d^{-1}\left(\int_{\tilde I} |\mu'|^2_{SW}(s)\,ds\right)^{\frac12} \left(\int_{\tilde I} \|\ru_s\|_{L^2(\widehat\mu_s)}^2\,ds\right)^{\frac12} \text{ for any } \ru\in\overline V \text{ with } \supp \ru\subset\tilde I\times\R^d.
\end{equation}
Consider the minimization problem 
\[
\min\left\{\frac{1}{2c_d}\int_I \|\ru_s\|_{L^2(\widehat\mu_s)}^2\,ds- L(\ru):\;\ru\in\overline V\right\}=\min\left\{\frac{1}{2c_d}\int_I \|\theta\cdot\ru_s\|_{L^2(\widehat\mu_s)}^2\,ds- L(\ru):\;\ru\in\overline V\right\}.
\]
Note that the functional we are minimizing is the sum of a quadratic term and a bounded linear functional in $\|\cdot\|_{L^2((\hat\mu_t)_{t\in I})}$, thus is coercive in $\|\cdot\|_{L^2((\hat\mu_t)_{t\in I})}$ and lower semicontinuous in the weak topology of $L^2((\hat\mu_t)_{t\in I})$. Thus, by the direct method of calculus of variations, this minimization problem admits a solution $(\rv_t)_{t\in I}\in\overline V$. Furthermore, the functional is strictly convex, thus this minimizer is unique.
Moreover, the minimizer $(\rv_t)_{t\in I}$ satisfies the Euler-Lagrange equation 
\begin{align*}
    0&=c_d^{-1}\int_I \dashint_{\S^{d-1}}\int_{\R} \Lambda_d\widehat{\nabla\varphi}\cdot\rv_s\,d\hat\mu_s\,ds-L(\Lambda_d\widehat{\nabla\varphi}) \; \text{ for all } \varphi\in C_c^\infty(I\times\R^d).
\end{align*}
As $\rv_s\in L^2(\widehat\mu_s)$ for $\L^1$-a.e. $s\in I$, by Lemma~\ref{lem:Hts-leq-Bsw} we can find $J_s\in \mathscr{H}(\R^d;\R^d)\subset\Sch'(\R^d;\R^d)$ such that $RJ_s = \hat\mu_s\rv_s$ for $\L^1$-a.e. $s\in I$.
From this we deduce $(\mu_t,J_t)_{t\in I}$ satisfies the continuity equation in the sense of distributions, as \eqref{eq:duality;radon;dist} implies that for each $\varphi\in C_c^\infty(I\times\R^d)$, writing $\varphi_s(x)=\varphi(s,x)$,
\begin{align*}
    \int_I J_s(\nabla\varphi_s)\,ds &= \int_I c_d^{-1} RJ_s(\Lambda_d \widehat{\nabla\varphi_s})\,ds 
    =c_d^{-1}\int_I \dashint_{\S^{d-1}} \int_{\R} \Lambda_d \widehat{\nabla\varphi_s}(r)\cdot\,dRJ_s(r,\theta)\,ds  \\
    &= c_d^{-1}\int_I\dashint_{\S^{d-1}}\int_{\R} \Lambda_d\widehat{\nabla\varphi}\cdot \rv_s\,d\hat\mu_s\,ds
    =L(\Lambda_d\widehat{\nabla\varphi})
    =-\int_{I}\int_{\R^d}\partial_s\varphi_s(x)\,d\mu_s(x).
\end{align*}

    In order to establish the pointwise inequality \eqref{eq:BleqMetder}, first recall that, as $\widehat J_s\parallel \theta$ for a.e. $s\in I$,
    \begin{align*}
        B_{SW}\left(\mu_s,J_s\right)=\dashint_{\S^{d-1}}\int_{\R} \frac{d\widehat J_s}{d\widehat\mu_s}\cdot d\widehat J_s 
        =\|\rv_s\|_{L^2(\widehat\mu_s)}^2 \text{ for } \L^1\text{-a.e. }s\in I.
    \end{align*}
    Choose an interval $\tilde I\subset I$ and $\eta\in C_c^\infty(\tilde I)$ with $0\leq\eta\leq 1$, and $\varphi^k\in C_c^\infty(I\times\R^d)$ such that $\Lambda_d \widehat{\nabla\varphi_k}$ converges to the minimizer $\rv\in\overline V$ in $\|\cdot\|_{L^2((\widehat\mu_t)_{t\in I})}$. 
    By replacing with a suitable subsequence, we may further assume $\varphi^k_t=\varphi^k(t,\cdot)\in C_c^\infty(\R^d)$ converges to $\rv_t$ in $L^2(\hat\mu_t;\R^d)$ for $\L^1$-a.e. $t\in I$. Thus, using the bounds \eqref{eq:dsphi_est}, we obtain
    \begin{align*}
        &\int_I \dashint_{\S^{d-1}}\int_{\R} \eta\frac{d\widehat J_s}{d\widehat\mu_s}\cdot d\widehat J_s\,ds
        =\lim_{k\rightarrow\infty}\int_I \dashint_{\S^{d-1}}\int_{\R}  c_d^{-1}\eta\Lambda_d\widehat{\nabla\varphi^k}\cdot \rv_s\,d\hat\mu_s\,ds = \lim_{k\rightarrow\infty}L(\eta\Lambda_d\widehat{\nabla \varphi^k})\\
        &\leq \left(\int_{\tilde I} |\mu'|_{SW}^2(s)\,ds\right)^{\frac12} \lim_{k\rightarrow\infty}\left(\int_{\tilde I} \|\Lambda_d \widehat{\nabla\varphi^k_s}\|_{L^2(\widehat\mu_s)}^2\,ds\right)^{\frac12}
        \leq \left(\int_{\tilde I} |\mu'|_{SW}^2(s)\,ds\right)^{\frac12} \left(\int_{\tilde I} \|\rv_s\|_{L^2(\widehat\mu_s)}^2\,ds\right)^{\frac12},
    \end{align*}
    Letting $\eta\nearrow \one_{\tilde I}$, we have
    \[
    \int_{\tilde I} B_{SW}\left(\mu_s,J_s\right)\,ds =\int_{\tilde I} \dashint_{\S^{d-1}}\int_{\R} \frac{d\widehat J_s}{d\widehat\mu_s}\cdot d\widehat J_s\,ds
    \leq \left(\int_{\tilde I} |\mu'|_{SW}^2(s)\,ds\right)^{\frac12} \left(\int_{\tilde I} \|\rv_s\|_{L^2(\widehat\mu_s)}^2\,ds\right)^{\frac12}.
    \]
    As $B_{SW}(\mu_s,J_s)=\|\rv_s\|_{L^2(\widehat\mu_s)}^2$, by reorganizing and squaring both sides we obtain
    \[\int_{\tilde I} B_{SW}\left(J_s, \mu_s\right)\,ds  \leq \int_{\tilde I} |\mu'|_{SW}^2(s)\,ds.\]
    As $\tilde I\subset I$ was arbitrary, we deduce
    \[B_{SW}\left(J_t,\mu_t\right)\leq |\mu'|_{SW}^2(t) \text{ for } \L^1\text{-a.e. } t\in I.\]
    One can readily check that the pair $(\mu_t,J_t)_{t\in I}$ satisfies all the conditions in Definition~\ref{def:CE}, hence is in $\CE_I$; see Remark~\ref{rmk:CE}.
    \end{proof}

\begin{remark}\label{rmk:SW2_ac-curves}
The characterization of the space \eqref{eq:Jt;space} in which the identified flux $J_t$ belong will be useful in characterizing the tangent space in Section~\ref{ssec:sw-tanspace}. Note that this space consists of vectors satisfying $R J_t\parallel\theta$ thus assumption (ii) in Lemma~\ref{lem:Hts-leq-Bsw} is not restrictive, as noted in Remark~\ref{rmk:Hts-leq-Bsw}. 
    
As we do not know in general if the flux is a vector-valued measure, the question of whether $J_t\ll\mu_t$ may not even make sense. Furthermore, even if $J_t$ is sufficiently regular, it is obtained by the Radon inversion and thus it is difficult to determine whether the flux is absolutely continuous with respect to the measure $\mu_t$; the finiteness of $B_{SW}(\mu_t,J_t)$ implies $\widehat J_t\ll \hat\mu_t$ but not necessarily $J_t\ll\mu_t$. For instance, let $\mu$ be a normalized measure on the unit sphere $\partial B(0,1)$ in $\R^d$ and $\sigma\in\P_2(\R^d)$ be suitable normalization of $\one_{B(0,1/2)}$. Letting $\mu_t = (1-t)\mu + t \grn\sigma\nc$ for $t\in[0,1]$, one can check that $(\mu_t)_{t\in[0,1)}$ is an absolutely continuous curve in the SW space, as in each projection $\supp\hat\mu_t=\overline{B(0,1)}$ for all $t\in[0,1)$. On the other hand, for any $h>0$ the curve $(\mu_t)_{t\in[0,h)}$ cannot be absolutely continuous in the Wasserstein space, as mass is created outside the support of $\mu_t$.
\end{remark}
    
\subsection{Tangential structure of the sliced Wasserstein space}\label{ssec:sw-tanspace}
Unlike absolutely continuous curves in the Wasserstein space which allow corresponding solutions of the form $(\mu_t,v_t\mu_t)_{t\in I}$ of the continuity equation, Theorem~\ref{thm:SW2_ac-curves} only guarantees the solution of the continuity equation in the flux form $(\mu_t,J_t)_{t\in I}$ (see Remark~\ref{rmk:SW2_ac-curves}), where we only know $J_t\subset \mathscr{H}(\R^d;\R^d)\subset\Sch'(\R^d;\R^d)$ in general. Furthermore, from the proof we saw that \eqref{eq:Jt;space} must hold in order to ensure $B_{SW}(\mu_t,J_t)=|\mu'|_{SW}^2(t)$ for a.e. $t\in I$. This motivates us to define the tangent space at $\mu\in\P_2(\R^d)$ after the space appearing in \eqref{eq:Jt;space}, namely
\begin{equation}\label{def:tanspace;lsw:mu}
\begin{split}
    \tanspace_\mu(\P_2(\R^d),SW)&:=\overline{\left\{ \Lambda_d R^\ast(\widehat\mu\Lambda_d(\widehat{\nabla\varphi}))\;:\;\varphi\in C_c^\infty(\R^d)\right\}}^{B_{SW}(\mu,\cdot)}\\
    &=\grn R^{-1}\left[\hat\mu\overline{\left\{\Lambda_d\widehat{\nabla\varphi}:\;\varphi\in C_c^\infty(\R^d)\right\}}^{L^2(\hat\mu)}\right] \nc\subset \mathscr{H}(\R^d;\R^d).
\end{split}
\end{equation}
In this section we highlight some properties of the tangent space.
To begin, recall that in the case of the Wasserstein space, the tangent space $\tanspace_\mu (\P_2(\R^d),W)$ satisfies the optimality property~\cite[Lemma 8.4.2]{AGS}
\[v\in\tanspace_\mu(\P_2(\R^d),W) \text{ if and only if }\|v+w\|_{L^2(\mu)}\geq \|v\|_{L^2(\mu)} \text{ for all } \nabla\cdot(w\mu)=0.\]
We see that $\tanspace_{\mu}(\P_2(\R^d),SW)$ satisfies the analogous property.
\begin{proposition}[Tangent space and optimality]\label{prop:tanspace;lsw-optimal}
    Let $\mu\in\P_2(\R^d)$ and $J\subset \mathscr{H}(\R^d;\R^d)$ such that $B_{SW}(\mu,J)<\infty$. Then $J\in\tanspace_{\mu}(\P_2(\R^d),SW)$ if and only if 
\begin{equation}\label{eq:tanspace;lsw;minimizer}
    \norm{\frac{d\widehat J}{d\widehat\mu}+\frac{d\widehat E}{d\widehat\mu}}_{L^2(\hat\mu)} \geq \norm{\frac{d\widehat J}{d\hat\mu}}_{L^2(\hat\mu)}
\end{equation}
for all \grn$E\in \mathscr{H}(\R^d;\R^d)$ \nc such that $\frac{d\widehat E}{d\hat\mu}\in L^2(\hat\mu;\Pd^d)$ and $\nabla\cdot E=0$ (in the sense of distributions). Moreover, such minimizer $J\in \tanspace_{\mu}(\P_2(\R^d),SW)$ is unique.
    
\end{proposition}
\begin{proof}
\grn Squaring both sides of \eqref{eq:tanspace;lsw;minimizer}, a simple scaling argument reveals that \eqref{eq:tanspace;lsw;minimizer} is true if and only if
$\left\langle \widehat{E},\frac{d\widehat J}{d\hat\mu} \right\rangle_{\Pd}=0$ for all such $E$. Indeed, as $E\in\mathscr{H}(\R^d;\R^d)$ we may apply the duality formula \eqref{eq:duality;radon;dist} to see
\[\langle \widehat{E},\Lambda_d \widehat{\nabla\varphi}\rangle_{\Pd}=c_d\langle E,\nabla\varphi\rangle_{\R^d} \text{ for all } \varphi\in C_c^\infty(\R^d).\]
As $\nabla\cdot E=0$ in the sense of distributions, we see that $\left\langle \widehat{E},\frac{d\widehat J}{d\hat\mu} \right\rangle_{\Pd}=0$ if and only if
$\frac{d\widehat J}{d\hat\mu}$ is in the $L^2(\hat\mu)$ closure of $\Lambda_d\widehat{\nabla\varphi}$, which characterizes $\tanspace_\mu(\P_2(\R^d),SW)$.

Furthermore, uniqueness follows from strict convexity of the $L^2(\hat\mu)$-norm and the linearity of the continuity equation in the flux. \nc
\end{proof}

From this we deduce the following key property of the tangent space.
\begin{proposition}\label{prop:sw-ac;Jt_unique}
Let $(\mu_t)_{t\in I}\in AC(\P_2(\R^d);SW)$. For any $(\mu_t,J_t)_{t\in I}\in\CE_I$, we have
\[|\mu'|_{SW}^2(t)=B_{SW}(\mu_t,J_t) \text{ for a.e. } t\in I \grn\text{ if and only if }\nc J_t\in\tanspace_{\mu_t}(\P_2(\R^d),SW) \text{ for  a.e. } t\in I.\]
In particular, such $(J_t)_{t\in I}$ is determined uniquely for a.e. $t\in I$ given $(\mu_t)_{t\in I}\in AC(\P_2(\R^d);SW)$.
\end{proposition}

\begin{proof}
    Let $(\mu_t,J_t)_{t\in I}\in \CE_I$, then we have $J_t\in \mathscr{H}(\R^d;\R^d)$ for a.e. $t\in I$ -- see Remark~\ref{rmk:CE}. Thus by Theorem~\ref{thm:SW2_ac-curves} we know that in general $|\mu'|_{SW}^2(t)\leq B_{SW}(\mu_t,J_t)$, where the equality is attained by some flux $J_t$ for a.e. $t\in I$. Finally, Proposition~\ref{prop:tanspace;lsw-optimal} implies the minimizer is in the respective tangent space $\tanspace_{\mu_t}(\P_2(\R^d),SW)$, and that it is unique.
\end{proof}

\begin{remark}[Nonexistence of \grn $\ell_{SW}$-geodesics \nc in some directions]\label{rmk:Vmu}    
We emphasize that not every flux $J\in\tanspace_{\mu}(\P_2(\R^d),SW)$ can be attained as a velocity flux of an absolutely continuous curve. The following example illustrates that some fluxes $J$ are admissible while $-J$ is not. Fix a small $\eps>0$ and consider
\[\mu_t=c((2-c_\eps t) \grn\one_{B(0,1)\setminus B(0,\eps)}\nc+t\one_{B(0,\eps)})\,dx \in \P_2(\R^d) \text{ where } c,c_\eps \text{ are normalizing constants }.\]   
To allow sufficient regularity of the objects constructed from the Radon inversion, we consider $\mu_t^{\eps^2}=\mu_t\ast\eta_{\eps^2}$; choice of the convolution radius $\eps^2$ ensures that when $\eps<\frac12$ then
$|\mu_0^{\eps^2}|(B(0,\eps/2))= 0$.
On the other hand $R_\theta \mu_0^{\eps^2}>0$ in the interior of its support as long as $\eps$ is chosen sufficiently small. 
Clearly $(\mu_t^{\eps^2})_{t\in[0,1]}\in AC(\P_2(\R^d);SW)$, and thus by Theorem~\ref{thm:SW2_ac-curves} we can find corresponding $(J_t)_{t\in I}$ in the tangent space, and as $\mu_t^{\eps^2}$ is smooth, $J_t$ must also be a smooth function. As
\[\partial_t\mu_t^{\eps^2}|_{t=0}=-(\nabla\cdot J_0) > 0 \text{ on } B(0,\eps/2),\]
proceeding in the direction $-J_0$ from $\mu_0$ introduces negative mass in $B(0,\eps/2)$. From this we see $J_0$ is achievable as a tangent vector to a curve, but $-J_0$ is not.

\grn In general, not all fluxes in the tangent space vanish outside the support of $\mu$, hence cannot be attained by absolutely continuous curves in the space of probability measures. Thus, despite the definition \eqref{def:tanspace;lsw:mu} of $\tanspace_\mu(\P_2(\R^d),SW)$ as a vector space, the tangent vectors attainable by curves form a convex cone rather than a linear space. This suggests that $(\P_2(\R^d);\ell_{SW})$ should be (formally) considered as a manifold with corners.  \nc
\end{remark}

\section{The sliced Wasserstein length space}\label{sec:ellSW}

Given an interval $I$  and
 a curve $(\mu_t)_{t\in I}\in AC(\P_2(\R^d);SW)$, define its sliced Wasserstein length $L_{SW}((\mu_t)_{t\in I})$ by
\begin{equation}\label{def:length}
L_{SW}((\mu_t)_{t\in I})=\int_I |\mu'|_{SW}(t)\,dt.
\end{equation}
We note that \eqref{def:length} is consistent with the usual notion of length in a metric space (see~\cite[Theorem 2.7.6]{BurIva01}):
\begin{equation}\label{eq:Lsw=sup}
L_{SW}((\mu_t)_{t\in I})=\sup\left\{\sum_{i=0}^{n-1} SW(\mu_{t_i},\mu_{t_{i+1}}):\;\; t_0<t_1<\cdots<t_n,\;\: t_i \in I \te{ for } i=0, \dots,n  \right\}.
\end{equation}

In this section we examine the length metric $\ell_{SW}$ induced by $SW$
\begin{equation}\label{def:lsw}
\begin{split}
    \ell_{SW}(\mu,\nu)=&\inf\left\{L_{SW}((\mu_t)_{t\in[0,1]}):\;(\mu_t)_{t\in[0,1]}\in AC([0,1];\P_2(\R^d),SW), \;\, \mu_0=\mu,\;\mu_1=\nu\right\}
\end{split}
\end{equation}
and the associated length space $(\P_2(\R^d),\ell_{SW})$. \grn As $SW\leq \frac{1}{\sqrt{d}}W$ and $W$ coincides with its length metric, it immediately follows that $\ell_{SW}\leq \frac{1}{\sqrt{d}}W$.\nc
    
While in general the study of the intrinsic metric and the geometry is mathematically natural, it is particularly relevant for $SW$ for the following reasons. Firstly, using the characterization of metric derivatives via the quadratic functional $B_{SW}$ in Theorem~\ref{thm:SW2_ac-curves}, we can consider a formal Riemannian structure on $(\P_2(\R^d),\ell_{SW})$ analogous to that on $(\P_2(\R^d),W)$. Furthermore, in applications we are often interested in continuous deformations of probability measures, hence the geodesic distance that can be attained by absolutely continuous curves can be more relevant than the original distance. 

After noting $(\P_2(\R^d),\ell_{SW})$ is a complete metric space, in Lemma~\ref{lem:curve_cpct} we show the narrow precompactness of absolutely continuous curves in $(\P_2(\R^d), SW)$ and the lower semicontinuity of $L_{SW}$. From this we deduce the lower semicontinuity of $\ell_{SW}$ respect to the narrow convergence in Lemma~\ref{cor:lsw;lsc} and the existence of $\ell_{SW}$ geodesics in Proposition~\ref{prop:lsw-geodesic}; in particular, the latter implies that in general $\ell_{SW}\neq SW$, as we have seen in Example~\ref{ex:SW2_notgeo} that $(\P_2(\R^d),SW)$ is not a geodesic space. 
\medskip

We first note that completeness of $(\P_2(\R^d),\ell_{SW})$ follows from completeness of $(\P_2(\R^d),SW)$.
\begin{corollary}[Completeness]\label{cor:lsw_complete}
    $(\P_2(\R^d),\ell_{SW})$ is a complete metric space.
\end{corollary}
\begin{proof}
    \grn Note that any Cauchy sequence $(\mu_n)$ in $(\P_2(\R^d),\ell_{SW})$, is also Cauchy with respect to $SW$. By Proposition~\ref{prop:SW2_complete} we can find a limit $\mu_0\in\P_2(\R^d)$ such that $SW(\mu_n,\mu_0)\xrightarrow[]{n\rightarrow\infty} 0$. By the topological equivalence of $SW$ and $W$~\cite[Theorem 2.3]{BayGuo21}, $\ell_{SW}(\mu_n,\mu_0)\leq d^{-1/2}W(\mu_n,\mu_0)\xrightarrow[]{n\rightarrow\infty} 0$. \nc
\end{proof}

In locally compact metric spaces the compactness of paths, lower semicontinuity of length, and existence of geodesics follow by classical arguments; see Section 4 of~\cite{AmbTil}. However, in $(\P_2(\R^d), SW)$ balls are not precompact; see Remark~\ref{rmk:SWball=notswcpct}. On the other hand balls are precompact with respect to the narrow topology (Proposition~\ref{prop:SWball=wcpct}) and the SW distance is lower semicontinuous with respect to narrow convergence. This allows to use instead the following refined version of Ascoli-Arzel\`a theorem~\cite[Proposition 3.3.1]{AGS} to construct limiting curves and establish the existence of geodesics. 

\begin{proposition}[Proposition 3.3.1. of~\cite{AGS}]\label{prop:ArzelaAscoli-weak}
Let $(X,m)$ be a complete metric space. Let $T>0$ and $K\subset X$ be a sequentially compact set with respect to topology $\sigma$, and let $u_n:[0,T]\rightarrow X$ be curves such that
\begin{equation}\label{eq:u_n;eqcts}
u_n(t)\in K\quad\forall n\in\N,\;t\in[0,T],\\
\limsup_{n\rightarrow\infty} m(u_n(s),u_n(t))\leq \omega(s,t)\quad \forall s,t\in[0,T],
\end{equation}
for a symmetric function $\omega:[0,T]\times[0,T]\rightarrow[0,+\infty)$, such that
\[\lim_{(s,t)\rightarrow(r,r)}\omega(s,t)=0\quad\forall r\in[0,T]\setminus\mathscr{C}\]
where $\mathscr{C}$ is an (at most) countable subset of $[0,T]$. Then there exists an increasing subsequence $k\mapsto n(k)$ and a limit curve $u:[0,T]\rightarrow X$ such that
\[u_{n(k)}(t)\xrightharpoonup[]{\sigma}u(t)\quad\forall t\in[0,T],\;u \text{ is continuous with respect to }  m \text{ in } [0,T]\setminus\mathscr{C}.\]
\end{proposition}

Setting $m=SW$ and $\sigma$ to be the topology generated by the narrow convergence in Proposition~\ref{prop:ArzelaAscoli-weak}, we can modify the standard arguments to show pointwise narrow compactness of curves.

\begin{lemma}[Pointwise narrow compactness for curves and lower semiconinuity of length]\label{lem:curve_cpct}
$ $\newline\noindent Let $I$ be a closed interval
and suppose a sequence of curves $(\mu^k_t)_{t\in I}\in AC(\P_2(\R^d);SW)$ satisfies
\[\sup_{k\geq 1} L_{SW}((\mu^k_t)_{t\in I})<\infty \text{ and } \sup_{k,l\geq 1} SW(\mu^k_0,\mu^l_0)<\infty.\]
Then, up to a reparametrization, there exists a curve $(\mu_t)_{t\in I}$ continuous in $SW$ such that along a subsequence (which we do not relabel)
\[\mu^k_t\rightharpoonup \mu_t \text{ narrowly for all } t\in I .\]
Moreover,
\begin{equation}\label{eq:Lsw;lsc}
    L_{SW}((\mu_t)_{t\in I})\leq \liminf_{k\rightarrow\infty} L_{SW}((\mu^k_t)_{t\in I}).
\end{equation}
In particular, $(\mu_t)_{t\in I}\in AC(\P_2(\R^d);SW)$.
\end{lemma}

\begin{proof}
As each $(\mu_t^k)_{t\in I}$ is an absolutely continuously curve with uniformly bounded length, we may instead consider their Lipschitz reparametrizations~\cite[Lemma 1.1.4]{AGS} to the interval $I:=[0,1]$ with each of the Lipschitz constant is bounded above by the length of the curve. Thus the equicontinuity condition \eqref{eq:u_n;eqcts} is satisfied at all points $s,t\in I$ with \grn$\omega(s,t)=|s-t|\sup_k L_{SW}((\mu^k_\tau)_{\tau\in I})$\nc. Furthermore, the condition $\sup_{k,l\geq 1}SW(\mu^k_0,\mu^l_0)<\infty$ allows us to choose $\nu\in\P_2(\R^d)$ such that $M_2:=\sup_k SW(\nu,\mu^k_0)$ is finite. 
Then
\[SW(\mu^k_t,\nu)\leq SW(\mu^k_0,\mu^k_t)+SW(\mu^k_0,\nu)\leq L((\mu^k_t)_{t\in I})+\sup_k SW(\nu,\mu^k_0)\leq M_1+M_2.\]
By Proposition~\ref{prop:SWball=wcpct}, $\overline B_{SW}(\nu,M_1+M_2)$ is compact with respect to the narrow topology. Thus the refined Ascoli-Arzel\`a Theorem (Proposition~\ref{prop:ArzelaAscoli-weak}) implies the existence of curve $(\mu_t)_{t\in I}$ continuous in $t\in I$ such that $(\mu^k_t)_{t\in I}$ pointwise converge narrowly at all $t\in I$ as $k\nearrow \infty$.

By Proposition~\ref{prop:SW2_lsc} $(\mu,\nu)\mapsto SW(\mu,\nu)$ is lower semicontinuous with respect to narrow convergence of measures. Thus for any fixed partition $0<t_1<\cdots<t_N$, we can find sufficiently large $k$ such that $SW(\mu^k(t_j),\mu(t_{j}))<\eps/N$ for all $j\leq N$ and thus
\[L_{SW}((\mu_t)_{t\in I})\leq \sum_{j\leq N-1}SW(\mu(t_{j}),\mu(t_{j+1})) \leq \sum_{j\leq N-1}SW(\mu^k(t_{j}),\mu^k(t_{j+1})) + 2\eps \leq L_{SW}((\mu^k_t)_{t\in I}) +2\eps.\]
Letting $\eps\searrow 0$ and $k\to \infty$, we see that indeed \grn $L_{SW}((\mu_t)_{t\in I})=\lim_{k\nearrow\infty} L_{SW}((\mu^k_t)_{t\in I})$.\nc
\end{proof}

From Lemma~\ref{lem:curve_cpct} we deduce the lower semiconitnuity of $\ell_{SW}$.
\begin{corollary}[lower semicontinuity of $\ell_{SW}$]\label{cor:lsw;lsc}
The map $(\mu,\nu)\mapsto \ell_{SW}(\mu,\nu)$ on $\P_2(\R^d)\times\P_2(\R^d)$ is lower semicontinuous with respect to the narrow convergence of measures.
\end{corollary}

\begin{proof}
Let $\mu^k,\nu^k$ be narrowly convergent sequences in $\P_2(\R^d)$ with respective limits $\mu,\nu$. Fix $\eps>0$, and
for each $k=1,2,\cdots$ let $(\mu_t^k)_{t\in I_k}$, $I_k=[0,\ell_{SW}(\mu^k,\nu^k)+\eps]$, be the arc-length parametrized curve~\cite[Lemma 1.1.4]{AGS} such that
\[L_{SW}((\mu_t^k)_{t\in I_k})\leq \ell_{SW}(\mu^k,\nu^k)+\eps.\]
By setting the $\mu_{t}^k=\nu^k$ for $t\geq \ell_{SW}(\mu^k,\nu^k)+\eps$, we can define all $(\mu^k_t)_t$ on a common bounded interval $I\supset I_k$. As $SW$ is lower semicontinuous, $SW(\mu^k,\mu)\xrightarrow[]{k\rightarrow\infty} 0$. Thus by Lemma~\ref{lem:curve_cpct}, there exists a limiting curve $(\mu_t)_{t\in I}\in AC(\P_2(\R^d);SW)$ such that
\[\mu^k_t\rightharpoonup \mu_t \text{ narrowly for all } t\in I.\]
Moreover, $(\mu_t)_t$ is a curve connecting $\mu$ and $\nu$, and by \eqref{eq:Lsw;lsc}
\[\ell_{SW}(\mu,\nu^0)\leq L_{SW}((\mu_t)_{t\in I})\leq \liminf_{k\rightarrow\infty} L_{SW}((\mu^k_t)_{t\in I}) = \liminf_{k\rightarrow\infty}\ell_{SW}(\mu^k,\nu^k)+\eps.\]
We conclude by letting $\eps\searrow 0$.
\end{proof}

Existence of geodesics in $(\P_2(\R^d),\ell_{SW})$ also follows from Lemma~\ref{lem:curve_cpct}.
\begin{proposition}[$\ell_{SW}$ is a geodesic metric]\label{prop:lsw-geodesic}
    For each $\mu,\nu\in\P_2(\R^d)$ there exists a length minimizing curve $(\mu_t)_{t\in[0,1]}$ such that $\ell_{SW}(\mu,\nu)=L_{SW}((\mu_t)_{t\in[0,1]})$. In particular, $(\P_2(\R^d),\ell_{SW})$ is a geodesic space.
\end{proposition}
\begin{proof}
Let $\mu,\nu\in\P_2(\R^d)$. \grn As $SW\leq \frac{1}{\sqrt{d}}W$, the length-minimizing sequence of curves can be chosen to have length $L_{SW}$ controlled unformly by $W(\mu,\nu)$. \nc Then Lemma~\ref{lem:curve_cpct} directly implies the existence of a length minimizing curve $(\mu_t)_{t\in[0,1]}$. We also note that by Theorem~\ref{thm:SW2_ac-curves} one can associate $(\mu_t,J_t)_{t\in[0,1]}\in\CE_{[0,1]}$ such that
$\ell_{SW}^2(\mu,\nu)=\int_0^1 B_{SW}(\mu_t,J_t)\,dt$.
\end{proof}

\begin{remark}\label{rmk:lsw-geodesic}
In case the geodesic $(\mu_t)_{t\in[0,1]}$ attains the sliced Wasserstein distance between $\mu,\nu\in\P_2(\R^d)$, the geodesic can be characterized as the Radon inverse of the 1D displacement interpolant between $\hat\mu^\theta$ and $\hat\nu^\theta$. However, in general such Radon inverse is not a probability measure, as noted in Example~\ref{ex:SW2_notgeo}.

While the $\ell_{SW}$-geodesic $(\mu_t)_{t\in[0,1]}$ remains in $\P_2(\R^d)$, we cannot guarantee that the corresponding pair $(\mu_t,J_t)_{\in[0,1]}\in\CE_{[0,1]}$ satisfies $\mu_t\ll J_t$, or even that $J_t$ is a measure for a.e. $t\in[0,1]$. See Remark~\ref{rmk:SW2_ac-curves}. However, it can be approximated by solution of the continuity equation with $J_t\ll\mu_t$ by concatenating $(\mu_t\ast\eta_\eps,J_t\ast\eta_\eps)_{t\in[0,1]}$ with the Wasserstein geodesics from $\mu$ to $\mu\ast\eta_\eps$ and $\nu\ast\eta_\eps$ to $\nu$, where $\eta_\eps$ is a suitable smooth convolution kernel with bandwidth $\eps\ll1$.
\end{remark}

\begin{remark}\label{rmk:nonexpansive}
We note that the  projection to closed balls is not a sliced Wasserstein contraction. This property holds for all transportation distances which are increasing with Euclidean distance, but we show by explicit example that it does not hold for the sliced Wasserstein distance: let
\[\grn \mu^\eps=\frac12\delta_{(0,-1)}+\frac12\delta_{(0,\sqrt{1+\eps})}\nc,\quad \nu^\eps=\frac12\delta_{(-\sqrt{\frac{\eps}{1+\eps}},\sqrt{\frac{1}{1+\eps}})}+\frac12\delta_{(\sqrt{\frac{\eps}{1+\eps}},\sqrt{\frac{1}{1+\eps}})}.\]
Then $\supp \nu^\eps \subset \overline B(0,1)\subset\R^2$. \purp Let $\pi^B:\R^2\rightarrow\overline B(0,1)$ be the projection onto $\overline B(0,1)$\nc. Then via explicit computation one can verify that
\[SW(\pi^B_\#\mu^\eps,\nu^\eps)> SW(\mu^\eps,\nu^\eps) \;\text{ for sufficiently small } \eps>0.\]

In light of this observation, the following question is nontrivial:
Consider $\mu,\nu\in\P_2(\R^d)$ supported in a closed unit ball $\overline B$ centered at 0. Does it hold that for all $t \in [0,1]$ the  measures along the geodesic are supported within the same ball? This property sounds natural as $\hat\mu^\theta,\hat\nu^\theta$ would be supported on $[-1,1]$ for all $\theta\in\S^{d-1}$, but remains an open problem.
\end{remark}

Recall that in Example~\ref{ex:SW2_notgeo} established that in general the $SW$ geodesics do not exist, whereas $\ell_{SW}$ geodesics between $\mu,\nu\in\P_2(\R^d)$ always exist by Proposition~\ref{prop:lsw-geodesic}. Thus we conclude this section with the following corollary.
\begin{corollary}\label{cor:sw2-neq-lsw}
    $(\P_2(\R^d),SW)$ is not a length space.
\end{corollary}

\section{Comparisons with negative Sobolev norms and the Wasserstein distance}\label{sec:comparison} 
    
We establish two comparison results, Theorem~\ref{thm:lsw-sw-H;comparison}, near absolutely continuous measures, and Theorem~\ref{thm:SWnear-discrete}, near discrete measures. The former states that for suitable absolutely continuous measures, $\ell_{SW}$ is equivalent to $SW$ and both are comparable to the $\dot H^{-(d+1)/2}(\R^d)$-norm as a consequence of the averaging effect of the Radon transform.
On the other hand, Theorem~\ref{thm:SWnear-discrete} states that, roughly speaking, $SW$ and $\ell_{SW}$ are very close to $\frac{1}{\sqrt{d}}W$ near discrete measures, as the smoothing effect due to averaging does not take place. 
    
For absolutely continuous measures with densities bounded below by $a$ and above by $b$, Peyre~\cite{RPeyre18} established the metric equivalence
\[\|\mu-\nu\|_{\dot H^{-1}(\R^d)}\lesssim_b W(\mu,\nu)\lesssim_a \|\mu-\nu\|_{\dot H^{-1}(\R^d)};\]
see~\cite[Proposition 2.8]{Loe06} for an earlier proof of the first inequality above. 
Our results can be seen as providing analogous comparisons between the SW distance and a norm in a Hilbert space. As $\ell_{SW}$ differs from SW, a question particular to our setup is whether the intrinsic distance also enjoys such comparison with $\dot H^{-(d+1)/2}(\R^d)$ norm. \blue We answer this affirmatively in Theorem~\ref{thm:lsw-sw-H;comparison}.\nc 

\vio Recall that a measure $\lambda\in\P(\R^d)$ is log-concave if for any $t\in(0,1)$ and Borel measurable sets $A,B\subset\R^d$
\begin{equation}
    \lambda((1-t)A + t B)\geq \lambda^{1-t}(A)\lambda^t(B).
\end{equation}
We first prove a useful lemma, which relies on that log-concavity is preserved by the pushforward with respect to projection $x\mapsto x\cdot\theta$~\cite{Bo74}, and that log-concave measures have log-concave density~\cite{Bo75}.
\nc
\blue
\begin{lemma}\label{lem:hatmut_ubd}
    Let $\lambda\in\P_2(\R^d)$ be a log-concave measure. Let $\mu,\nu\in\P_2(\R^d)$, and suppose there exists $b>0$ such that
    \begin{equation}\label{eq:hatmunu_ubd}
        \hat\mu^\theta,\hat\nu^\theta\leq b\hat\lambda^\theta\; \text{ for a.e. }\theta\in\S^{d-1}.
    \end{equation}
    Then for a.e. $\theta\in\S^{d-1}$ the displacement interpolation $(\hat\mu_t^\theta)_{t\in[0,1]}$ from $\hat\mu^\theta$ to $\hat\nu^\theta$ satisfies the same upper bound for all $t\in[0,1]$.
\end{lemma}

\begin{proof}
    As projection preserves log-concavity of a measure, each $\hat\lambda^\theta$ is log-concave, and thus $\hat\lambda^\theta$ has a log-concave density with respect to $\L^1$ unless it is a dirac mass~\cite[Theorem 3.2]{Bo75}. For $\theta\in\S^{d-1}$ such that $\hat\lambda^\theta$ is a dirac mass, the conclusion of this lemma is trivial, so we consider $\theta\in\S^{d-1}$ such that $\hat\lambda^\theta\ll\L^1$.
    It follows that $\hat\mu^\theta,\hat\nu^\theta\ll\L^1$, thus we can fix the optimal transport map $T^\theta:\R\rightarrow\R$ mapping $\hat\mu^\theta$ to $\hat\nu^\theta$, and define
    \[T^\theta_t(r)=(1-t)r+t T^\theta(r).\]    
    In the remainder of this proof, we identify the measures with their densities with respect to $\L^1$. 
    The displacement interpolation $\hat\mu_t^\theta$ is given by $\hat\mu_t^\theta=(T_t^\theta)_\#\hat\mu^\theta$. As $\hat\mu_t^\theta\ll (T_t^\theta)_\#\hat\lambda^\theta$, it suffices to show $\hat\mu_t^\theta(T_t^\theta(r))\leq b \hat\lambda^\theta(T_r^\theta(r))$ for all $r\in\R$. Arguing as in the proof of\cite[Proposition D.2]{McCann94},
    \begin{align*}
        \hat\mu_t^\theta(T_t^\theta(r)) = \frac{\hat\mu^\theta(r)}{\partial_r T^\theta_t(T^\theta(r))}
        =\frac{\hat\mu^\theta(r)}{1-t+t\frac{\hat\mu^\theta(r)}{\hat\nu^\theta(T^\theta(r))}}
        =\left(\frac{1-t}{\hat\mu^\theta(r)}+\frac{t}{\hat\nu^\theta(T^\theta(r))}\right)^{-1}
        \leq (\hat\mu^\theta(r))^{1-t}(\nu(T^\theta(r)))^t,
    \end{align*}
    where we have used the harmonic mean-geometric mean inequality. On the other hand, log-concavity of $\hat\lambda^\theta$ implies 
    \[\hat\mu_t^\theta(T_t^\theta(r))\leq(\hat\mu^\theta(r))^{1-t}(\nu(T^\theta(r)))^t\leq b(\hat\lambda^\theta(r))^{1-t}(\hat\lambda^\theta(T^\theta(r)))^{t}\leq \hat\lambda^\theta(T_t^\theta(r)).\]
\end{proof}

\begin{theorem}[Comparison between $\ell_{SW},SW$ and the $\dot H^{-(d+1)/2}$-norm.]\label{thm:lsw-sw-H;comparison}
    Let $\mu,\nu,\lambda\in\P_2(\R^d)$, and let $0<a\leq b<\infty$ such that
\begin{equation}\label{cond:mu-nu;comparison_theta}
    a\hat\lambda^\theta\leq \hat\mu^\theta \leq b \hat\lambda^\theta \; \text{ and } \;\hat\nu^\theta\leq b\hat\lambda^\theta\;\; \text{ for a.e. }\theta\in\S^{d-1}.
\end{equation}
    Then we have the following.
    \begin{listi}
        \item  If $\lambda$ is log-concave  then
        \begin{equation}\label{eq:lsw-sw;direct}
            \ell_{SW}(\mu,\nu)\leq 2\sqrt{\frac{b}{a}}SW(\mu,\nu).
        \end{equation}
        
        \item If $\lambda\ll\L^d$ with
        $C_\lambda:=\esssup_{r\in\R,\theta\in\S^{d-1}}\frac{d\hat\lambda^\theta}{d\L^1}(r)<\infty$, then
        \begin{equation}\label{eq:Hneg-sw}
            \sqrt{\frac{1}{bC_\lambda}}\|\mu-\nu\|_{\dot H^{-(d+1)/2}(\R^d)} \leq  SW(\mu,\nu).
        \end{equation}
        Suppose further $\lambda=\lambda_\Omega:=\frac{1}{|\Omega|}\L^d\restr_\Omega$ for an open connected bounded $\Omega\subset\R^d$. Furthermore, let $\mu=\nu$ on $\Omega\setminus\tilde\Omega$ for some $\tilde\Omega\subset\subset \Omega$. Then there exists $C=C(d,\Omega,\dist(\tilde\Omega,\partial\Omega))$ such that
        \begin{equation}\label{eq:lsw-Hneg}
            \sqrt{\frac{1}{bC_\lambda}}\|\mu-\nu\|_{\dot H^{-(d+1)/2}(\R^d)} \leq  SW(\mu,\nu)\leq \ell_{SW}(\mu,\nu)\leq \frac{C}{\sqrt{a}}\|\mu-\nu\|_{\dot H^{-(d+1)/2}(\R^d)}.
        \end{equation}
    \end{listi}
    In particular, if $\Omega$ in (ii) is also convex, then
    \begin{equation}\label{eq:lsw-sw-Hneg;comparison}
        \begin{split}
            \sqrt{\frac{1}{bC_\lambda}}\|\mu-\nu\|_{\dot H^{-(d+1)/2}(\R^d)} &\leq  SW(\mu,\nu)\\
            &\leq \ell_{SW}(\mu,\nu)\leq \min\left\{2\sqrt{\frac{b}{a}}SW(\mu,\nu), \frac{C}{\sqrt{a}}\|\mu-\nu\|_{\dot H^{-(d+1)/2}}\right\}.
        \end{split}
        \end{equation}    
\end{theorem}
\nc

\begin{remark}\label{rmk:lsw-sw;comparison1}
We leave a few remarks on the conditions of Theorem~\ref{thm:lsw-sw-H;comparison}.
A simple, and useful, condition that implies \eqref{cond:mu-nu;comparison_theta} is the following:
\begin{equation}\label{cond:mu-nu;comparison}
    a\lambda\leq\mu\leq b\lambda \;\text{ and }\; \nu\leq b\lambda.
\end{equation}
We note that  \eqref{cond:mu-nu;comparison_theta} only requires the comparison to hold after integrating over hyperplanes. 

\blue
The condition $C_\lambda<\infty$ is satisfied whenever $\lambda\ll\L^d$ is compactly supported and has bounded density. Indeed, denoting by $B_k(0,M)$ the ball of radius $M$ in $\R^k$ centered at 0, if $\supp\lambda\subset B_d(0,M)$, 
\[\frac{d\hat\lambda^\theta}{d\L^1}(r)=\int_{r+\theta^\perp}\frac{d\lambda}{d\L^d}\,d\L^{d-1}\leq \norm{\frac{d\lambda}{d\L^d}}_{\infty} \L^{d-1}(B_{d-1}(0,M)),\]
thus $C_\lambda<\infty$ whenever $\|d\lambda/d\L^d\|_{\infty}<\infty$.
However, $\lambda$ need not be compactly supported for $C_\lambda<\infty$; for instance, consider a Gaussian measure on $\R^d$.
\nc

\grn Observe also that the second part of (ii) requires connectedness but not convexity of $\Omega$, whereas the comparison $\dot H^{-1}$ with $W$ on $\R^d$ requires $\Omega$ to be convex~\cite{RPeyre18}. This is because we only use displacement interpolation between 1D projections, and connected and convex sets coincide in $\R$. In fact, our proof only requires connectedness for each projection $\Omega^\theta$ defined in~\eqref{def:Omega_theta}. However, to keep the statement simpler we use a stronger assumption that $\Omega$ is connected. \nc
\end{remark}
  
\blue
\begin{proof}
Our proof is a careful adaptation of the argument by Peyre~\cite{RPeyre18} to the sliced Wasserstein setting. The main difficulty comes from the fact that the density of the projections $\hat\mu^\theta,\hat\nu^\theta$ with respect to $\L^1$ is not bounded away from zero, near the edge of their supports.

In this proof, we will use $\Omega^\theta\subset\R$ to denote the projection of $\Omega\subset\R^d$ in the direction $\theta\in\S^{d-1}$, namely
\begin{equation}\label{def:Omega_theta}
        \Omega^\theta=\{r\in\R:\;\exists x\in\Omega \text{ s.t. } r=x\cdot\theta\}.
\end{equation}
Defining $\Omega$ to be the interior of $\supp\lambda$ in the case (i), observe that in both cases (i) and (ii) $\Omega^\theta\subset\R$ is connected for each $\theta\in\S^{d-1}$, hence convex; in particular, the displacement interpolation between $\hat\mu^\theta$ and $\hat\nu^\theta$ remains in $\Omega^\theta$.

Noting that $\L^d\restr_\Omega$ is log-concave whenever $\Omega$ is convex, \eqref{eq:lsw-sw-Hneg;comparison} follows directly from \eqref{eq:lsw-sw;direct} and \eqref{eq:lsw-Hneg}. Thus it suffices to prove items (i) and (ii).

\vspace{3mm}
\noindent\emph{Step 1$^o$} In this step we show that when $\lambda$ is log-concave, the condition \eqref{cond:mu-nu;comparison_theta} implies the upper bound
\[\ell_{SW}(\mu,\nu)\leq 2\sqrt{\frac{b}{a}}SW(\mu,\nu).\] 
    
We do this using comparison of $W$ distances of the projections along each $\theta$ with corresponding weighted $\dot H^{-1}$ norms. Consider the linear interpolation $\tilde\mu_t = (1-t)\mu+t\nu$ and write $\tilde\mu_t^\theta=R_\theta\tilde\mu_t=(1-t)\hat\mu^\theta+t\hat\nu^\theta$. Then $\tilde\mu_t^\theta \geq (1-t)\hat\mu^\theta$ and thus by duality $\|\cdot\|_{\dot H^{-1}(\tilde\mu_t^\theta)} \leq (1-t)^{-1/2}\|\cdot\|_{\dot H^{-1}(\hat\mu^\theta)}$. Hence, using Benamou-Brenier formula for each projection, we have
\begin{equation}\label{eq:lsw-weightedHneg}
\ell_{SW}(\mu,\nu)\leq \dashint_{\S^{d-1}} \int_0^1 \|\hat\mu^\theta-\hat\nu^\theta\|_{\dot H^{-1}(\tilde\mu_t^\theta)}\,dt\,d\theta \leq  2\dashint_{\S^{d-1}} \|\hat\mu^\theta-\hat\nu^\theta\|_{\dot H^{-1}(\hat\mu^\theta)}\,d\theta.
\end{equation}
Note that \eqref{eq:lsw-weightedHneg} required no assumption on $\mu,\nu$. For each $\theta\in\S^{d-1}$ let $(\hat\mu_t^\theta)_{t\in[0,1]}$ be a constant speed $W$-geodesic from $\hat\mu^\theta$ to $\hat\nu^\theta$. By Lemma~\ref{lem:hatmut_ubd} $\hat\mu_t^\theta\leq b\hat\lambda^\theta$ for all $t\in[0,1]$. As $a\hat\lambda^\theta \leq \hat\mu^\theta$, we have
\[\|\cdot\|_{\dot H^1(\hat\mu_t^\theta)} \leq \sqrt{b}\|\cdot \|_{\dot H^1(\hat\lambda^\theta)} \leq \sqrt{\frac{b}{a}} \,\|\cdot\|_{\dot H^1(\hat\mu^\theta)},\]
and thus by duality
\[\|\cdot \|_{\dot H^{-1}(\hat\mu^\theta)}\leq \sqrt{\frac{b}{a}} \,\|\cdot\|_{\dot H^{-1}(\hat\mu_t^\theta)}.\]
As $\hat\mu_t^\theta$ is a constant speed geodesic, $\|\partial_t\hat\mu_t^\theta\|_{\dot H^{-1}(\hat\mu_t^\theta)}=W(\hat\mu^\theta,\hat\nu^\theta)$ and thus
\begin{align*}
    \|\hat\mu^\theta-\hat\nu^\theta\|_{\dot H^{-1}(\hat\mu^\theta)}
    \leq \int_0^1 \|\partial_t\hat\mu_t^\theta\|_{\dot H^{-1}(\hat\mu^\theta)}\,dt
    \leq \sqrt{\frac{b}{a}} \int_0^1 \|\partial_t\hat\mu_t^\theta\|_{\dot H^{-1}(\hat\mu_t^\theta)}\,dt
    =\sqrt{\frac{b}{a}} W(\hat\mu^\theta,\hat\nu^\theta).
\end{align*}
Hence
\begin{align*}
\ell_{SW}^2(\mu,\nu) \leq 4\dashint_{\S^{d-1}} \|\hat\mu^\theta-\hat\nu^\theta\|_{\dot H^{-1}(\hat\mu^\theta)}^2\,d\theta
\leq \frac{4b}{a}\dashint_{\S^{d-1}} W^2(\hat\mu^\theta,\hat\nu^\theta)\,d\theta = \frac{4b}{a} SW^2(\mu,\nu).
\end{align*}

\vspace{3mm}

\noindent\emph{Step 2$^o$} In this step we establish the lower bound
\[\sqrt{\frac{1}{bC_\lambda}}\|\mu-\nu\|_{\dot H^{-(d+1)/2}(\R^d)}\leq SW(\mu,\nu).\]
under the assumption that $\lambda\ll\L^d$ and $C_\lambda<\infty$. By construction
\[\hat\mu^\theta,\hat\nu^\theta \leq b\hat\lambda^\theta \leq bC_\lambda \L^1 \text{ for a.e. } \theta\in\S^{d-1}.\]
As $\L^1$ is log-concave, by Lemma~\ref{lem:hatmut_ubd} the displacement interpolation $\hat\mu_t^\theta$  between $\hat\mu^\theta$ and $\hat\nu^\theta$ satisfies $\hat\mu^\theta_t\leq b C_\lambda \L^1$ for all $t\in[0,1]$ and a.e. $\theta\in\S^{d-1}$. Arguing as in\cite[Theorem 5]{RPeyre18}, we have
\begin{align*}
    \|\hat\mu^\theta-\hat\nu^\theta\|_{\dot H^{-1}(\R)}
    \leq \int_0^1 \|\partial_t\hat\mu_t\|_{\dot H^{-1}(\R)}
    \leq \sqrt{b C_\lambda}\int_0^1 \|\partial_t\hat\mu_t\|_{\dot H^{-1}(\hat\mu_t^\theta)}
    = \sqrt{b C_\lambda} W(\hat\mu^\theta,\hat\nu^\theta). 
\end{align*}
By averaging over $\theta\in\S^{d-1}$ we obtain
\[\frac{1}{bC_\lambda}\|\hat\mu-\hat\nu\|_{\dot H^{-1}(\Pd)}^2\leq SW^2(\mu,\nu).\]
Moreover, identifying the measures $\mu,\nu$ and their densities, we have $\mu-\nu\in L^1(\R^d)$. Thus by the Fourier slicing theorem (Proposition~\ref{prop:radon-fourier}) and change of variables $\xi=\theta\zeta$, we have
\begin{align*}
    \|\hat\mu-\hat\nu\|_{\dot H^{-1}(\Pd)}^2 &= \frac{1}{2(2\pi)^{d-1}}\int_{\S^{d-1}}\int_{\R} |\zeta|^{-2} |(\F_1 R_\theta(\mu-\nu))(\zeta)|^2\,d\zeta\,d\theta  \\
    &= \frac{1}{2(2\pi)^{d-1}}\int_{\S^{d-1}}\int_{\R} |\zeta|^{-2} |\F_d(\mu-\nu)(\theta\zeta)|^2\,d\zeta\,d\theta  \\
    &= \int_{\R^d} |\xi|^{-(d+1)} |\F_d(\mu-\nu)(\xi)|^2\,d\xi = \|\mu-\nu\|_{\dot H^{-(d+1)/2}(\R^d)}^2.
\end{align*}

\vspace{3mm}
\noindent\emph{Step 3$^o$} We show \eqref{eq:lsw-Hneg} under the additional assumption that $\lambda=\lambda_\Omega:=\frac{1}{|\Omega|}\L^d\restr_\Omega$ for some bounded connected $\Omega\subset\R^d$ and $\mu=\nu$ in $\Omega\setminus\tilde\Omega$. Let $\alpha:=\dist(\tilde\Omega,\partial\Omega)$. Then
\[\hat\mu^\theta=\hat\nu^\theta \text{ on } (\inf_{r\in\Omega^\theta} r,\inf_{r\in\Omega^\theta} r+\alpha)\cup (\sup_{r\in\Omega^\theta}r-\alpha,\sup_{r\in\Omega^\theta}r) \text{ for each } \theta\in\S^{d-1}.\]
Recall $\frac{d\hat\mu^\theta}{d\L^1}=\frac{d\hat\mu^\theta}{d\hat\lambda^\theta_\Omega}\frac{d\hat\lambda^\theta_\Omega}{d\L^1}\geq a \frac{d\hat\lambda^\theta_\Omega}{d\L^1}$. Furthermore, as $\Omega$ is open and connected, $\frac{d\hat\lambda^\theta_\Omega}{d\L^1}>0$ on the interval $(\inf_{r\in\Omega^\theta} r, \sup_{r\in\Omega^\theta} r)$. Thus there exists some constant $C$ depending only on $\Omega,\alpha,d$ such that
\[\hat\mu^\theta \geq \frac{1}{a}\hat\lambda_\Omega^\theta \geq \frac{C}{a}\L^1 \text{ on } [\inf_{r\in\Omega^\theta}r+\alpha,\sup_{r\in\Omega^\theta}r-\alpha].\]
Combining this with \eqref{eq:lsw-weightedHneg}, we can find some $C=C(\Omega,\alpha,d)>0$ such that
\begin{align*}
    \ell_{SW}(\mu,\nu)\leq 2\dashint_{\S^{d-1}} \|\hat\mu^\theta-\hat\nu^\theta\|_{\dot H^{-1}(\hat\mu^\theta)}\,d\theta
    \leq \frac{C}{\sqrt{a}}\dashint_{\S^{d-1}} \|\hat\mu^\theta-\hat\nu^\theta\|_{\dot H^{-1}(\R)}\,d\theta,
\end{align*}
By Jensen's inequality and the Radon isometry argument as in Step 2$^\circ$, we deduce
\begin{align*}
    \ell_{SW}^2(\mu,\nu) \leq \frac{C^2}{a}\|\hat\mu-\hat\nu\|_{\dot H^{-1}(\Pd)}^2 = \frac{C^2}{a} \|\mu-\nu\|_{\dot H^{-(d+1)/2}(\R^d)}.
\end{align*}
\end{proof}
\nc

\begin{remark}\label{rmk:lsw-sw;comparison}
We make a few further remarks. Firstly, the lower bound 
\[\|\mu-\nu\|_{\dot H^{-(d+1)/2}(\R^d)} \lesssim  SW(\mu,\nu)\]
does not hold for general measures. In fact, while the Sobolev embedding Theorem implies $\delta_0\in H^{-(d+1)/2}(\R^d)$, the same is not true for $\dot H^{-(d+1)/2}(\R^d)$. Recalling that $\F_d\delta_0$ is a constant, we see that for any $f\in\Sch(\R^d)$
\begin{align*}
    f(0)=\int_{\R^d} f(x)\,d\delta_0(x) = C\int_{\R^d} (\F_d f)(\xi)\,d\xi.
\end{align*}
By \eqref{eq:Hts;duality}, if $\delta_0\in \dot H^{-(d+1)/2}(\R^d)$ then the right-hand side must be controlled by $\| f\|_{\dot H^{(d+1)/2}(\R^d)}$, which is clearly not true; for instance, consider increasingly concentrated Gaussians centered at zero.
Similarly, $\delta_x-\delta_y\not\in \dot H^{-(d+1)/2}(\R^d)$ in general for $x,y\in\R^d$, whereas $SW(\delta_x,\delta_y)\leq \frac{1}{\sqrt{d}}|x-y|$. 
   
\vio We further note that the upper bound \eqref{eq:lsw-Hneg} requires the additional condition that $\mu=\nu$ near the boundary $\partial\Omega$. Indeed, letting $\lambda=c\L^d\restr_\Omega\in\P_2(\R^d)$ for some bounded domain $\Omega\subset\R^d$, the density of $\hat\lambda^\theta$ is not bounded away from zero; this makes it difficult to control $\|\mu-\nu\|_{\dot H^{-1}(\hat\lambda^\theta)}$ with $\|\mu-\nu\|_{\dot H^{-1}(\R)}$. Indeed, denoting by $F_\sigma$ the cumulative distribution function (CDF) of $\sigma\in\P(\R)$, $a\lambda\leq \mu,\nu\leq b\lambda$ does not in general guarantee
\[\|F_{\hat\mu^\theta}-F_{\hat\nu^\theta}\|_{L^2((\hat\lambda^\theta)^{-1})}=\|\hat\mu^\theta-\hat\nu^\theta\|_{\dot H^{-1}(\hat\lambda^\theta)}\lesssim_{\lambda,a,b} \|\hat\mu^\theta-\hat\nu^\theta\|_{\dot H^{-1}(\R)}=\|F_{\hat\mu^\theta}-F_{\hat\nu^\theta}\|_{L^2(\R)}.\]
\vio To see this, let $d\geq 2$ and $\Omega=B(0,1)$. As $\lambda$ is radially symmetric, without loss of generality we restrict our attention to the projection onto the $e^1$ direction. Then $\hat\lambda^{e_1}(r)=c (r+1)^{(d-1)/2}$ for $-1\leq r\leq 0$, and is symmetric about $0$. Consider one-dimensional measures $\mu_h,\nu_h\in\P_2(\R)$ such that their CDFs satisfy $F_{\mu_h}=F_{\nu_h}$ on $[-1+2h,1]$ while $\mu_h=b\hat\lambda^{e_1}$, $\nu_h=a\hat\lambda^{e_1}$ on $[-1,-1+h]$; note this is possible by prescribing suitable behavior on the interval $(-1+h,-1+2h)$. Then we have $F_{\mu_h}-F_{\nu_h}=(b-a)F_{\hat\lambda^{e_1}}$ on $[-1,-1+h]$. From direct calculations one can check that 
\[\frac{\|\mu_h-\nu_h\|_{\dot H^{-1}(\hat\lambda^{e_1})}^2}{\|\mu_h-\nu_h\|_{\dot H^{-1}(\R)}^2} \gtrsim_d h^{-\frac{d-1}{2}}\xrightarrow[]{h\searrow 0} \infty.\]  
As $\lambda$ is radially symmetric, we can come up with examples of measures in $\P_2(\R^d)$ such that their projections satisfy similar estimates. \nc
\end{remark}

\grn We now study the behavior of $SW$ around discrete measures. The $\infty$-Wasserstein distance $W_\infty$ is defined by
\begin{equation}\label{def:W_infty}
    W_\infty(\mu,\nu):= \inf_{\gamma\in\Gamma(\mu,\nu)}\gamma-\esssup_{(x,y)\in\supp\gamma} |x-y|.
\end{equation}\nc
We have seen that for any $x\in\R^d$ $SW(\nu,\delta_x)=\frac{1}{\sqrt{d}}W(\nu,\delta_x)$. Similarly, if $\mu$ is a discrete measure with support $\{x_i\}_{i=1,\cdots,n}$ and $W_\infty(\mu,\nu)$ is sufficiently small, any optimal transport map should map all the mass of $\nu$ near $x_i\in\supp\mu$ to $x_i$. Moreover, for most directions $\theta$ the same is true at the level of projections as well. This allows us to show that within $W_\infty$-balls of a discrete measure, $SW$ metric can be well approximated by $\frac{1}{\sqrt{d}}W$. 
\begin{theorem}\label{thm:SWnear-discrete}
Assume $\mu$ is a discrete probability measure: $\mu = \sum_{i=1}^n m_i \delta_{y_i}$ where all masses are positive and all points are distinct. Let $l_\mu=\min_{i\neq j}|y_i-y_j|$. Then there exists $C\geq 1$ only dependent on $d$ such that if \purp$W_\infty(\mu,\nu)< \frac{l_\mu}{4Cn}$\nc, we have
\begin{equation}\label{eq:W2=SW;discrete}
    0 \leq \frac1d W^2(\mu,\nu)-SW^2(\mu,\nu)\leq \frac{4Cn}{l_\mu} W_\infty(\mu,\nu) SW^2(\mu,\nu).
\end{equation}
Thus, we have the comparison
\begin{equation}\label{eq:sw_comparison;discrete}
    SW^2(\mu,\nu)\leq \ell_{SW}^2(\mu,\nu)\leq \frac{1}{d}W^2(\mu,\nu)\leq (1+4Cn l_\mu^{-1}W_\infty(\mu,\nu))SW^2(\mu,\nu).
\end{equation}
\end{theorem}
\begin{proof}
\grn Let $\delta< \frac{l_\mu}{2}$. We claim that \nc we can find $C=C(d)\geq 1$ such that
\begin{equation}\label{eq:sw_comparison;claim}
    \frac1dW^2(\mu,\nu)-SW^2(\mu,\nu)\leq \frac{2Cn\delta}{l_\mu} W^2(\mu,\nu) \text{ for } W_\infty(\mu,\nu)\leq \delta.
\end{equation}
\grn The desired result \eqref{eq:W2=SW;discrete} follows from the claim \eqref{eq:sw_comparison;claim}. Indeed, setting $\delta=W_\infty(\mu,\nu)$ whenever $W_\infty(\mu,\nu)<\frac{l_\mu}{4Cn}$, as $\eps=\frac{2Cn\delta}{l_\mu}<\frac12$ implies $(1-\eps)^{-1}\leq 1+2\eps$.

Thus it remains to prove \eqref{eq:sw_comparison;claim}\nc. To this end, let $\gamma^\infty\in\Gamma^\infty_o(\nu,\mu)$ be the $\infty$-transport plan. As $\delta<\frac{l_\mu}{2}$ we know that $\gamma^\infty=(\id\times T^\infty)_\# \nu$ for some transport map $T^\infty$, which is also the optimal transport map for the quadratic cost, and satisfies
\[\|T^\infty-\id\|_{L^\infty(\nu)}\leq \delta.\]
For each $\theta\in\S^{d-1}$, let $\gamma^\theta\in\Gamma^2_o(\hat\nu^\theta,\hat\mu^\theta)$. For each $x\in\supp\nu$, define the set of angles $A_x$ where the $\gamma^\theta$ differs from the 1D-coupling induced by $T^\infty$ -- i.e.
\[A_x:=\{\theta\in\S^{d-1}:\;\;\exists (x\cdot\theta,y^\theta)\in\supp\gamma^\theta \text{ s.t. }|y^\theta-x\cdot\theta|<|\theta\cdot(T^\infty(x)-x)|\}.\]
\grn To control $\frac1d W^2(\mu,\nu)-SW^2(\mu,\nu)$, it suffices to control the size of $A_x$, as
\begin{align*}
    \frac1d W^2(\mu,\nu)&-SW^2(\mu,\nu)\leq \frac1d \|T^\infty-\id\|_{L^2(\nu)}^2-SW^2(\mu,\nu) \\
    &= \int_{\R^d}\dashint_{\S^{d-1}}  |\theta\cdot (T^\infty(x)-x)|^2\,d\theta\,d\nu(x)-\int_{\R^d}\dashint_{\S^{d-1}} |y^\theta-x^\theta|\,d\gamma^\theta(x^\theta,y^\theta)\,d\theta  \\
    &\leq \int_{\R^d}\frac{1}{|\S^{d-1}|}\int_{A_x}|\theta\cdot(T^\infty(x)-x)|^2\,d\theta\,d\nu(x).
\end{align*}
\nc
As $\|T_\infty-\id\|_{L^\infty(\nu)}\leq\delta$,
\begin{align*}
    A_x &\subset \bigcup_{i=1}^n \{\theta\in\S^{d-1}: |y_i\cdot\theta-x\cdot\theta|<|\theta\cdot(T^\infty(x)-x)|\}    \\
    &\subset\bigcup_{\substack{y_i\in\supp\mu\\ y_i\neq T^\infty(x)}} \{\theta\in\S^{d-1}:\; |y_i\cdot\theta-x\cdot\theta|<\delta\} 
    \subset \bigcup_{\substack{y_i\in\supp\mu\\ y_i\neq T^\infty(x)}} \left\{\theta\in\S^{d-1}:\; \frac{l_\mu}{2}\theta\cdot \frac{y_i-x}{|y_i-x|}<\delta \right\},
\end{align*}
where we have used that for $y_i\neq T^\infty(x)$
\[|y_i-x|\geq |T^\infty(x)-y_i|-|T^\infty(x)-x|\geq l_\mu-\delta > \frac{l_\mu}{2}.\]
Thus by Chebyshev's inequality
\[\frac{|A_x|}{|\S^{d-1}|}\leq \frac{2Cn\delta}{l_\mu}\]
for some $C=C(d)$. As $T^\infty$ is also the optimal transport map for the quadratic cost,\grn
\begin{align*}
    \frac1d W^2(\mu,\nu)-SW^2(\mu,\nu) &\leq \int_{\R^d} \frac{1}{|\S^{d-1}|}\int_{A_x} |\theta\cdot(T^\infty(x)-x)|^2\,d\theta\,d\nu(x)     \\
    &\leq   \frac{1}{|\S^{d-1}|}\int_{A_x}\int_{\R^d}  |T^\infty(x)-x|^2 \,d\nu(x)\,d\theta \leq \frac{2Cn\delta}{l_\mu} W^2(\mu,\nu).
\end{align*}
As this is precisely our claim \eqref{eq:sw_comparison;claim}, we conclude the proof.
\end{proof}

\section{Statistical properties of the sliced Wasserstein length}\label{sec:lsw;stat}

In this section we investigate the approximation error in $\ell_{SW}$ distance between absolutely continuous  measures $\mu$ and the empirical measure of their i.i.d. samples, $\mu^n=\frac1n\sum_{i=1}^n \delta_{X_i}$ with $X_i\overset{i.i.d.}{\sim}\mu$. The parametric rate of estimation for $SW$ has already been observed, for instance by  Manole, Balakrishnan, and Wasserman~\cite[Proposition 4]{ManBalWas22} in the form $\Ex SW(\mu^n,\mu)\lesssim n^{-1/2}$.

The main result of this section is Theorem~\ref{thm:lsw;para_rate} which shows that the corresponding concentration result holds for the $\ell_{SW}$ distance, namely that 
\[\ell_{SW}(\mu,\mu^n)\lesssim\sqrt{\frac{\log n}{n}}\; \text{ with high probability. }\]
Note that this directly implies $SW(\mu,\mu^n)\lesssim \sqrt{\log n/n}$ with high probability, which is also new, to the best of our knowledge. \nc
We note that while proving $SW(\mu,\mu_n)\lesssim\sqrt{\log n/n}$ only requires showing the estimation of one-dimensional Wasserstein distances holds in an integrated form over all projections to lines, showing estimates for $\ell_{SW}(\mu,\mu_n)$ requires constructing curves of length at most $\sqrt{\log n/n}$ connecting $\mu$ and $\mu_n$. 
We first provide a geometric intuition as to why this is to be expected. If $\partial_t\mu_t+\nabla\cdot(v_t\mu_t)=0$, then
\[
    \|\partial_t\mu_t\|_{\dot H^{-1}(\mu_t)}=\sup_{\|\varphi\|_{\dot H^{1}(\mu_t)}\leq 1}\int_{\R^d}\varphi\cdot\nabla\cdot(v_t\mu_t)\,dx 
    = \sup_{\|\varphi\|_{\dot H^{1}(\mu_t)}\leq 1} \int_{\R^d} v_t\cdot\nabla\varphi\,d\mu_t = \|v_t\|_{L^2(\mu_t)}=|\mu'|_{W}(t),
\]
where we obtain the last equality by choosing $\varphi$ such that $\nabla\varphi= v_t/\|v_t\|_{L^2(\mu_t)}$. Thus
\[|\mu'|_{SW}^2(t)=\dashint_{\S^{d-1}} |(\hat\mu^\theta)'|_{W}^2(t)\,d\theta =\|\partial_t \hat\mu_t\|_{\dot H^{-1}(\hat\mu_t)}^2.\]
From \eqref{eq:Radon_isometry;Hts} we know that the Radon transform is an isometry from $\dot H^{-(d+1)/2}(\R^d)$ to $\dot H^{-1}(\Pd)$.  
We note that the related Sobolev space $H^{(d+1)/2}(\R^d)$ is a  Reproducing Kernel Hilbert Space (RKHS). Heuristically we can view $(\P_2(\R^d),\ell_{SW})$ as having an RKHS as a dual at each point $\mu\in\P_2(\R^d)$. It is important to note that dual metrics of RKHS norms -- also known as Maximum Mean Discrepancy (MMD) -- can be approximated at parametric rate~\cite{Sri16}.  Thus it is reasonable  that the same holds for  the nonlinear analogue, $\ell_{SW}$.

\blue We will see that, under suitable assumptions, considering linear interpolation between $\mu,\mu^n$ is sufficient to establish the parametric rate of estimation in $\ell_{SW}$. Recall that in  \eqref{eq:lsw-weightedHneg} we established 
\[\ell_{SW}(\mu,\nu)\leq 2\dashint_{\S^{d-1}} \|\hat\mu^\theta-\hat\nu^\theta\|_{\dot H^{-1}(\hat\mu^\theta)}\,d\theta.\] 
Take $\nu=\mu^n$ and suppose $\hat\mu^\theta\ll\L^1$ for a.e. $\theta\in\S^{d-1}$. Write $d\hat\mu^\theta=f^\theta\,d\L^1$ and let $F^\theta$ and $F_n^\theta$ denote the cumulative distribution functions (CDFs) of $R_\theta\mu$ and $R_\theta\mu^n$,  respectively. Then, for each test function $\varphi\in\Sch(\R)$ we have
\begin{align*}
    -\int_{\R}\varphi\,d(R_\theta(\mu-\mu^n))= \int_{\R} (F^\theta(r)-F_n^\theta(r))\varphi'(r)\,dr \leq \|\varphi\|_{\dot H^1(\hat\mu^\theta)} \left(\int_{\R} \frac{|F^\theta(r)-F_n^\theta(r)|^2}{f^\theta(r)}\,dr\right)^{1/2}
\end{align*}
as the (weak) derivative of $F^\theta-F^\theta_n$ is $R_\theta\mu-R_\theta\mu^n$, and the boundary term from integration by parts vanishes as $\lim_{|r|\to \infty}F^\theta(r)-F^\theta_n(r) =0$. Thus from \eqref{def:weightedH1neg;1D} we conclude
\[\ell_{SW}^2(\mu,\mu^n)\leq 4\dashint_{\S^{d-1}} \|R_\theta \mu-R_\theta\mu^n\|_{\dot H^{-1}(\hat\mu^\theta)}^2\,d\theta \leq 4\dashint_{\S^{d-1}} \int_{\R}\frac{|F^\theta(r)-F_n^\theta(r)|^2}{f^\theta(r)}\,dr\,d\theta.\]

Thus the key is to uniformly bound $|F^\theta-F_n^\theta|^2$ relative to $f^\theta$, which can decay rapidly near the boundary. This can be done using the relative VC-inequality due to Vapnik and Chervonenkis~\cite[Theorem 1]{VapChe94} (see also~\cite[Chapter 3]{Vapnik13}). We state below the version of the relative VC inequality that can be found in~\cite[Theorem 2.1]{AntTay93} and~\cite[Exercise 3.3]{DevLug01}. \vio The theorem provides an upper bound in terms of the shattering number (also known as the growth function or the shattering coefficient) of a class of sets, which quantifies richness or complexity of the class; we refer the readers to~\cite[Section 2.7]{Vapnik13} for a precise definition.\nc

\begin{theorem}[Vapnik and Chervonenkis, Theorem 1 of~\cite{VapChe94}]\label{thm:VC_relative}
    Let $\mu\in\P(\R^d)$ and $\mu^n=\frac1n\sum_{i=1}^n \delta_{X_i}$ be the empirical measure of i.i.d samples $X_i\overset{i.i.d.}{\sim}\mu$. For each class $\A$ of measurable subsets of $\R^d$, and let $S_\A(k)$ be its shattering number for $k$ points. Then 
    \begin{equation}\label{eq:VC_relative}
        \prob\left(\sup_{A\in\A}\frac{|\mu(A)-\mu_n(A)|}{\sqrt{\mu(A)}}\geq \eps\right)\leq 4S_{\A}(2n)\exp\left(-\frac{n\eps^2}{4}\right).
    \end{equation}
\end{theorem}

By considering $\A$ to be the collection of half-spaces, we can deduce the following uniform concentration result of the empirical CDFs $\,F_n^\theta$.
\begin{corollary}\label{cor:VC_relative_CDF}
    Let $\mu,\mu^n$ be as in Theorem~\ref{thm:VC_relative}, and for each $\theta\in\S^{d-1}$ let $F^\theta,F_n^\theta$ be the respective cumulative distribution functions of $R_\theta\mu,\,R_\theta\mu^n\in\P(\R)$. Then
    \begin{equation}\label{eq:VC_relative_CDF}
        \prob\left(\sup_{r\in{\R},\theta\in\S^{d-1}}\frac{|F^\theta(r)-F_n^\theta(r)|}{\sqrt{F^\theta(r)(1-F^\theta(r))}}\geq \eps\right)\leq 8(2n+1)^{d+1}\exp\left(-\frac{n\eps^2}{16}\right).
    \end{equation}
\end{corollary}

\begin{proof}
    Take $\A$ in Theorem~\ref{thm:VC_relative} to be the collection of all half spaces in $\R^d$. The VC-dimension of half spaces in $\R^d$ is $d+1$. Thus, by the Sauer-Shelah lemma~\cite{Shelah72,Sauer72} (see also ~\cite[Corollary 1.4]{DevLug01}), we have $S_{\A}(2n)\leq (2n+1)^{d+1}$. As
    \[\mu(\{x\in\R^d:\;x\cdot\theta\leq r\})=F^\theta(r) \;\text{ and }\; \mu(\{x\in\R^d:\;x\cdot\theta> r\})=1-F^\theta(r),\]
    we obtain
    \begin{align*}
        \prob\left(\sup_{r\in{\R},\theta\in\S^{d-1}}\frac{|F^\theta(r)-F_n^\theta(r)|}{\sqrt{F^\theta(r)}}\geq \eps\right)\leq 4(2n+1)^{d+1}\exp\left(-\frac{n\eps^2}{4}\right)
    \end{align*}
    and
    \begin{align*}
        \prob\left(\sup_{r\in{\R},\theta\in\S^{d-1}}\frac{|F^\theta(r)-F_n^\theta(r)|}{\sqrt{1-F_\theta(r)}}\geq \eps\right)\leq 4(2n+1)^{d+1}\exp\left(-\frac{n\eps^2}{16}\right).
    \end{align*}
    We deduce \eqref{eq:VC_relative_CDF} using that $s(1-s)\geq \frac{1}{2}\min\{s,1-s\}$ for $s\in[0,1]$, as noted in~\cite[Example 2]{ManBalWas22}.
\end{proof}
    
\blue
We establish the parametric rate of $\ell_{SW}$ for measures with finite values of $SJ_2:\P_2(\R^d)\rightarrow[0,+\infty]$ defined by
\begin{equation}\label{def:SJ2}
    SJ_2(\mu)=\dashint_{\S^{d-1}}\int_{\R}\frac{F^\theta(r)(1-F^\theta(r))}{f^\theta(r)}\,dr\,d\theta
\end{equation}
where $f^\theta$ and $F^\theta$ are respectively the density and the CDF of $R_\theta \mu$, and we use the convention $0/0=0$. The functional $SJ_2$, introduced in~\cite{ManBalWas22}, is a sliced analogue of the functional $J_2$ introduced by Bobkov and Ledoux~\cite{BobLed19} for one dimensional measures; in general, $f^\theta$ is defined as the $\L^1$-density of the absolutely continuous component of $\hat\mu^\theta$, which need not be absolutely continuous with respect to $\L^1$. In the 1D case, finiteness of $J_2$ is necessary and sufficient for $\Ex W_2(\mu,\mu^n)$ to decay at rate $n^{-1/2}$~\cite[Section 5]{BobLed19}.
\smallskip

Manole, Balakrishnan, and Wasserman established~\cite[Proposition 4]{ManBalWas22}
\[\grn\Ex SW_2(\mu,\mu^n) \nc\leq C \sqrt{SJ_2(\mu)}\, n^{-1/2}\]
for some {constant $C>0$ independent of $\mu$}. Theorem~\ref{thm:lsw;para_rate} provides an analogous concentration result for $\ell_{SW}$ only with the additional assumption that $\hat\mu^\theta\ll\L^1$ for a.e. $\theta\in\S^{d-1}$; note that this assumption holds whenever $\mu$ is absolutely continuous with respect to the Lebesgue measure on an affine hyperplane of dimension at least 1.
    
\begin{theorem}[Parametric estimation rate of empirical measures in $\ell_{SW}$]\label{thm:lsw;para_rate}
Let $\mu\in\P_2(\R^d)$ be such that $\hat\mu^\theta\ll\L^1$ for a.e. $\theta\in\S^{d-1}$. 
Let $\mu^n=\frac1n\sum_{i=1}^n \delta_{X_i}$ where $X_i$ are i.i.d samples of $\mu$. Then for each $c>d+1$, we have
\begin{equation}\label{eq:lsw;para_rate}
    SW_2(\mu^n,\mu)\leq \ell_{SW}(\mu^n,\mu) \leq \sqrt{\frac{64c\log n}{n}}\sqrt{SJ_2(\mu)}
\end{equation}
    with probability at least $1-8(2n+1)^{d+1}n^{-c}$, where $SJ_2$ is as defined in \eqref{def:SJ2}.
\end{theorem}

\begin{proof}
As noted earlier, letting $d\hat\mu^\theta=f^\theta\,d\L^1$ for a.e. $\theta\in\S^{d-1}$ we have 
\[\grn\ell_{SW}^2(\mu,\mu^n)\leq 4\dashint_{\S^{d-1}} \|R_\theta\mu-R_\theta\mu^n\|_{\dot H^{-1}(\hat\mu^\theta)}^2\,d\theta\nc\leq 4\dashint_{\S^{d-1}} \int_{\R}\frac{|F^\theta(r)-F_n^\theta(r)|^2}{f^\theta(r)}\,dr\,d\theta\]
where $F^\theta$ and $F_n^\theta$ are the respective CDFs of $\hat\mu^\theta$ and $R_\theta\mu^n$. By Corollary~\ref{cor:VC_relative_CDF} we have 
\begin{align*}
    \prob\left(\sup_{r\in{\R},\theta\in\S^{d-1}}|F^\theta(r)-F_n^\theta(r)|> s\sqrt{F^\theta(r)(1-F^\theta(r)})\right)\leq 8(2n + 1)^{d+1}e^{-\frac{ns^2}{16}}.
\end{align*}
Choosing $s= 4\sqrt{c\frac{\log n}{n}}$, with probability at least $1-8(2n+1)^{d+1}n^{-c}$ we have
\begin{align*}
    \ell_{SW}^2(\mu,\mu^n)\leq 4\dashint_{\S^{d-1}}\|R_\theta\mu-R_\theta\mu^n\|_{\dot H^{-1}(\hat\mu^\theta)}^2\,d\theta
    &= 4\dashint_{\S^{d-1}} \int_{\R}\frac{|F^\theta(r)-F_n^\theta(r)|^2}{f^\theta(r)}\,dr\,d\theta\\
    &\leq 64c\frac{\log n}{n}\dashint_{\S^{d-1}}\int_{\R}\frac{F^\theta(r)(1-F^\theta(r))}{f^\theta(r)}\,dr\,d\theta.
\end{align*}
\end{proof}
We now turn to providing practical, intuitive, and geometric conditions for finiteness of $SJ_2(\mu)$. We show that one can uniformly bound the ratio $F^\theta(1-F^\theta)/f^\theta$ using a Cheeger-type isoperimetric constant $h(\mu)$ of the probability measure $\mu$, defined in the following way by Bobkov~\cite{Bob99}.
\begin{definition}[Cheeger-type isoperimetric constant]\label{def:Cheeger_const}
    Let $\mu\in\P(\R^d)$. The isoperimetric constant $h(\mu)$ of $\mu$ is defined by
    \begin{equation}\label{eq:Cheeger_const}
        h(\mu)=\inf_{A\subset\R^d}\frac{\mu^+(A)}{\min\{\mu(A),1-\mu(A)\}},
    \end{equation}
    where the infimum is taken over all Borel sets $A\subset \R^d$ and $\mu^+$ is defined by
    \begin{align*}
        \mu^+(A)=\liminf_{r\rightarrow 0^+}\frac{\mu(A^r)-\mu(A)}{r},
    \end{align*}
    where $A^r=\{x\in\R^d:\;|x-a|<r \text{ for some } a\in A\}$ is the open $r$-neighborhood of $A$. 
\end{definition}

\begin{corollary}\label{cor:lsw;para_rate;cheeger}
     Let $\mu\in\P_2(\R^d)$ be a probability measure with $\supp\mu\subset B(0,M)$ such that $h(\mu)>0$ and $\hat\mu^\theta\ll\L^1$ for a.e. $\theta\in\S^{d-1}$. Then
     \begin{equation}\label{eq:SJ2;cheeger-bd}
     SJ_2(\mu) \leq \frac{2M}{h(\mu)}.
     \end{equation}
     In particular, letting $\mu^n=\frac1n\sum_{i=1}^n \delta_{X_i}$ where $X_i$ are i.i.d samples of $\mu$, we have, for each $c>d+1$
        \begin{equation}\label{eq:lsw;para_rate;cheeger}
        SW_2(\mu^n,\mu)\leq \ell_{SW}(\mu^n,\mu) \leq \sqrt{\frac{128 M c}{h(\mu)}}\sqrt{\frac{\log n}{n}}
        \end{equation}
        with probability at least $1-8(2n+1)^{d+1}n^{-c}$.
\end{corollary}
\begin{proof}
As noted by Bobkov and Houdr\'e~\cite[Theorem 1.3]{BobHou97}, in one dimensions we have the characterization
    \begin{align*}
h(\hat\mu^\theta)=\essinf_{r\in\supp\hat\mu^\theta}\frac{f^\theta(r)}{\min\{F^\theta(r),(1-F^\theta(r))\}} \;\text{ where } d\hat\mu^\theta=f^\theta\,d\L^1,\;F^\theta(r)=\int_{-\infty}^{r} f^\theta(s)\,ds.
    \end{align*}
    Noting $F^\theta(r)(1-F^\theta(r))\leq \min\{F^\theta(r),1-F^\theta(r)\}$, we deduce
    \[\int_{-M}^{M}\frac{F^\theta(r)(1-F^\theta(r))}{f^\theta(r)}\,dr\leq \frac{2M}{h(\hat\mu^\theta)}.\]
\grn Moreover, note that for  any Borel set $A\subset\R$,
$\; \frac{(\hat\mu^{\theta})^+(A)}{\min\{\hat\mu^\theta(A),1-\hat\mu^\theta(A)\}} = \frac{\mu^+((\tilde\pi^\theta)^{-1}(A))}{\min\{\mu((\tilde\pi^\theta)^{-1}(A)),1-\mu((\tilde\pi^\theta)^{-1}(A))\}}$, where $\tilde\pi^\theta(x)=x\cdot\theta$. In particular $h(\hat\mu^\theta)\geq h(\mu)$, and we obtain~\eqref{eq:SJ2;cheeger-bd}.
Thus \eqref{eq:lsw;para_rate;cheeger} follows directly from Theorem~\ref{thm:lsw;para_rate}.\nc
\end{proof}

\begin{remark}\label{rmk:Cheeger_const}
    The isoperimetric constant $h(\mu)$ quantifies the narrowness of the `bottleneck' of $\mu$. Note that if the support of $\mu$ is disconnected then $h(\mu)=0$.
    On the other hand, for log-concave $\mu\in\P_2(\R^d)$ Bobkov established a positive lower bound on $h(\mu)$\cite[Theorem 1.2]{Bob99}. 

    Furthermore, $h(\mu)$ is bounded from below by the $L^1$-Poincar\'e constant of $\mu$. Indeed, $h(\mu)$ can be alternatively characterized as the largest constant satisfying the inequality
    \[h\int_{\R^d} |\varphi-\med_\mu(\varphi)|\,d\mu\leq \int_{\R^d}|\nabla\varphi|\,d\mu\]
    for all integrable locally Lipschitz $\varphi:\R^d\rightarrow\R$ with median $\med_\mu \varphi$ (while median is nonunique, the statement holds for every median) with respect to the measure $\mu$; see for instance the proof of~\cite[Theorem 3.1]{BobHou97}. As
    \[\int_{\R^d} |\varphi-\med_\mu(\varphi)|\,d\mu\leq \int_{\R^d} \left|\varphi-\int_{\R^d}\varphi\,d\mu\right|\,d\mu,\]
    the constant $h(\mu)$ is at least as large as the $L^1$-Poincar\'e constant of $\mu$. Consequently, for any bounded open connected $\Omega$ with Lipschitz boundary, the measure $\lambda_\Omega:=|\Omega|^{-1}\L^d\restr_\Omega$ satisfies $h(\lambda_\Omega)>0$.

    Moreover, if $\mu,\lambda\in\P(\R^d)$ satisfy $a\lambda\leq\mu\leq b\lambda$ for some $0<a\leq b<\infty$, then we have $h(\mu)\geq \frac{a}{b}h(\lambda)$. Indeed, for each Borel set $A\subset\R^d$
    \[\mu(A)\leq b\lambda(A),\qquad 1-\mu(A)=\mu(\R^d\setminus A)\leq b\lambda(\R^d\setminus A) = b(1-\lambda(A)),\]
    whereas $\mu^+(A)\geq a\lambda^+(A)$. Thus
    \begin{align*}
        \frac{\mu^+(A)}{\min\{\mu(A),1-\mu(A)\}} \geq \frac{a\lambda^+(A)}{b\min\{\lambda(A),1-\lambda(A)\}}.
    \end{align*}
    Thus any measure $\mu$ comparable to (bounded above and below by) $\L^d\restr\Omega$ satisfies $h(\mu)>0$, given that $\Omega$ is bounded open connected and has a Lipschitz boundary.
\end{remark}
\nc

\purp
\begin{remark}\label{rmk:VC_relative_CDF;sharp}
    Tudor Manole pointed out to us that the $O(\sqrt{\log n/n})$-rate concentration bound of Corollary~\ref{cor:lsw;para_rate;cheeger} is likely not sharp, as the relative VC inequality (Theorem~\ref{thm:VC_relative}) may be suboptimal when applied to CDFs.
    
    Indeed, the asymptotically sharp uniform bound on $|F(r)-F_n(r)|/\sqrt{F(r)(1-F(r))}$ is of order $O(\log\log n/n)$, where $F$ is the CDF of $\mu$ and $F_n$ the corresponding empirical CDF; see the recent survey~\cite[Section 3.1]{SarkarKuchibhotla23} and references therein. This would lead to an improvement to $O(\sqrt{\log\log n/n})$-rate in Corollary~\ref{cor:lsw;para_rate;cheeger}. As the relative VC inequality allows convenient uniform bound on empirical CDFs over all $\theta\in\S^{d-1}$, we do not pursue this refinement in this paper.
\end{remark}
\nc

\section{Metric slopes and gradient flows in the sliced Wasserstein  space}\label{sec:SWGF}
We examine the consequences of the local geometry of the sliced Wasserstein space on metric slopes and gradient flows. As in the previous sections, we  contrast the behaviors of the metric slopes and gradient flows at absolutely continuous and discrete measures. 

Since we are dealing with both the SW metric and the induced intrinsic distance $\ell_{SW}$, we start by commenting on the relationship between gradient flows with respect to the ambient and the intrinsic metric. In a general metric space $(X,m)$ the ``gradient flows'' of an energy $\E: X \to (-\infty, +\infty]$ are defined as \emph{curves of maximal slope}, namely the continuous curves $u:[0,T] \to X$ that satisfy
\begin{equation}\label{def:CMS}
    \frac{d}{dt}\E(u_t)\leq -\frac{1}{2}|u'|_{m}^2(t)-\frac12|\partial\E|_m^2(u_t) \text{ for a.e. }t\in I,
\end{equation}
where $|u'|_m$ is the metric derivative \eqref{def:met_der} and $|\partial\E|_m$ is the metric slope defined in \eqref{def:met_slope}; see~\cite[Definition 1.3.2]{AGS} for a precise and more general definition. 

Suppose $\ell_m$ is the length metric induced by $m$; the definition above allows one to consider curves of maximal slope of $\E$ in $(X, \ell_m)$ as well. We note that if $X$ is a Riemannian manifold isometrically embedded in $\R^d$ and $m$ is the Euclidean metric, then $\ell_m$ is the Riemannian distance with respect to the Riemannian metric of the manifold. It is straightforward to see that the gradient flows in the classical sense on the manifold coincide with the curves of maximal slope in both $(X, \ell_m$) and $(X,m)$. 

This equivalence is not as clear in full generality for curves of maximal slopes. As $m\leq \ell_m$, in general $|\partial\E|_{\ell_m}\leq|\partial\E|_m$ and \grn$|u'|_{m}(t)\leq |u'|_{\ell_m}(t)$ \nc for any absolutely continuous curve $u:I\rightarrow X$.
Furthermore, 
\[|u'|_{\ell_m}(t) \leq \lim_{h\searrow 0}\frac{\ell_m(t,t+h)}{h}\leq\lim_{h\searrow 0}\frac{1}{h}\int_t^{t+h}|u'|_m(s)\,ds = |u'|_m(t)\; \text{ for }\L^1\text{-a.e. } t\in I,\]
where the last equality holds by absolute continuity. Consequently, any curve of maximal slope with respect to $m$ is a curve of maximal slope in $\ell_m$.

However, it is in general unclear exactly when $|\partial\E|_{\ell_m}=|\partial\E|_m$ holds. Muratori and Savar\'{e} showed the equivalence for approximately $\lambda$-convex functional $\E$~\cite[Proposition 2.1.6]{MurSav20}. On a different note, the weighted energy dissipation (WED) approach to constructing curves of maximal slope, studied by 
Rossi, Savar\'{e}, Segatti, and Stefanelli~\cite{RSS19}  
relies on functionals that only involve metric derivatives and 
the energy, but not the metric slope and thus does not distinguish between $m$ and $\ell_m$. The authors \grn construct \nc solutions  of \eqref{def:CMS} 
with metric slope $|\partial \E|$ replaced by its relaxation $|\partial^- \E|$, provided  $|\partial^- \E|$ is a strong upper gradient, which is not the case in general, and in particular is not true for potential energies in the SW space; see Corollary~\ref{cor:relaxed_slope;pot}.
\smallskip

Let us now return to the discussion of metric slopes in the SW space. At an absolutely continuous measure $\mu\in\P_2(\R^d)$ where we have the comparison (see Theorem~\ref{thm:lsw-sw-H;comparison})
\[\|\mu-\nu\|_{\dot H^{-(d+1)/2}(\R^d)}\lesssim SW(\mu,\nu)\leq \ell_{SW}(\mu,\nu)\lesssim \|\mu-\nu\|_{\dot H^{-(d+1)/2}(\R^d)} \,\text{ for suitable } \nu\in\P_2(\R^d),\]
we formally expect
\begin{equation}\label{eq:slope_comparison;ac}
     |\partial\E|_{\dot H^{-(d+1)/2}(\R^d)}(\mu)\lesssim |\partial\E|_{\ell_{SW}}(\mu)\leq |\partial\E|_{SW}(\mu) \lesssim |\partial\E|_{\dot H^{-(d+1)/2}(\R^d)}(\mu).    
\end{equation}
On the other hand, at a discrete measure $\mu^n=\sum_{i=1}^n m_i \delta_{y_i}$, where we have comparison
\[SW(\mu,\nu)=\frac{1}{\sqrt{d}}W(\mu,\nu)+o(SW(\mu,\nu)) \text{ for suitable } \nu\in\P_2(\R^d),\]
we expect
\begin{equation}\label{eq:slope_comparison;discrete}
|\partial\E|_{SW}(\mu^n)=\sqrt{d}\,|\partial\E|_{W}(\mu^n).
\end{equation}
Of course, the comparison theorems of Section~\ref{sec:comparison} require restrictive conditions on $\mu,\nu$ and thus the comparisons of $|\partial\E|_{SW}$ above are formal; rigorously establishing this in generality would be challenging. Hence, we provide rigorous proofs of \eqref{eq:slope_comparison;ac} and \eqref{eq:slope_comparison;discrete} for the potential energy $\Vcal(\mu):=\int_{\R^d} V(x)\,d\mu(x)$ for suitable $V:\R^d\rightarrow [0,+\infty)$, at absolutely continuous measures in Section~\ref{ssec:SWGF;ac;potential} and at discrete measures in Section~\ref{ssec:SWGF;discrete}, respectively. 
Understanding of the metric slope allows us to show instability of curves of maximal slope in terms of initial data; see Proposition~\ref{prop:swslope;potential} and Remark~\ref{rmk:SWGF;instability}. 
    
\subsection{Formal sliced Wasserstein gradient flows at smooth densities}\label{ssec:SWGF;smooth_formal}
We begin by formally deriving partial differential equations corresponding to sliced Wasserstein gradient flows, emphasizing that they are of order $d-1$ higher than their Wasserstein counterparts. 
 For this purpose, it is convenient to limit our attention to the space of smooth positive measures $\P_2^\infty(\R^d)$, defined by
\begin{equation}\label{def:P2infty}
        \P_2^\infty(\R^d)=\left\{\rho\,d\L^d \in\P_2(\R^d):\;\rho\in C^\infty(\R^d),\;\rho>0\right\}.
\end{equation}    
Consider $\partial_t\mu_t+\nabla\cdot J_t=0$ and set  \grn$(\mu,J) :=(\mu_0,J_0)$\nc. Writing $\diff\E_\mu(J)=\left.\frac{d}{dt}\right|_{t=0}\E(\mu_t)$, note that the (standard) Wasserstein gradient $\nabla_W \E \in L^2(\mu;\R^d)$ satisfies
\[\diff \E_\mu(J)=\left\langle\frac{dJ}{d\mu},\nabla_W\E_\mu\right\rangle_{L^2(\mu)}=\langle J, \nabla_{W} \E_\mu \rangle_{\R^d}.\] 
Since the quadratic form $J\mapsto \|dRJ/dR\mu\|_{L^2(R\mu)}^2$ characterizes the local metric of the SW space at $\mu\in\P_2(\R^d)$, formally the sliced Wasserstein gradient flux $\nabla_{SW}\E_\mu\in\tanspace_\mu(\P_2(\R^d),SW)$ satisfies
\begin{equation}\label{eq:SWgrad_differential}
\diff\E_\mu(J) = \left\langle \frac{dRJ}{dR\mu}, \frac{dR (\nabla_{SW}\E_\mu)}{dR\mu}\right \rangle_{L^2(R\mu)}
=\left\langle RJ, \frac{dR (\nabla_{SW}\E_\mu)}{dR\mu}\right\rangle_{\Pd}.
\end{equation}
Suppose there exists some $\rv_\mu\in L^2(\hat\mu;\R^d)$ such that for all $J\in\tanspace_\mu (\P_2(\R^d),SW)$
\begin{equation}\label{eq:duality;wass_grad}
    \langle J, \nabla_W \E_{\mu}\rangle_{\R^d} = \langle RJ,\grn\rv_{\mu}\hat\mu\nc\rangle_{\Pd}.
\end{equation}
then $\nabla_{SW}\E_\mu=R^{-1}(\rv_\mu \hat\mu)$ satisfies \eqref{eq:SWgrad_differential}.
For simplicity, suppose $\nabla_W \E_{\mu} \in \Sch(\R^d;\R^d)$ and thus the inversion formula is valid.
Then by \eqref{eq:duality;radon;dist}
\[\rv_\mu = c_d^{-1} \Lambda_d R\nabla_W \E_{\mu}\]
satisfies \eqref{eq:duality;wass_grad}, and thus by the inversion formula,
\[\nabla_{SW}\E_\mu=R^{-1}(\rv_\mu\hat\mu)=c_d^{-2} R^\ast \Lambda_d (\hat\mu \Lambda_d R\nabla_W\E_\mu).\]
By the definition \eqref{def:tanspace;lsw:mu}, if $\nabla_W \E_\mu=\nabla\varphi$ for some potential $\varphi$ then $\nabla_{SW}\E_\mu\in\tanspace_\mu(\P_2(\R^d);SW)$.
Thus, formally, the gradient flow of $\E$ in $(\P_2^\infty(\R^d),\ell_{SW})$ satisfies the equation
\begin{equation}\label{eq:GF;lsw}
\partial_t\mu_t - c_d^{-2} \nabla\cdot \left(R^\ast \Lambda_d(\hat\mu_t\Lambda_d R\nabla_W\E_\mu)\right) = 0.
\end{equation}
Observe that the order of \eqref{eq:GF;lsw} is $d-1$ higher than the corresponding Wasserstein gradient flow equation. Namely, each $\Lambda_d$ is a differential operator of order $d-1$, whereas \grn$R^\ast$ and $R$ jointly regularizes the function by $d-1$ derivatives\nc.
Note that the energy dissipation for \eqref{eq:GF;lsw} is formally 
\[\frac{d}{dt}\E(\mu_t)=-c_d^{-2}\langle \hat\mu_t\Lambda_d R(\nabla_W\E_{\mu_t}),\Lambda_d R(\nabla_W\E_\mu)\rangle_{\Pd}=-c_d^{-2}\|\Lambda_d R(\nabla_W\E_\mu)\|_{L^2(\hat\mu_t)}^2.\]

\subsection{Metric slopes of potential energies at absolutely continuous measures}\label{ssec:SWGF;ac;potential}
In the formal computations we have seen that a gradient flow $(\mu_t)_{t\in I}$ of $\E$ satisfies
\[\frac{d}{dt}\E(\mu_t)=-c_d^{-2}\|\Lambda_d R(\nabla_W\E_\mu)\|_{L^2(\hat\mu_t)}^2\sim \|\nabla_W\E_\mu\|_{\dot H^{(d-1)/2}(\R^d)}^2.\]
Letting $\E(\mu)=\Vcal(\mu)=\int_{\R^d}V(x)\,d\mu$ for smooth $V:\R^d\rightarrow\R$, we know $\nabla_W\Vcal_\mu=\nabla V$. Thus along the $SW$-gradient flow \grn $(\mu_t)_{t\geq 0}$\nc, we have
\[\frac{d}{dt}\Vcal(\mu_t)\sim - \|\nabla V\|_{\dot H^{(d-1)/2}(\R^d)}^2 = -\|V\|_{\dot H^{(d+1)/2}(\R^d)}^2.\]
Remark~\ref{rmk:grad;negsob} shows that \grn the dissipation of $\dot H^{-(d+1)/2}$-gradient flow $(\tilde\mu_t)_{t\geq 0}$ of $\Vcal$ is of the same order:
\[\frac{d}{dt}\Vcal(\tilde\mu_t)=-\|V\|_{\dot H^{(d+1)/2}(\R^d)}^2.\]
\nc
\begin{remark}[Gradient flows with respect to the $\dot H^{-s}$ norm]\label{rmk:grad;negsob}
        Let $\mu\in\P_2(\R^d)$ be a measure with $L^2(\R^d)$ density, and let us identify $\mu$ with its density. Let $\E:L^2(\R^d)\rightarrow \R$ be a functional that admits an $L^2$ gradient -- i.e. at suitable $\mu\in L^2(\R^d)$ there exists $\nabla_{L^2}\E_\mu\in L^2(\R^d)$ such that for each $v\in L^2(\R^d)$ with $\int v=0$
        \begin{align*}
            \left.\frac{d}{d\eps}\right|_{\eps=0}\E(\mu+\eps v)=\langle \nabla_{L^2} \E_{\mu},v\rangle_{L^2(\R^d)}.
        \end{align*}
        Assuming $\nabla_{L^2}\E_\mu$ is sufficiently smooth, the gradient $\nabla_{\dot H^{-s}}\E_\mu$ of $\E$ with respect to the $\dot H^{-s}$ norm is formally given by
        \[\nabla_{\dot H^s} \E_\mu = (-\Delta)^{s}\nabla_{L^2}\E_{\mu},\]
        as $\langle (-\Delta)^{s}\nabla_{L^2}\E_{\mu}, v\rangle_{\dot H^{-s}(\R^d)} =\langle \nabla_{L^2} \E_{\mu},v\rangle_{L^2(\R^d)}$.
        Thus the $\dot H^{-s}(\R^d)$ gradient flow of $\E$ formally satisfies the PDE
        \[\partial_t\mu_t + (-\Delta)^s \nabla_{L^2}\E_{\mu_t} =0\]
        and we see that the PDE is precisely of order $2s$ higher than that of the $L^2$ gradient flow equation; \vio note that the $\dot H^{-s}$-gradient flow equation has the structure $\partial_t\mu_t=-\nabla_{\dot H^s}\E_\mu$, whereas the Wasserstein gradient flow satisfies an equation formulated in terms of the continuity equation. \nc Furthermore, dissipation of the gradient flow is
        \[\frac{d}{dt}\E(\mu_t)= -\|\nabla_{\dot H^{-s}}\E_{\mu_t}\|_{H^{-s}(\R^d)}^2 = -\|(-\Delta)^s \nabla_{L^2}\E_{\mu_t}\|_{H^{-s}(\R^d)}^2=-\|\nabla_{L^2}\E_{\mu_t}\|_{\dot H^{s}(\R^d)}^2.\]
For  $\E=\Vcal$,  we note that $\nabla_{L^2}\Vcal_\mu=V$, and hence $\dot H^s$-gradient flow of $\Vcal$ satisfies
        \begin{align*}
            \frac{d}{dt}\Vcal(\mu_t)=-\|V\|_{\dot H^s(\R^d)}^2.
        \end{align*}
    \end{remark}

    Applying Theorem~\ref{thm:lsw-sw-H;comparison}, we demonstrate that \eqref{eq:slope_comparison;ac} holds for potential energy functionals with smooth compactly supported $V$.
    
\begin{proposition}[Slope of potential energies at absolutely continuous measures]\label{prop:metslope;potential-ac}
Let $V:\R^d\rightarrow[0,+\infty)$ be smooth and compactly supported.
Let $\Vcal:\P_2(\R^d)\rightarrow [0,+\infty)$ be the potential energy functional
\[\Vcal(\nu)=\int_{\R^d} V(x)\,d\nu(x).\] 
Let $\Omega$ an open bounded connected domain containing $\supp V$ with $\dist(\supp V, \partial\Omega)=:\alpha>0$, and $\mu\in\P_2(\R^d)$ an absolutely continuous probability measure such that
\[\frac{a}{|\Omega|}\L^d\restr_{\Omega}\leq \mu\leq \frac{b}{|\Omega|}\L^d\restr_{\Omega} \text{ for some } 0<a<b<\infty.\]
Then
\begin{equation}\label{eq:metslope;potential_ac}
     \|V\|_{\dot H^{(d+1)/2}(\R^d)}\lesssim_{a,\alpha,|\Omega|} |\partial\Vcal|_{\ell_{SW}}(\mu) \leq |\partial\Vcal|_{SW}(\mu) \leq c_d^{-1}\|\partial_r \Lambda_d R V\|_{L^2(\hat\mu)} \lesssim_{c_d,b,\Omega} \|V\|_{\dot H^{(d+1)/2}(\R^d)}.
\end{equation}
\end{proposition}
\begin{proof}
To obtain the upper bound, note that as $V\in C_c^\infty(\R^d)$ the Radon inversion \eqref{eq:radon_inversion_f2} and duality formula  for finite measures \eqref{eq:duality;radon;measure} imply that for any $\nu\in\P_2(\R^d)$, writing $\widehat\gamma\in\widehat\Gamma_o(\mu,\nu)$, we have
\begin{align*}
\Vcal(\mu)-\Vcal(\nu)&=\int_{\R^{d}} V\,d(\mu-\nu)=c_d^{-1}\int_{\R^{d}} R^\ast \Lambda_d RV d(\mu-\nu) = c_d^{-1}\dashint_{\S^{d-1}}\int_{\R} \Lambda_d RV\, d(\hat\mu^\theta-\hat\nu^\theta)\,d\theta  \\
&= c_d^{-1}\dashint_{\S^{d-1}}\int_{\R} \frac{\Lambda_d RV(r,\theta)-\Lambda_d RV(q,\theta)}{r-q}(r-q)\,d\hat\gamma^\theta(r,q)\,d\theta.
\end{align*}
Furthermore, as $\Lambda_d RV$ is of class $C^\infty$,
\begin{align*}
    |\Lambda_d RV(q,\theta)-\Lambda_d RV(r,\theta)|\leq |\partial_r \Lambda_d R V(r,\theta)||r-q|+O(|r-q|^2)
\end{align*}
and thus
\begin{align*}
    &\left|\dashint_{\S^{d-1}}\int_{\R} \frac{\Lambda_d RV(r,\theta)-\Lambda_d RV(q,\theta)}{r-q}(r-q)\,d\hat\gamma^\theta(r,q)\,d\theta\right|  \\
    &\leq \dashint_{\S^{d-1}}\int_{\R} |r-q|(|\partial_r \Lambda_d R V(r,\theta)|+O(|r-q|))\,d\hat\gamma^\theta(r,q)\,d\theta  \\
    &\leq \dashint_{\S^{d-1}}\int_{\R} |r-q||\partial_r \Lambda_d R V(r,\theta)|)\,d\hat\gamma^\theta(r,q)\,d\theta + \dashint_{\S^{d-1}}\int_{\R} O(|r-q|^2)\,d\hat\gamma^\theta(r,q)\,d\theta \\
    &\leq SW(\mu,\nu)\|\partial_r \Lambda_d R V\|_{L^2(\hat\mu)}+SW^2(\mu,\nu).
\end{align*}
Hence
\begin{align*}
    |\Vcal(\mu)-\Vcal(\nu)|\leq c_d^{-1}SW(\mu,\nu)(\|\partial_r \Lambda_d R V\|_{L^2(\hat\mu)}+SW(\mu,\nu))
\end{align*}
and deduce the upper bound
\begin{align*}
    |\partial\Vcal|_{SW}(\mu) &=\limsup_{\nu\rightarrow \mu}\frac{[\Vcal(\mu)-\Vcal(\nu)]_+}{SW(\mu,\nu)}\\
    &\leq \limsup_{\nu\rightarrow\mu}\frac{c_d^{-1}SW(\mu,\nu)(\|\partial_r \Lambda_d R V\|_{L^2(\hat\mu)}+SW(\mu,\nu))}{SW(\mu,\nu)} = c_d^{-1}\|\partial_r \Lambda_d R V\|_{L^2(\hat\mu)}.
\end{align*}
Note that $\hat\mu^\theta\leq b C_\Omega\L^1$, where $C_\Omega=C_{\lambda_\Omega}$ is as in Theorem~\ref{thm:lsw-sw-H;comparison} with $\lambda_\Omega=|\Omega|^{-1}\L^d\restr_\Omega$. Thus by the Radon isometry \eqref{eq:Radon_isometry;Hts}
\[\|\partial_r \Lambda_d R V\|_{L^2(\hat\mu)}\lesssim_{b/|\Omega|} \|\Lambda_d R V\|_{\dot H^1(\Pd)} = \|V\|_{\dot H^{(d+1)/2}(\R^d)}.\]    

As $SW\leq \ell_{SW}$, $|\partial\Vcal|_{SW}(\mu)\geq|\partial\Vcal|_{\ell_{SW}}(\mu)$ and it only remains to prove the lower bound
\[\|V\|_{\dot H^{(d+1)/2}(\R^d)}\lesssim_{a,\alpha,|\Omega|} |\partial\Vcal|_{\ell_{SW}}(\mu)\]
for $\mu$ satisfying the provided conditions.
To do so, define for each $\eps>0$
\[\sigma_\eps:=\mu-\eps(-\Delta)^{\frac{d+1}{2}} V(x)\,dx.\]
As $V\in C_c^\infty$, $\|(-\Delta)^{\frac{d+1}{2}}V\|_\infty<\infty$ and $\mu$ is bounded away from zero on $\Omega\supset \supp V$, $\sigma_\eps\geq 0$ when $\eps$ is sufficiently small.
Furthermore, as $(-\Delta)^{\frac{d+1}{2}}V\in C_c^\infty(\R^d)$, $\sigma_\eps$ has bounded second moments, and integrating by parts in a sufficiently large ball containing $\supp V$ we may deduce $\int_{\R^d}(-\Delta)^{\frac{d+1}{2}} V =0$, hence $\sigma_\eps\in\P_2(\R^d)$. Moreover, $\sigma_\eps$ has uniformly bounded second moments and converges to $\mu$ narrowly, thus $\sigma_\eps\rightarrow\mu$ in $SW$ as $\eps\searrow 0$. Therefore
\begin{align*}
    \int_{\R^d} V(x)\,d(\mu-\sigma_\eps)=\eps \int_{\R^d} V(x) (-\Delta)^{\frac{d+1}{2}}V\,dx = \eps  \|(-\Delta)^{\frac{d+1}{4}} V\|_{L^2(\R^d)}^2 = \eps \|V\|_{\dot H^{(d+1)/2}(\R^d)}^2.
\end{align*}
As further $\mu=\mu_\eps$ on $\Omega\setminus\supp V$, by the comparison theorem at absolutely continuous measures (Theorem~\ref{thm:lsw-sw-H;comparison}),
\begin{align*}
    |\partial\Vcal|_{\ell_{SW}}(\mu)&=\limsup_{\nu\rightarrow\mu}\frac{\Vcal(\mu)-\Vcal(\nu)}{\ell_{SW}(\mu,\nu)}
    \geq \limsup_{\eps\searrow 0} \frac{\Vcal(\mu)-\Vcal(\sigma_\eps)}{\ell_{SW}(\mu,\sigma_\eps)} \\
    &\gtrsim_{a,\alpha,\Omega} \limsup_{\eps\searrow 0} \frac{\Vcal(\mu)-\Vcal(\sigma_\eps)}{\|\mu-\sigma_\eps\|_{\dot H^{-(d+1)/2}(\R^d)}}
    = \frac{\|V\|_{\dot H^{(d+1)/2}(\R^d)}^2}{\|(-\Delta)^{\frac{d+1}{2}}V\|_{\dot H^{-(d+1)/2}(\R^d)}} = \|V\|_{\dot H^{(d+1)/2}(\R^d)}.
\end{align*}
\end{proof}

\subsection{Metric slopes of potential energies at discrete measures}\label{ssec:SWGF;discrete}
In this section we focus on the equivalence of the Wasserstein and the sliced Wasserstein metric slopes of potential energies at discrete measures. 

\purp Given a functional $\E:\P_2(\R^d)\rightarrow (-\infty,+\infty]$, a metric $m:\P_2(\R^d)\times\P_2(\R^d)\rightarrow [0,+\infty)$, a time-step $\tau>0$, and a base point $\mu\in\P_2(\R^d)$ let us write
\begin{equation}\label{def:Phi_SW}
    \E^{m}(\nu;\tau,\mu):=\frac{m^2(\nu,\mu)}{2\tau}+\E(\nu) \text{ for each }\nu\in\P_2(\R^d).
\end{equation}
We denote by $\E^m_\tau:\P_2(\R^d)\rightarrow (-\infty,+\infty]$ the Moreau-Yosida approximation of $\E$ with with respect to metric $m$ and time step $\tau>0$
\begin{equation}\label{def:Em_tau}
    \E^m_\tau(\mu)=\inf_{\nu\in\P_2(\R^d)}\E^m(\nu;\tau,\mu).
\end{equation}
Existence and uniqueness of the minimizer of \eqref{def:Phi_SW} with $m=SW$ in certain cases was discussed in~\cite{Bonet22}. General existence readily follows from the direct method of calculus of variations as we will see in the proof of Lemma~\ref{lem:sw-JKO-min;winfty;V}.

We will impose two weak regularity assumptions on $\E$, namely lower semicontinuity with respect to $SW$ and coercivity; we say $\E$ is coercive if there exists $\tau_\ast>0$ and $\mu_\ast\in\P_2(\R^d)$ such that
\begin{equation}\label{ass:coercivity}
    \E^m_\tau(\mu_\ast)>-\infty.
\end{equation}
It is well-known that the potential energy functional $\V(\mu)= \int_{\R^d} V(x)\,d\mu(x)$ is coercive for instance when the negative part of $V$ grows at most quadratically -- i.e. $V(x)\geq -C_1- C_2|x|^2$ for some $C_1,C_2>0$. Moreover, lower semicontinuity of $V$ in $\R^d$ implies lower semicontinuity of $\V$ with respect to the narrow topology, hence with respect to $SW$ and $W$. 

We stress that while we utilize the variational problem \eqref{def:Em_tau} in this section to characterize the metric slope via the duality formula, we \emph{do not} study the limiting curves of the minimizing movements scheme.

The duality formula for the local slope~\cite[Lemma 3.1.5]{AGS} in terms of the minimizers of the functional \eqref{def:Phi_SW} along with Theorem~\ref{thm:SWnear-discrete} allows us to establish the following sufficient condition for an energy functional $\E$ to satisfy $|\partial\E|_{SW}=\sqrt{d}\,|\partial\E|_W$ at discrete measures. 
\nc
\begin{lemma}\label{lem:swslope;E}
Let $\E:\P_2(\R^d)\rightarrow [0,+\infty)$ \vio be coercive and lower semicontinuous with respect to $SW$. Additionally, suppose that \nc at each discrete measure $\mu^n=\sum_{i=1}^n m_i\delta_{x_i}$, for sufficiently small $\tau>0$ \grn the functional $\E^{SW}(\cdot;\tau,\mu^n)$ as defined in \eqref{def:Phi_SW} \nc
admits minimizers $\mu_\tau$ such that $W_\infty(\mu_\tau,\mu^n)\xrightarrow[]{\tau\searrow 0}0$.
      
Then the slope of $\E$ at each discrete probability measures $\mu^n$ w.r.t $SW$ coincide with the slope with respect to $W/\sqrt{d}$ -- i.e.
\begin{equation}\label{eq:swslope;potential}
    |\partial\E|_{SW}(\mu^n)=\sqrt{d}\,|\partial\E|_{W}(\mu^n).
\end{equation}
\end{lemma}

\begin{proof}
Fix $\mu^n=\sum_{i=1}^n m_i\delta_{x_i}$. By hypothesis, for sufficiently small $\tau>0$ we can find minimizers $\mu_\tau^{SW}$ of the JKO functional $\E^{SW}(\cdot;\tau,\mu^n)$ such that $W_\infty(\mu^n,\mu_\tau^{SW})\xrightarrow[]{\tau\searrow 0}0$. \vio By coercivity and lower semicontinuity with respect to $SW$, we can apply the duality formula for the local slope~\cite[Lemma 3.1.5]{AGS} to choose a sequence $\tau_k\rightarrow 0$ such that \nc
\[|\partial\E|_{SW}^2(\mu)=\lim_{k\rightarrow\infty}\frac{SW^2(\mu^n,\mu_{\tau_k})}{\tau_k^2}=\lim_{k\rightarrow\infty} \frac{\E(\mu^n)-\E(\mu_\tau^{SW})}{\tau_k}.\]
As $W_\infty(\mu^n,\mu_{\tau_k}^{SW})\xrightarrow[]{\tau\searrow 0}0$, by Theorem~\ref{thm:SWnear-discrete} we have
\[\lim_{k\rightarrow\infty}\frac{SW^2(\mu^n,\mu_{\tau_k}^{SW})}{2\tau_k^2}= \lim_{k\rightarrow\infty}\frac{W^2(\mu^n,\mu_{\tau_k}^{SW})}{2d\tau_k^2}.\]
On the other hand, \grn the Moreau-Yosida approximation $\E_\tau^{W/\sqrt{d}}$ satisfies \nc $\E_\tau^{W/\sqrt{d}}(\mu^n)\leq \E(\mu_\tau^{SW})+\frac{W^2(\mu^n,\mu_\tau^{SW})}{2d\tau}$, and thus
\[\frac{\E(\mu^n)-\E(\mu_\tau^{SW})}{\tau}-\frac{W^2(\mu^n,\mu_\tau^{SW})}{2d\tau^2}
\leq \frac{\E(\mu^n)-\E_\tau^{W/\sqrt{d}}(\mu^n)}{\tau}.
\]
Combining the estimates and again using the duality formula for $|\partial\mathcal{E}|_{W/\sqrt{d}}$,
\begin{align*}
    \frac12|\partial\E|_{SW}^2(\mu^n) &=\lim_{k\rightarrow\infty} \frac{\E(\mu^n)-\E(\mu_{\tau_k}^{SW})}{\tau_k}-\frac{SW^2(\mu^n,\mu_{\tau_k}^{SW})}{2\tau_k^2}    
    = \lim_{k\rightarrow\infty}\frac{\E(\mu^n)-\E(\mu_{\tau_k}^{SW})}{\tau}-\frac{W^2(\mu^n,\mu_{\tau_k}^{SW})}{2d\tau_k^2} \\
    & \leq \limsup_{\tau\searrow 0} \frac{\E(\mu^n)-\E_\tau^{W/\sqrt{d}}(\mu^n)}{\tau} 
    = \frac{1}{2}|\partial\E|_{W/\sqrt{d}}^2(\mu^n)= \frac{d}{2}\, |\partial\E|_{W}^2(\mu^n).
\end{align*}
\end{proof}

Our last step is to verify that the hypotheses of Lemma~\ref{lem:swslope;E} are satisfied for a general class of potential energy functionals.
\begin{lemma}\label{lem:sw-JKO-min;winfty;V}
Let $V:\R^d\rightarrow [0,+\infty)$ be \grn continuous \nc and let $\mathcal{V}:\mu\mapsto\int_{\R^d} V(x)\,d\mu(x)$. \grn Then, for each discrete measure $\mu^n=\sum_{i=1}^n m_i\delta_{y_i}\in\P_2(\R^d)$ \nc there exists $\tau^\ast=\tau^\ast(\mu,V)$ such that for all $0<\tau<\tau^\ast$, $\grn\Vcal^{SW}(\,\cdot\,;\tau,\mu^n)\nc$ as defined in \eqref{def:Phi_SW} admits a minimizer $\mu_\tau$ such that
\begin{equation}\label{eq:sw-JKO-min;winfty;V}
    W_\infty(\mu_\tau,\mu^n)\leq \grn\tau^{1/3}\nc.
\end{equation}
\end{lemma}

\begin{proof}        
As $\Vcal$ is nonnegative, for any fixed $\tau>0$ and $c<\infty$ the corresponding sublevel set
  \[\{\nu\in\P_2(\R^d):\Vcal^{SW}(\cdot;\tau,\mu^n)\leq c\}\]
is contained in the ball $B_{SW}(\mu^n,\sqrt{2\tau c})$, which is sequentially compact with respect to the {narrow} convergence of measures by {Proposition~\ref{prop:SWball=wcpct}}. On the other hand, $\Vcal^{SW}(\cdot;\tau,\mu^n)$ is lower semicontinuous with respect to the narrow convergence: $\Vcal$ is lower semicontinuous as $V$ is {continuous} and we know from Lemma~\ref{prop:SW2_lsc} that $\sigma\mapsto SW(\sigma,\mu^n)$ is also narrowly lower semicontinuous In conclusion,  for each $\tau>0$ minimizer $\mu_\tau$ of the JKO functional $\Vcal^{SW}(\cdot;\tau,\mu^n)$ exists.

In the remainder of the proof, we denote by $\kappa_d>0$ a constant only dependent on the dimension such that
\begin{equation}\label{def:kappa_d}
|\{\theta\in\S^{d-1}:\;|\theta_1|\leq s\}|\leq \kappa_d |\S^{d-1}| s.
\end{equation}
\vspace{3mm}
\noindent\emph{Step 1$^\circ$}
Let $\mu_\tau$ be a minimizer of \grn$\Vcal^{SW}(\cdot;\tau,\mu^n)$\nc. 
Let us write 
\[\Omega_{h}^n=\{x\in\R^d:\;|x-y_i|\leq h \text{ for some } i=1,2,\cdots,n\}\]
and decompose $\mu_\tau$ into
\begin{equation}\label{eq:mu_tau;tau1/6}
\mu_\tau=m_\tau^{far}\mu_\tau^{far}+\sum_{i=1}^n m_\tau^i \mu_\tau^i \text{ where } \supp\mu_\tau^i\subset \overline B(y_i,\tau^{1/3}),\;\supp\mu_\tau^{far}\cap \Omega_{\tau^{1/3}}^n =\emptyset.
\end{equation}
Here $\mu_\tau^{far}\in\P_2(\R^d)$ is the (normalized) part of $\mu_\tau$ that is far from the support of $\mu^n$, whereas $\mu_\tau^i\in\P_2(\R^d)$ is the part of $\mu_\tau$ close to $y_i$ for each $i=1,\cdots,n$. Let
\[m_{\tau}^{\Delta}=\sum_{i=1}^n (m_\tau^i-m_i)_+ = -m_\tau^{far}+\sum_{i=1}^n (m_i-m_\tau^i)_+.\]
Intuitively, $m_\tau^{far}$ is the mass outside the balls of radius $\tau^{1/3}$ about $y_i$, and $m_{\tau}^{\Delta}$ is the `total misplaced mass'. We want to show that both $m_\tau^{far}$ and $m_\tau^{\Delta}$ are 0 when $\tau$ is sufficiently small.
To deduce this, note that $\Vcal^{SW}(\mu_\tau;\tau,\mu^n)\leq \Vcal^{SW}(\mu^n;\tau,\mu^n)$, and thus
\begin{align*}
    \frac{SW^2(\mu_\tau,\mu^n)}{2\tau}
    &\leq \Vcal(\mu^n)-\Vcal(\mu_\tau) = \Vcal(\mu^n)-\Vcal(\sum_{i=1}^n m_\tau^i \delta_{y_i})+\Vcal(\sum_{i=1}^n m_\tau^i \delta_{y_i})-\Vcal(\mu_\tau)    \\
    &= \sum_{i=1}^n (m_i-m_\tau^i) V(y_i) + \sum_{i=1}^n m_\tau^i \int_{\R^d} V(y_i)-V(x)\,d\mu_\tau^i(x)    \\
    &\leq (m_{\tau}^{\Delta}+m_\tau^{far})\max_{i=1,\cdots,n} V(y_i)+\omega_V(\tau^{1/3}).
\end{align*}
where $\omega_V$ is \grn the uniform modulus of continuity of $V$ in, say $\bigcup_{i=1}^n\overline B(y_i,1)$\nc. On the other hand, we claim
\begin{equation}\label{eq:SW22mutau_lbd}
    \frac{\tau^{2/3}m_\tau^{far}}{16n^2\kappa_d^2 \tau}+\left(\frac{l_{\mu^n}^2}{16n^2 C_d^2}-\tau^{2/3}\right)\frac{m_{\tau}^{\Delta}}{2\tau}\leq \frac{SW^2(\mu_\tau,\mu^n)}{2\tau}
\end{equation}
\grn where $l_{\mu^n}=\min_{i\neq j}|y_i-y_j|$\nc. Note that this implies 
\begin{equation}\label{eq:mb-mdiff;bd}
    \frac{\tau^{2/3}m_\tau^{far}}{16n^2\kappa_d^2 \tau}+\left(\frac{l_{\mu^n}^2}{16n^2 C_d^2}-\tau^{2/3}\right)\frac{m_{\tau}^{\Delta}}{2\tau}\leq (m_{\tau}^{\Delta}+m_\tau^{far})\max_{y_i} V(y_i)+\omega_V(\tau^{1/3}).
\end{equation}
\grn The choice of radius $\tau^{1/3}$ in \eqref{eq:mu_tau;tau1/6} allows, after squaring, the coefficient of $m_\tau^{far}$ in the left-hand side to blow up as $\tau\searrow 0$, and may be replaced by $\tau^{1/2-\eps}$ for any $\eps>0$ \nc. From this we deduce that we can find $\tau^\ast=\tau^\ast(n,d,l_{\mu^n},\max_{i}V(y_i),\omega_V)>0$ such that for all $\tau<\tau_\ast$, 
\[m_{\tau}^{\Delta}=0 \text{ and } m_\tau^{far} =0.\]
By definition of $m_\tau^\Delta,m_\tau^{far}$, this means $\mu_\tau$ has exactly $m_i$ mass in each $\overline B(y_i,\tau^{1/3})$, thus
\[W_\infty(\mu_\tau,\mu^n)\leq \tau^{1/3} \text{ for } \tau\in(0,\tau^\ast).\]
Thus it remains to prove \eqref{eq:SW22mutau_lbd}. Without loss of generality we may assume $\mu_\tau\ll\L^d$. Otherwise, we approximate $\mu_\tau$ by its convolution $\mu_\tau^\eps$ with a smooth and compactly supported kernel. Then all the quantities involved in \eqref{eq:mb-mdiff;bd} such as $m_\tau^{far}(\eps),m_{\tau}^{\Delta}(\eps), SW^2(\mu_\tau^\eps,\mu^n)$, $V(\mu_\tau^\eps)$ converge as $\eps\searrow 0$, hence we can deduce \eqref{eq:mb-mdiff;bd} in the limit as $\eps\searrow 0$.

Furthermore, as $y\mapsto y\cdot\theta$ is a bijection on $\{y_1,\cdots,y_n\}$ for a.e. $\theta\in\S^{d-1}$, we have the existence of transport map $T^\theta:\R^d\rightarrow\R^d$ for a.e. $\theta\in\S^{d-1}$ such that 
\[SW^2(\mu_\tau,\mu^n)=\dashint_{\S^{d-1}}\int_{\R^d} |(T^\theta(x)-x)\cdot\theta|^2\,d\mu_\tau(x).\]

\vspace{3mm}
\noindent\emph{Step 2$^\circ$} 
To prove \eqref{eq:SW22mutau_lbd}, we separately consider $\mu_\tau^{far}$ and $\sum_{i=1}^n m_\tau^i \mu_\tau^i$ -- namely
\begin{align*}
    SW^2(\mu_\tau,\mu^n) &= \int_{\R^d} \dashint_{\S^{d-1}} |(T^\theta(x)-x)\cdot\theta|^2\,d\mu_\tau(x)\,d\theta \\
    &= m_\tau^{far} \int_{\R^d} \dashint_{\S^{d-1}} |(T^\theta(x)-x)\cdot\theta|^2\,d\mu_\tau^{far}(x)\,d\theta \\
    & \phantom{=} \, +\sum_{i=1}^n m_\tau^i \int_{\R^d} \dashint_{\S^{d-1}} |(T^\theta(x)-x)\cdot\theta|^2\,d\mu_\tau^i(x)\,d\theta \\
    &=: I_\tau^{far} + I_\tau^{\Delta}.
\end{align*}
We first deal with the term $I_\tau^{far}$
\begin{align*}
    I_\tau^{far}=m_\tau^{far} \int_{\R^d} \dashint_{\S^{d-1}} |(T^\theta(x)-x)\cdot\theta|^2\,d\mu_\tau^{far}(x)\,d\theta
    \geq m_\tau^{far} \int_{\R^d} \dashint_{\S^{d-1}} \min_{i}|(y_i-x)\cdot\theta|^2\,d\mu_\tau^{far}\,d\theta.
\end{align*}
Let
\[\Theta_x^i:=\left\{\theta\in\S^{d-1}:\;|(y_i-x)\cdot\theta|\leq \frac{\tau^{1/3}}{2n\kappa_d} \right\}.\]
As $|y_i-x|\geq \tau^{1/3}$, we deduce $|\Theta_x^i|\leq \frac{1}{2n}$ and $|\bigcup_{i=1}^n \Theta_x^i|\leq \frac12$. On $\S^{d-1}\setminus\bigcup_{i=1}^n \Theta_x^i$ we have $\min_i |(y_i-x)\cdot\theta|\geq \frac{\tau^{1/3}}{2n\kappa_d}$, and thus 
\begin{equation}\label{mb_ineq}
    I_\tau^{far} \geq m_\tau^{far} \int_{\R^d} \dashint_{\S^{d-1}} \min_{i}|(y_i-x)\cdot\theta|^2\,d\mu_\tau^{far}\,d\theta
    \geq \frac{m_\tau^{far} \tau^{2/3}}{8n^2\kappa_d^2}.
\end{equation}

\vspace{3mm}
\noindent\emph{Step 3$^\circ$}. To deal with the second term $I_\tau^{\Delta}$, define
\[B_\theta:=\{x\in\R^d:\;|T^\theta(x)-x|\geq l_{\mu^n}-\tau^{1/3}\}.\]
\grn The threshold $l_{\mu^n}-\tau^{1/3}$ is a lower bound on how far incorrectly assigned mass must travel. Thus \nc
\[m_\tau^i \mu_\tau^i(B_\theta)\geq (m_\tau^i-m_i)_+ \text{ for each } i=1,\cdots,n, \text{ and } \theta\in\S^{d-1}.\]
Furthermore, on $B_\theta\cap\supp\mu_\tau^i$ we have
\[|(T^\theta(x)-x)\cdot\theta|\geq \min_{j\neq i}|(y_i-y_j)\cdot\theta|-\tau^{1/3}.\]
\grn Using these two properties, we deduce \nc
\begin{align*}
    I_\tau^{\Delta} &= \sum_{i=1}^n m_\tau^i \dashint_{\S^{d-1}}\int_{\R^d} |(T^\theta(x)-x)\cdot\theta|^2\,d\mu_\tau^i(x)\,d\theta  \\  
    &\geq \frac12\sum_{i=1}^n m_\tau^i\dashint_{\S^{d-1}}\int_{B_\theta} \min_{j\neq i} |(y_i-y_j)\cdot\theta|^2-2\tau^{2/3}\,d\mu_\tau^i(x)\,d\theta    \\
    &\geq \frac12 \sum_{i=1}^n (m_\tau^i-m_i)_+ \dashint_{\S^{d-1}} \min_{j\neq i}|(y_i-y_j)\cdot\theta|^2-2\tau^{2/3}\,d\theta    \\
    &= \grn\frac12 \left(\sum_{i=1}^n (m_\tau^i-m_i)_+ \dashint_{\S^{d-1}} \min_{j\neq i}|(y_i-y_j)\cdot\theta|^2\,d\theta\right) - m_{\tau}^{\Delta}\tau^{2/3}.\nc
\end{align*}
Let
\[\Theta_i^j:=\left\{\theta\in\S^{d-1}:\;|(y_i-y_j)\cdot\theta|\leq \frac{l_{\mu^n}}{2n C_d}\right\}.\]
As $l_{\mu^n}\left|\theta\cdot\frac{y_i-y_j}{|y_i-y_j|}\right|\leq |(y_i-y_j)\cdot\theta|,$
by Chebyshev's inequality we have $\frac{|\Theta_i^j|}{|\S^{d-1}|}\leq \frac{1}{2n}$, and thus
\[\frac{|\bigcup_{j=1}^n \Theta_i^j|}{|\S^{d-1}|}\leq \frac{1}{2},\]
whereas 
\[\min_{j\neq i}|(y_i-y_j)\cdot\theta|>\frac{l_{\mu^n}}{2nC_d}\;\text{ for }\theta\in\S^{d-1}\setminus \bigcup_{j=1}^n\Theta_i^j.\]
Thus
\begin{align*}
    &\sum_{i=1}^n (m_\tau^i-m_i)_+ \dashint_{\S^{d-1}} \min_{j\neq i}|(y_i-y_j)\cdot\theta|^2\,d\theta   \\
    &\geq \sum_{i=1}^n (m_\tau^i-m_i)_+ \frac{1}{|\S^{d-1}|}\int_{\S^{d-1}\setminus \bigcup_{j=1^n}\Theta_i^j} \frac{l_{\mu^n}^2}{4n^2 C_d^2}\,d\theta  
    \geq \frac{l_{\mu^n}^2}{8n^2 C_d^2}\sum_{i=1}^n (m_\tau^i-m_i)_+  = \frac{l_{\mu^n}^2 m_\tau^{\Delta}}{8n^2 C_d^2}.
\end{align*}
Collecting the estimates, we have
\begin{align*}
    I_\tau^{\Delta} \geq  \grn\left(\frac12 \sum_{i=1}^n (m_\tau^i-m_i)_+ \dashint_{\S^{d-1}} \min_{j\neq i}|(y_i-y_j)\cdot\theta|^2\,d\theta\right) - m_{\tau}^{\Delta}\tau^{2/3}\nc
    \geq \left(\frac{l_{\mu^n}^2}{16n^2 C_d^2}-\tau^{2/3}\right)m_{\tau}^{\Delta}
\end{align*}
Recalling $SW^2(\mu_\tau,\mu^n)=I_\tau^{far}+I_\tau^\Delta$, we combine the above estimate with \eqref{mb_ineq} to obtain \eqref{eq:SW22mutau_lbd}.
\end{proof}

We are ready to state the main result of this section.
\begin{proposition}[Slope of potential energy at discrete measures]\label{prop:swslope;potential}
Let $V:\R^d\rightarrow [0,+\infty)$ be continuously differentiable. Let $\V$ be the potential energy functional
\[\mathcal{V}(\mu)=\int_{\R^d} V(x)\,d\mu(x).\]
Then the slope of $\Vcal$ at discrete probability measures $\mu^n=\sum_{i=1}^n m_i\delta_{x_i}$ w.r.t $SW$ coincide with the slope with respect to $W/\sqrt{d}$ -- i.e.
\begin{equation}\label{eq:swslope;V;discrete}
    |\partial\mathcal{V}|_{SW}(\mu^n)=\sqrt{d}\,|\partial\mathcal{V}|_{W}(\mu^n).
\end{equation}
\end{proposition}

\begin{proof}
    \vio As $\V$ satisfies the assumptions of Lemma~\ref{lem:swslope;E} and Lemma~\ref{lem:sw-JKO-min;winfty;V}, \nc \eqref{eq:swslope;V;discrete} follows from the two lemmas. 
\end{proof}

\grn We conclude this section by discussing implications of Proposition~\ref{prop:swslope;potential}. We first note in Remark~\ref{rmk:SWGF;instability} that the curves of maximal slope of the potential energy with respect to $SW$ is not stable in the initial data. \nc

\begin{remark}[Lack of stability of sliced Wasserstein gradient flows]\label{rmk:SWGF;instability}
We note that the curves of maximal slopes $(\mu^n_t)_t$ of $\Vcal$ with respect to the Wasserstein metric starting at any discrete measure $\mu^n$ are, up to a multiplicative constant, curves of maximal slopes with respect to $SW$ (or $\ell_{SW}$) metric. To see this, recall that the $W$-gradient flow of $\Vcal$ starting at $\mu^n=\sum_{i=1}^n m_i\delta_{x_i}$ is given by $\mu^n_t = \sum_{i=1}^n m_i\delta_{x_i(t)}$ where $x_i(t)$ solves $x_i'(t)=-\nabla V(x_i(t))$. Thus by Proposition~\ref{prop:swslope;potential} 
\[|\partial\Vcal|_{SW}(\mu^n_t)=\sqrt{d}\,|\partial\Vcal|_W(\mu^n_t),\]
whereas $\sqrt{d}|(\mu^n)'|_{SW}(t)=|(\mu^n)'|_{W}(t)$ for a.e. $t\in I$ by Theorem~\ref{thm:SWnear-discrete}. Consequently
\begin{align*}
    \frac{d}{dt}(d^{-1}\Vcal(\mu_t^n))=-\frac{1}{2d}|(\mu^n)'|_{W}^2(t)-\frac{1}{2d}|\partial\Vcal|_W^2(\mu_t^n) = - \frac{1}{2}|(\mu^n)'|_{SW}^2(t)-\frac12|d^{-1}\Vcal|_{SW}^2(\mu^n_t),
\end{align*}
thus $(\mu^n_t)_{t\geq 0}$ is a curve of maximal slope of $\nu\mapsto d^{-1}\Vcal(\nu)$ with respect to $SW$ or, equivalently, $\tilde\mu^n_t:=\mu^n_{dt}$ is a curve of maximal slope of $\Vcal$ with respect to $SW$. Note that we do not claim that $\tilde\mu^n_t$ is the only such $SW$ curve of maximal slope, as uniqueness is unknown.

\grn If $V$ is semiconvex, the Wasserstein gradient flow of $\V$ is stable in initial data, for instance by the Evolution Variational Inequality~\cite[Theorem 11.1.4]{AGS}. Thus, if $W(\mu^n,\mu)\rightarrow 0$, then \nc $(\mu^n_t)_t$ converge \grn with respect to $W$ \nc to the Wasserstein gradient flow starting at $\mu$, which satisfies
\[\frac{d}{dt}\Vcal(\mu_t)=-|\partial\Vcal|_{W}^2(\mu_t)=-\|\nabla V\|_{L^2(\mu_t)}^2 \text{ a.e. }t\geq 0\]
On the other hand, if further $V\in C_c^\infty(\R^d)$, Proposition~\ref{prop:metslope;potential-ac} asserts that
\[\|V\|_{\dot H^{(d+1)/2}(\R^d)}\lesssim |\partial\Vcal|_{\ell_{SW}}(\mu) \leq |\partial\Vcal|_{SW}(\mu) \lesssim \|V\|_{\dot H^{(d+1)/2}(\R^d)}\]
for suitable $\mu$ bounded away from zero on a bounded open convex set compactly containing $\supp V$.
Thus, if a $SW$ gradient flow $(\mu_t^{SW})_{t\geq 0}$ starting at $\mu$ were to exist, it must satisfy
\[\frac{d}{dt}\Vcal(\mu_t^{SW})=-|\partial\Vcal|_{SW}^2(\mu_t^{SW})\lesssim - \|V\|_{\dot H^{(d+1)/2}(\R^d)}^2 \text{ a.e. } t\geq 0.\]
Note that the Wasserstein gradient flow does not satisfy this. So the SW gradient flow (should one exist) is in this case distinct from the Wasserstein gradient flow. 
This implies that potential energy is not $\lambda$-convex in the SW geometry and furthermore  suggests that we cannot hope for stability of sliced Wasserstein gradient flows in initial data in the set of measures, even for smooth potential energies. We believe that equation
\eqref{eq:GF;lsw} may be more amenable to PDE-based approaches. 
\end{remark}

\grn
Proposition~\ref{prop:swslope;potential} also implies that the local slope $\mu\mapsto |\partial\Vcal|_{SW}(\mu)$ is not lower semicontinuous with respect to $SW$ nor the narrow convergence. This lack of regularity makes difficult rigorously studying the limit of the minimizing movements scheme as the time-step vanishes, which we do not pursue in this paper. \nc

More precisely, solutions of the minimizing movements scheme for gradient flows converge to the curve of maximal slope with respect to the relaxed slope~\cite[Chapter 2]{AGS}; \purp recall that the relaxed slope $|\partial^-\Vcal|_{m}(\mu)$ of $\V:X\rightarrow(-\infty,+\infty]$ in metric space $(X,m)$ at each $\mu\in X$ is defined by \nc
\begin{equation}\label{def:relaxed_slope}
        |\partial^-\Vcal|_{m}(\mu)=\inf\{\liminf_{k\rightarrow\infty}|\partial\Vcal|_{m}(\mu_k):\;\mu_k\rightharpoonup \mu \text{ narrowly and } \sup_k\{m(\mu,\mu_k),\Vcal(\mu_k)\}<\infty\}.
\end{equation}
Under regularity assumptions such as the $\lambda$-geodesic-convexity of $\V$ with respect to $m$, we have the equivalence $|\partial^-\Vcal|_{m}=|\partial\Vcal|_m$.
However, Proposition~\ref{prop:swslope;potential} implies that this is not the case for the SW-slopes of potential energies. Namely we show that 
even for smooth  potentials the relaxed slope in the SW metric of the potential energy coincides with $\sqrt{d}\,|\partial\V|_W(\mu)$ rather than $|\partial\V|_{SW}(\mu)$. 

\begin{corollary}[Relaxed slope of the potential energy]\label{cor:relaxed_slope;pot}
Let $V:\R^d\rightarrow\R$ be continuously differentiable with uniformly bounded derivatives. Denote by $|\partial^-\Vcal|_{SW}$ the lower semicontinuous envelope of $|\partial\Vcal|_{SW}$ with respect to the narrow topology as defined in \eqref{def:relaxed_slope}.
Then
\begin{equation}\label{eq:swslope;lscenvelope;V}
    |\partial^-\Vcal|_{SW}(\mu) = \sqrt{d}\,|\partial\Vcal|_{W}(\mu) \text{ for each } \mu\in\P_2(\R^d).
\end{equation}
\end{corollary}

\begin{remark}\label{rmk:lswslope;lscenvelope}
    \vio As $\sqrt{d}|\partial\Vcal|_W\leq |\partial\Vcal|_{\ell_{SW}}\leq |\partial\Vcal|_{SW}$ in general, \eqref{eq:swslope;lscenvelope;V} implies
    \[\sqrt{d}|\partial\Vcal|_{W}(\mu)=\sqrt{d}|\partial^-\Vcal|_W(\mu)\leq|\partial^-\Vcal|_{\ell_{SW}}(\mu)\leq |\partial^-\Vcal|_{SW}(\mu)=\sqrt{d}\,|\partial\Vcal|_{W}(\mu).\]
    In particular, $|\partial^-\Vcal|_{\ell_{SW}}(\mu)=\sqrt{d}|\partial\Vcal|_{W}(\mu)$. \nc
\end{remark}

\begin{proof}
    It is well-known~\cite[Proposition 10.4.2]{AGS} that
      \[|\partial\Vcal|_W(\mu)=\|\nabla V\|_{L^2(\mu)},\]
    which is continuous with respect to the narrow convergence when $\|\nabla V\|_\infty<\infty$. Fix any $\mu\in\P_2(\R^d)$. Then, for any narrowly converging sequence $\mu_k\rightharpoonup \mu\in\P_2(\R^d)$,
    \begin{align*}
        \liminf_{k\rightarrow\infty}|\partial\V|_{SW}(\mu_k) \geq \liminf_{k\rightarrow\infty}\sqrt{d}\,|\partial\V|_W(\mu_k)=\sqrt{d}\,|\partial\V|_W(\mu),
    \end{align*}
    thus we conclude $|\partial^-\V|_{SW}(\mu)\geq \sqrt{d}\,|\partial\V|_W(\mu)$ by taking infimum over such sequences.

    On the other hand, approximating $\mu$ by discrete measures $\mu^n$ in $SW$, and taking a suitable subsequence to ensure $\sup_n \V(\mu^n)<\infty$, we have by Proposition~\ref{prop:swslope;potential}
    \begin{align*}
        |\partial^-\V|_{SW}(\mu)\leq \liminf_{n\rightarrow\infty}|\partial\Vcal|_{SW}(\mu^n) =\liminf_{n\rightarrow\infty}\sqrt{d}\,|\partial\Vcal|_W(\mu^n) = \sqrt{d}\,|\partial\Vcal|_W(\mu).
    \end{align*}
\end{proof}
From the proof it is clear that \eqref{eq:swslope;lscenvelope;V} also holds if the lower semicontinuous envelope is defined with respect to the topology generated by $SW$.

\begin{remark}\label{rmk:relaxed_slope}
We note that in general $|\partial^-\V|_{SW}$ is not a strong upper gradient. In particular, consider $2a\leq\mu\leq b/2$ on \vio a bounded convex domain $\Omega$ \nc for some $0<4a<b<\infty$ and let $V\in C_c^\infty(\R^d)$ with $\supp V\subset\subset\Omega$. We claim that $|\partial^-\V|_{SW}=\sqrt{d}|\partial\V|_W$ fails to be an upper gradient. 
As all derivatives of $V$ are bounded, for small time $t>0$ the path $\mu_t=\mu-t(-\Delta)^{(d+1)/2}V$ satisfies $a\leq \mu_t\leq b$ on $\Omega$, and in particular remains in the space of probability measures. Furthermore, $\partial_t\mu_t -(-\Delta)^{(d+1)/2} V = 0$ and thus
\[|\mu'|_{SW}(t)\leq \sqrt{d}|\mu'|_{W}(t) \leq \|\nabla (-\Delta)^{(d-1)/2}V\|_{L^2(1/\mu_t)}\leq \sqrt{d/a} \|V\|_{\dot H^{d/2}(\R^d)}\]
and $|\partial\Vcal|_W(\mu)= \|\nabla V\|_{L^2(\mu_t)}\leq \sqrt{b}\|\nabla V\|_{L^2(\R^d)}$. On the other hand,
\begin{align*}
    \left|\frac{d}{dt}\V(\mu_t)\right|=\left|\int_{\R^d} V(x) \partial_t\mu_t\right|=\left| \int_{\R^d} V(x)(-\Delta)^{\frac{d+1}{2}} V\,dx\right| 
    = \|V\|_{\dot H^{(d+1)/2}(\R^d)}^2.
\end{align*}
Thus, choosing sufficiently oscillatory $V$ such that
\[\| V\|_{\dot H^{(d+1)/2}(\R^d)}^2 > \sqrt{\frac{d^2b}{a}}\|\nabla V\|_{L^2(\R^d)}\| V\|_{\dot H^{d/2}(\R^d)},\]
we see $\left|\frac{d}{dt}\V(\mu_t)\right|> |\partial\V|_W(\mu_t) |\mu'|_W(\mu_t)$
which verifies that $\sqrt{d}|\partial\V|_W$ is not an upper gradient.
\end{remark}

\textbf{Acknowledgements.}
The authors are grateful to Jun Kitagawa for stimulating discussions, \grn and also to Tudor Manole for pointing us to the literature that lead to Remark~\ref{rmk:VC_relative_CDF;sharp}\nc.
The authors acknowledge the support of the National Science Foundation via the grant DMS-2206069. They are also thankful to the Center for Nonlinear Analysis for its support. SP was also supported by the NSF grant DMS-2106534. \grn The authors would also like to thank the anonymous referees for careful readings and numerous helpful suggestions, which greatly helped improve the exposition of this manuscript. \nc

    \bibliographystyle{siam}
    \bibliography{SW_bib.bib}


    \begin{appendix}

    \section{Preliminaries on the Radon transform}\label{app:radon}
    In this appendix we record some basic properties of the Radon transform in further detail.
    
    \begin{remark}[Radon transform of measures and distributions]\label{rmk:radon_on_measure}
        The duality formula \eqref{eq:duality;radon} is used to extend the Radon transform to distributions. For general distributions there are ambiguities, as $R^\ast\rg$ does not necessarily decay rapidly at infinity even for $\rg\in C_c^\infty(\Pd)$ ; see~\cite[Chapter 1.5]{Hel10} and ~\cite{Ramm95}. However, for bounded measures, pushforward by the projection map $\tilde\pi^\theta(x)=x\cdot\theta$ is consistent with the duality formula. To see this, let $\rg\in C_0(\Pd)$ -- i.e. $\rg$ is continuous in $(\theta,r)$ and vanishes as $|r|\rightarrow\infty$. \red Then, it can be verified that $R^\ast \rg\in C_0(\R^d)$. Thus by Fubini's theorem
        \begin{align*}
            \langle \mu, R^\ast\rg\rangle_{\R^d}
            &= \int_{\R^{d}} \dashint_{\S^{d-1}}\rg(\theta,x\cdot\theta)d\theta\,d\mu(x)\\
            &= \dashint_{\S^{d-1}} \int_{\R^d} \rg(\theta,x\cdot\theta)\,d\mu(x)\,d\theta 
            =\dashint_{\S^{d-1}} \int_{\R} \rg(\theta,r) d\tilde\pi^\theta_\# \mu(r)\,d\theta 
            = \langle \tilde\pi^\theta_\#\mu,\rg\rangle_{\Pd}.
        \end{align*}
        In the second last equality we used the change of variables formula and that $x\cdot\theta=r$ for all $x\in(\tilde\pi^\theta)^{-1}(r)=r\theta+\theta^\perp$\nc. As $\M_b(\R^d)$ equipped with the total variation is the dual of $C_0(\R^d)$, the Radon transform can be unambiguously extended to $\mu\in\M_b(\R^d)$.

        We also note that for distributions in $H_t^s(\R^d)$ the Radon transform (or its extension) can be defined unambiguously and the duality formula can be verified; see~\cite{Sha21}, and the discussion preceding \eqref{eq:duality;radon;dist}.
    \end{remark}
    
    Another important property of the Radon transform is its relationship to the Fourier transform.
    \begin{proposition}[The Fourier slicing property]\label{prop:radon-fourier}
        For $k=1,d$, let $\F_k$  denote the $k$-dimensional Fourier transform from $\Sch(\R^k)$ to itself. Then for each $f\in\Sch(\R^d)$
        \begin{equation}\label{eq:radon-fourier}
            (2\pi)^{\tfrac{d-1}{2}}\F_d f(\theta\zeta)= \F_1 R_\theta f(\zeta) \text{ for all }\theta\in\S^{d-1} \text{ and }\zeta\in\R.
        \end{equation}
        Moreover, \eqref{eq:radon-fourier} holds a.e. for $f\in L^1(\R^d)$.
    \end{proposition}

    \begin{proof}
    By definition,
        \begin{align*}
            (2\pi)^{\tfrac{d}{2}}\F_d f(\theta\zeta) = \int_{\R^d} e^{-i \zeta\theta\cdot x} f(x)\,dx
            &=  \int_{\R} e^{-i \zeta r} \int_{x\cdot\theta =r}  f(x)\,dx\,dr\\
            &= \int_{\R} e^{-i\zeta r} R_\theta f(r)\,dr = \sqrt{2\pi}\F_1 R_\theta f(\zeta).
        \end{align*}
        Note all the equalities above are justified for a.e. $\theta\in\S^{d-1},\zeta\in\R$ when $f\in L^1(\R^d)$.
    \end{proof}
    From \eqref{eq:radon-fourier} it follows that, for $f,g\in\Sch(\R^d)$
    \begin{equation}\label{eq:radon-convolution}
        R(f\ast g)(\theta,r)=(R_\theta f\ast R_\theta g)(r),
    \end{equation}
    where $\ast$ on the left-hand side denotes the $d$-dimensional convolution and $\ast$ on the right-hand side denotes the $1$-dimensional convolution. \grn Indeed, as $\F_d(f\ast g)=(2\pi)^{d/2}\F_d f \F_d g$ and $\F_1(R_\theta f\ast R_\theta g)=(2\pi)^{1/2}\F_1 R_\theta f \F_1 R_\theta g$, we have
    \begin{align*}
        (\F_1 R_\theta(f\ast g))(\zeta)&=(2\pi)^{\frac{d-1}{2}}(\F_d(f\ast g))(\theta\zeta) =(2\pi)^{d-1/2} \F_d f(\theta\zeta) \F_d g(\theta\zeta) \\
        &= (2\pi)^{1/2}\F_1 R_\theta f(\zeta) \F_1 R_\theta g(\zeta)= \F_1 (R_\theta f\ast R_\theta g)(\theta\zeta).
    \end{align*}\nc
    Moreover, the same computation is justified when $f\in\Sch(\R^d)$ and $g\in\Sch'(\R^d)$ and $Rg\in\Sch'(\Pd)$ is well-defined. In particular, \eqref{eq:radon-convolution} holds for $f\in\Sch(\R^d)$ and $g\in H_t^s(\R^d)$ with $t\in(-\frac d2,\frac d2)$, which includes the case \grn $g\in\M_b(\R^d)$\nc.
    
    Next we record the smoothing effect of $R^\ast R$.
    \begin{proposition}[Regularizing property of the Radon transform]\label{prop:radon;reg}
        Let us denote by $A_k$ the surface area of the $(k-1)$-dimensional sphere. Then
        \begin{equation}\label{eq:radon;reg}
            R^\ast R f(x) = \frac{A_{d-1}}{A_d} \int_{\R^d}\frac{f(y)}{|y-x|}\,dy.
        \end{equation}
    \end{proposition}
    
    \begin{proof}
    By using polar coordinates and Fubini's Theorem, we see
    \begin{align*}
        (R^\ast Rf)(x) &= \dashint_{\S^{d-1}}\hat f(\theta,x\cdot\theta)d\theta
        = \dashint_{\S^{d-1}} \int_{y\in\theta^\perp} f(x+y)\,dyd\theta \\
        &= \dashint_{\S^{d-1}} \int_{\{\omega\in\S^{d-1}:\omega\perp\theta\}} \int_0^\infty f(x+r\omega) r^{d-2}\,dr\,d\omega\,d\theta    \\
        &= \frac{1}{A_d}\int_0^\infty \int_{\S^{d-1}} \int_{\{\theta\in\S^{d-1}:\omega\perp\theta\}} f(x+r\omega) r^{d-2}\,d\theta\,d\omega \,dr   \\
        &= \frac{A_{d-1}}{A_d}\int_0^\infty \dashint_{\S^{d-1}} r^{d-2} f(x+r\omega)\,d\omega\,dr  
        = \frac{A_{d-1}}{A_d}\int_{\R^d} \frac{f(y)}{|y-x|}\,dy
    \end{align*}
    \end{proof}
    
    This confirms the intuition that $R^\ast R f$ should be more regular than $f$, as for $f\in\Sch(\R^d)$
    \[\lim_{\eps\searrow 0}\int_{\R^d} |x-y|^{-d+\eps} f(y)\,dy = C(d) f(x)\]
    for some dimension dependent constant $C(d)$. To examine the regularizing property in more detail, we first introduce the Riesz potentials; see~\cite[Chapter VII]{Hel10} for further details.
    
    \begin{definition}[Riesz potential]\label{def:riesz_trnfrm}
       For $\gamma\in\R$ and $f\in\Sch(\R^d)$ we define its Riesz potential $I_d^\gamma f \in \Sch(\R^d)$ by
       \begin{equation}\label{eq:riesz_trnsfrm}
           (I_d^\gamma f)(x) = \frac{1}{H_d(\gamma)}\int_{\R^d} \purp\frac{f(y)}{|y-x|^{d-\gamma}}\nc\, \quad \text{ where } H_d(\gamma)=\vio 2^\gamma \pi^{d/2}\frac{\Gamma(\tfrac{\gamma}{2})}{\Gamma(\tfrac{d-\gamma}{2})}\nc.
       \end{equation}
    \end{definition}
    
    \begin{proposition}[Properties of the Riesz potential]\label{prop:riesz_trnsfrm;properties}
        Let $f\in\Sch(\R^d)$. Then 
        
        \begin{listi}
            \item (Lemma 6.4 of~\cite{Hel10}) $\gamma\rightarrow (I_d^\gamma f)(x)$ extends to a holomorphic function in the set $\mathbb{C_d}=\{\gamma\in\mathbb{C}:\,\gamma-d\not\in 2\mathbb{Z}^+\}$. Also
            \begin{equation}\label{eq:I0=id}
                I_d^0 f = \lim_{\gamma\rightarrow 0} I_d^\gamma f=f \text{  and  } I_d^\gamma\Delta f = \Delta I_d^\gamma f = -I_d^{\gamma-2} f
            \end{equation}
            Thus, we will understand fractional orders of $\Delta$ and $\Box$ as
            \begin{equation}\label{def:LapBox_frac}
                (-\Delta)^{s}= I_{d}^{-2s},\,(-\Box)^s = I_1^{-2s},
            \end{equation}
            
            \item (Proposition 6.5 of~\cite{Hel10}) We have the following identity. 
            \begin{equation}\label{eq:I;composition}
                I_d^\alpha(I_d^\beta f) = I_d^{\alpha+\beta} f \text{ for } f\in\Sch(\R^d) \text{ whenever } \re \alpha,\re\beta>0 \text{ and }\re(\alpha+\beta)<d.
            \end{equation}
        \end{listi}
    We often suppress the dimensional notation and simply write $I^\gamma$ when it is clear. 
    \end{proposition}
    
    We record some properties of the Radon transform; see~\cite[Chapter 1]{Hel10} for proofs. Intertwining property between the Laplacian and the Radon transform follows from direct calculations.
    \begin{proposition}[Intertwining property]\label{prop:radon_intertwine}
    For $f\in\Sch(\R^d)$ and $\rg\in\Sch(\Pd)$, we have
        \begin{equation}\label{eq:radon_intertwine}
        R{(-\Delta) f}=(-\partial_r^2) R f,\,\quad R^\ast (-\partial_r^2)\rg=-\Delta R^\ast\rg.
    \end{equation}
    \end{proposition}
    
    \begin{proof}
    To show the first item, write $\tau_h f(x)=f(x-h)$ and notice
    \[R(\tau_{-h}f)(\theta,r)=\grn\int_{y\cdot\theta=r} \tau_{-h}f(y)\,dy\nc=\int_{y\cdot\theta=r} f(y+h)\,dy = \int_{y\cdot\theta=r+h\cdot\theta} f(y)\,dy = Rf(\theta,r+h\cdot\theta),\]
    thus
    \[(R(\partial_i f))(\theta,r)=\theta_i\partial_r Rf(\theta,r),\]    
    which, along with $\sum_{i=1}^d \theta_i^2 = 1$ gives the first equality in \eqref{eq:radon_intertwine}.
    
    The second item is immediate from the definition, as
    \begin{align*}
        \partial_i R^\ast\rg (x)= \partial_i \dashint_{\S^{d-1}} \rg(\theta,x\cdot\theta)\,d\theta = \theta_i \dashint_{\S^{d-1}} \partial_r\rg(\theta,x\cdot\theta)\,d\theta = \theta_i R^\ast\partial_r\rg(x).
    \end{align*}
    \end{proof}

    We give a precise definition of the $(d-1)$-order differential operator $\Lambda_d$ involved in inversion formula for the Radon and the dual transform.

    \begin{definition}\label{def:Lambda}
    Let $\Lambda_d:\Sch(\Pd)\rightarrow \Sch(\Pd)$ be defined by
    \begin{equation}\label{def:hatLambda}
        \Lambda_d = \begin{cases}
        (-i)^{d-1}\frac{\partial^{d-1}}{\partial r^{d-1}} &\text{ when } d \text{ is odd }\\
        (-i)^{d-1}\H\frac{\partial^{d-1}}{\partial r^{d-1}} & \text{ when } d \text{ is even } \end{cases}.
    \end{equation}
    where $\H:\Sch(\S^{d-1}\times\R)\rightarrow\Sch(\S^{d-1}\times\R)$ is the Hilbert transform in the scalar variable
    \begin{equation}\label{def:Hp}
        \H \rg(\theta,r)=\frac{i}{\pi}\int_{\R} \frac{\rg(\theta,s)}{r-s}\,ds.
    \end{equation}
    \end{definition}

    \begin{remark}\label{rmk:Lambda}
    From the interaction of derivatives and the Hilbert transform with the Fourier transform, we can easily verify that for each $\rg\in\Sch(\Pd)$
    \[(\F_1\Lambda_d \rg_\theta)(\zeta)=|\xi|^{d-1}(\F_1\rg_\theta)(\zeta).\]
    Consequently, for $\rg\in\Sch(\Pd)$ we have
    \begin{align*}
        \|\Lambda_d \rg\|_{H_{t-(d-1)}^{s-(d-1)}(\Pd)}^2 &= \frac{1}{2(2\pi)^{d-1}}\int_{\S^{d-1}}\int_{\R}|\zeta|^{2t-(2d-2)}(1+\zeta^2)^{s-t}|\F_1\Lambda_d \rg_\theta(\zeta)|^2\,d\zeta\,d\theta\\
        &= \frac{1}{2(2\pi)^{d-1}}\int_{\S^{d-1}}\int_{\R}|\zeta|^{2t}(1+\zeta^2)^{s-t}|\F_1 \rg_\theta(\zeta)|^2\,d\zeta\,d\theta 
        = \|\rg\|_{H_{t}^{s}(\Pd)}^2.
    \end{align*}
    Thus this we may extend $\Lambda_d$ as a bijective linear isometry from $H_t^s(\Pd)$ to $H_{t-(d-1)}^{s-(d-1)}(\Pd)$ when $t>d-3/2$.
    \end{remark}

    For Schwartz functions, we have the following inversion formulae~\cite[Theorem 8.1]{Sol87}.
     \begin{proposition}[Inversion formula for the Radon and the dual transform]\label{prop:radon_inversion}
        For all $f\in\Sch(\R^d)$
        \begin{equation}\label{eq:radon_inversion_f2}
          c_d f=R^\ast \Lambda_d R f,
        \end{equation}
        where $c_d=(4\pi)^{(d-1)/2}\Gamma(d/2)/\Gamma(1/2)$. Similarly, for all $\rg\in\Sch(\Pd)$
        \begin{equation}\label{eq:dual_inversion}
            c_d \rg(\theta,p)= R R^\ast(\Lambda_d \rg).
        \end{equation}
        Here, $\Sch(\Pd)$ is defined as in \eqref{def:Sch_radon}. Furthermore, $R^\ast \Lambda_d \rg\in C^\infty$ with $(R^\ast \Lambda_d \rg)(x)=O(|x|^{-d})$, hence $f\in L^p(\R^d)$ for $p\in (1,+\infty]$.
    \end{proposition}
    
    \begin{remark}[Formal derivation of the inversion formula]\label{rmk:dual_inversion;derivation}
        Let $\rg\in\Sch(\Pd)$. Denote by $\F_k$ with $k=1,d$ the Fourier transform in $k$-dimensions. Recalling $\F_1 (R_\theta f)(\zeta)=(2\pi)^{\tfrac{d-1}{2}}\F_d f(\zeta\theta)$, define $f$ in the Fourier domain by
        \[\F_d f(\zeta\theta):=(2\pi)^{\tfrac{1-d}{2}}\F_1 \rg_\theta(\zeta).\]
        As $\rg$ is even, the above is well-defined, and we have
        \[\F_1(R_\theta f)(\zeta) = \F_1 \rg_\theta(\zeta).\]
        By injectivity of the Fourier transform $\F_1$, we have $Rf=\rg$. Thus the key to proving the inversion formulae is to justify the Fourier and the inverse Fourier transforms, which comes down to regularity of the functions.
        
        Formally, we can find an expression for $f$ following the argument in Theorem 12.6 of~\cite{SmiSol77}. For any test function $\varphi\in C_c^\infty(\R^d)$,
        \begin{align*}
            \langle R\varphi,\Lambda_d \rg\rangle_{L^2(\Pd)}
            &= C\int_{\S^{d-1}}\int_{\R} \F_1(R_\theta \varphi)(\zeta) |\zeta|^{d-1} \F_1 \rg_\theta(\zeta)\,d\zeta\,d\theta \qquad    \text{ by the Plancherel formula for } \F_1\\
            &= C \int_{\S^{d-1}}\int_{\R} (\F_d \varphi)(\zeta\theta) |\zeta|^{d-1} (\F_d f)(\zeta\theta)\,d\zeta\,d\theta \qquad  \text{ as } \F_d \varphi(\zeta\theta)=\F_1(R_\theta \varphi)(\zeta)    \\  
            &= C \langle \F_d \varphi, \F_d f \rangle_{L^2(\R^d)} 
            = C\langle \varphi, f\rangle_{L^2(\R^d)}
        \end{align*}
        where we have used polar coordinates and the Plancherel formula for the Fourier transform. Thus
        \[\langle\varphi,R^\ast\Lambda_d\rg\rangle_{L^2(\Pd)}=\langle R\varphi,\Lambda_d \rg\rangle_{L^2(\Pd)} = C\langle \varphi,f\rangle_{L^2(\R^d)},\]
        and we may conclude, for for some constant $C$ only depending on the dimension $d$,
        \[f=C R^\ast\Lambda_d \rg\]
        As $Rf=\rg$, this gives us \eqref{eq:radon_inversion_f2}, and yields \eqref{eq:dual_inversion} by applying $R$ on both sides
    \end{remark}
    
    We end this section with a few results on the Sobolev spaces with attenuated/amplified low frequencies. We first note $H_t^s(\R^d)$ continuously embeds in $\Sch'(\R^d)$~\cite[Theorem 5.3]{Sha21}. 
    \begin{theorem}[$H_t^s(\R^d)$ and $\Sch'(\R^d)$]\label{thm:Hts;tempered_dist}
        Let $t\in (-d/2,d/2)$ and $r,s\in\R$, the identity map of $\Sch(\R^d)$ extends to the continuous embedding $H_t^s(\R^d)\subset\Sch'(\R^d)$. In other words, $H_t^s(\R^d)$ consists of tempered distributions.
    \end{theorem}
\begin{proofsketch}
    In general, for all $s,t\in\R$ one can show
    \begin{equation}\label{eq:Hts;duality}
        \left|\int_{\R^d} fg \right| \leq \|f\|_{H_t^s(\R^d)}\|g\|_{H_{-t}^{-s}(\R^d)} \text{ for all } f,g\in\Sch(\R^d);
    \end{equation}
    see~\cite[Theorem 5.3]{Sha21} for further details.
    Whenever $t\in(-d/2,d/2)$, for each $g\in\Sch(\R^d)$ we have $\|g\|_{H_{-t}^{-s}(\R^d)}<\infty$. Thus, for any $f\in H_t^s(\R^d)$, we can unambiguously define its action on each $g\in\Sch(\R^d)$ as a limit of actions by the approximating sequence -- i.e.
    \begin{align*}
        f(g) := \lim_{k\rightarrow\infty} \int_{\R^d} f_k(x) g(x)\,dx \quad\text{ for any } (f_k)_{k\geq 1} \text{ in } \Sch(\R^d) \text{ such that } f_k\xrightarrow[]{H_{t}^s} f.
    \end{align*}
    Furthermore, for each fixed $g\in\Sch(\R^d)$ the estimate \eqref{eq:Hts;duality} is preserved for $f\in H_t^s(\R^d)$, hence we deduce $f\in \Sch'(\R^d)$.
\end{proofsketch}
    
Sharafutdinov also established the following supercritical Sobolev-embedding type result~\cite[Theorem 5.4]{Sha21}.
\begin{theorem}[$H_t^s(\R^d)$ and continuous functions]\label{thm:Hts;cts}
    If $t\in(-\frac d2,\frac d2)$, $s>t+\frac d2$, then $H_t^s(\R^d)$ consists of bounded continuous functions.
\end{theorem}

\section{Continuity equation in each projection}\label{app:CE_proj}
\purp 
In this section we provide a proof of Lemma~\ref{lem:CE_proj}, which was used to establish Theorem ~\ref{thm:SW2_ac-curves}~(i). Our proof relies on the $H_t^{(q,s)}(\Omega)$-norms for $\Omega=\R^d,\Pd$ -- introduced by Sharafutdinov~\cite{Sha21} -- which generalize the $H_t^s(\Omega)$-norms defined in \eqref{def:Hts_norm} and \eqref{def:Hts_norm;radon} to include regularity in the direction $\theta\in\S^{d-1}$. In the simple case where $t=0$ and $q,s$ are nonnegative integers, the $H_t^{(q,s)}(\Pd)$-norm is the sum of $L^2$-norms of derivatives of order $\leq q$ in the $\theta$-variable and derivatives of order $\leq s$ in the scalar variable. In this section we provide minimal details necessary to prove Lemma~\ref{lem:CE_proj} and refer the interested readers to~\cite[Section 3]{Sha21} and references therein for further information. 

Let $Y_l\in C^\infty(\S^{d-1})$ be a spherical harmonic of degree $l$ if $Y_l=\tilde Y_l|_{\S^{d-1}}$ for a homogeneous polynomial $\tilde Y_l$ of degree $l$ on $\R^d$ satisfying $\Delta\tilde Y_l=0$. The space of spherical harmonics of degree $l$ on $\S^{d-1}$ has finite dimension $N(d,l)$, thus we can choose an orthonormal basis $(Y_{lm})_{m=1}^{N(d,l)}$ for the space. Then the spherical harmonics of degree $l$ are eigenfunctions of the spherical Laplacian $\Delta_\theta:C^\infty(\S^{d-1})\rightarrow C^\infty(\S^{d-1})$, where sign is chosen to ensure that $\Delta_\theta$ is positive definite,
\[\Delta_\theta Y_{lm}=\lambda(d,l)Y_{lm},\qquad \lambda(d,l)=l(l+d-2).\]

We can represent the Fourier transform of each $f\in\Sch(\R^d)$ by
\begin{equation}\label{def:Fdf;spherical}
    \F_d f(\xi)=\sum_{l=0}^\infty\sum_{m=1}^{N(d,l)} \bar f_{lm}(|\xi|)Y_{lm}(\xi/|\xi|)
\end{equation}
where the coefficients $\bar f_{lm}\in C^\infty([0,+\infty))$ and decays fast at infinity. Similarly, for each $\rg\in\Sch(\Pd)$
\begin{equation}\label{def:F1g;spherical}
    \F_1 \rg(\theta,\zeta)=\sum_{l=0}^\infty\sum_{m=1}^{N(d,l)} \bar \rg_{lm}(\zeta)Y_{lm}(\theta)
\end{equation}
with coefficients $\bar\rg_{lm}\in\Sch(\R)$ satisfying $\bar\rg_{lm}(-\zeta)=(-1)^l\bar\rg_{lm}(\zeta)$.

For any $r,s\in\R$ and $t>-d/2$, the $H_t^{(q,s)}(\R^d)$-norm is defined by
\begin{equation}\label{def:Htqs-norm;Rd}
    \|f\|_{H_t^{(q,s)}(\R^d)}^2 = \sum_{l=0}^\infty(\lambda(d,l)+1)^q \sum_{m=1}^{N(d,l)}\int_0^\infty \zeta^{2t+d-1}(1+\zeta^2)^{s-t}|\bar f_{lm}(\zeta)|^2\,d\zeta.
\end{equation}
The norm is independent of the choice of the orthonormal basis; see~\cite[Sections 3-4]{Sha21}. Similarly, for $r,s\in\R$ and $t>-1/2$ the $H_t^{(q,s)}(\Pd)$-norm is defined by
\begin{equation}\label{def:Htqs-norm;Pd}
    \|\rg\|_{H_t^{(q,s)}(\Pd)}^2 = \frac{1}{2(2\pi)^{d-1}}\sum_{l=0}^\infty(\lambda(d,l)+1)^q \sum_{m=1}^{N(d,l)}\int_{\R} |\zeta|^{2t}(1+\zeta^2)^{s-t}|\bar\rg_{lm}(\zeta)|^2\,d\zeta.
\end{equation}
The spaces $H_t^{(q,s)}(\R^d)$ and $H_t^{(q,s)}(\Pd)$ are respectively the closures of $\Sch(\R^d),\Sch(\Pd)$ under the corresponding norm.

In fact, $H_t^{(q,s)}(\Omega)$-norm for $\Omega=\R^d,\Pd$ with $q=0$ is exactly the $H_t^s(\Omega)$-norms; see the proof of~\cite[Theorem 5.1]{Sha21}. Moreover, when $q\geq 0$, $H_t^{(q,s)}(\Omega)$ is continuously embedded in $H_t^s(\Omega)$. Hence for $t\in(-d/2,d/2)$ for $\Omega=\R^d$ and $t\in (-1/2,1/2)$, $H_t^{(q,s)}(\Omega)$ is continuously embedded in $\Sch'(\Omega)$.

Sharafutdinov showed~\cite[Theorem 4.3]{Sha21} that the Radon transform extends to a bijective Hilbert space isometry between $H_t^{(q,s)}(\Omega)$-spaces. Namely, for all $q,s\in\R$ and $t>-d/2$
\begin{equation}\label{eq:Htqs;isometry}
    \|f\|_{H_t^{(q,s)}(\R^d)}=\|Rf\|_{H_{t+(d-1)/2}^{(q,s+(d-1)/2)}(\Pd)} \text{ for all } f\in H_t^{(q,s)}(\R^d).
\end{equation}

The crucial property we use in this section is the supercritical Sobolev-embedding-type inequality for $H_t^{(q,s)}(\Pd)$ spaces, which is due to Sharafutdinov~\cite[Corollary 5.11]{Sha21}. While the proof is omitted, the result readily follows from the analogous arguments for $H^q(\S^{d-1})$ and $H_t^s(\R)$~\cite[Corollary 5.5]{Sha21}.
\begin{theorem}[Supercritical Sobolev embedding for $H_t^{(q,s)}(\Pd)$]\label{thm:HtrsPd;Ck}
    If $t\in(-1/2,1/2)$, $s>t+1/2+k$, and $q>(d-1)/2+k$ then $H_t^{(q,s)}(\Pd)\subset C^k(\Pd)$ is a continuous embedding.
\end{theorem}

We now present a proof of Lemma~\ref{lem:CE_proj}.
\begin{proof}[Proof of Lemma~\ref{lem:CE_proj}]
Note that by hypothesis, for  a.e. $\theta\in\S^{d-1}$ and $t\in I$ we have $\hat\mu_t^\theta\in\P_2(\R)$ and $\widehat J_t^\theta\in\M_b(\R;\R^d)$. It suffices to show that for a.e $\theta\in\S^{d-1}$
\begin{equation}\label{eq:CE_theta}
    \int_I \alpha'(t)\langle \hat\mu_t^\theta,\psi\rangle_{\R} + \alpha(t)\langle \theta\cdot\widehat J_t^\theta,\partial_r\psi\rangle_{\R}\,dt \text{ for all } \alpha\in C_c^\infty(I) \text{ and } \psi\in C_c^\infty(\R).
\end{equation}
Indeed, linear combinations of test functions of the form $(t,x)\mapsto\alpha(t)\psi(r)$ are dense in $C_c^\infty(I\times K)$ for every compact $K\subset\R$ hence this implies \eqref{eq:CE_dist;proj}; see~\cite[Proposition 4.2 and Exercise 4.23]{San15} for instance.

\vspace{3mm}
\noindent\emph{Step 1$^o$.} We first show that 
\begin{equation}\label{eq:CE_radon}
    \int_I \alpha'(t)\langle \mu_t,\varphi_n\rangle_{\R^d} + \alpha(t)\langle J_t,\nabla\varphi_n\rangle_{\R^d}\,dt=0 \text{ for all } \alpha\in C_c^\infty(I) \text{ and } \rg\in C_c^\infty(\Pd).
\end{equation}
To this end, we first note by Proposition~\ref{prop:radon;reg} that $R^\ast \rg= R^\ast R R^{-1}\rg= cI_d^{d-1} R^{-1}\rg$ for some dimension-dependent constant $c>0$, where $I^{d-1}_d=(-\Delta)^{-(d-1)/2}$ (see Definition~\ref{def:riesz_trnfrm}). Fix any $q>(d+1)/2$. Then by \eqref{eq:Htqs;isometry} we can find some constant $C=C(d)>0$ such that
\[\|R^\ast \rg\|_{H_{(d-1)/2}^{(q,2+(d-1)/2)})(\R^d)}=\|cI_d^{d-1} R^{-1}\rg\|_{H_{(d-1)/2}^{(q,2+(d-1)/2)}(\R^d)}=C\|R^{-1}\rg\|_{H_{-(d-1)/2}^{(q,2-(d-1)/2)}(\R^d)}=C\|\rg\|_{H^{(q,2)}(\Pd)}.\]
As $\rg\in C_c^\infty(\Pd)\subset H^{(q,2)}(\Pd)$, $R^\ast\rg\in H_{(d-1)/2}^{(q,2+(d-1)/2)}(\R^d)$. Hence we can choose a sequence $\varphi_n\in \Sch(\R^d)$ such that $\|\varphi_n-R^\ast \rg\|_{H_{(d-1)/2}^{(q,2+(d-1)/2)}(\R^d)}\xrightarrow[]{n\rightarrow\infty} 0$. 

Observe that this implies
\begin{equation}\label{eq:rg;radon_approx}
    \|c_d^{-1}\Lambda_d R\varphi_n-\rg\|_{C^1(\Pd)}\lesssim_{d,q} \|c_d^{-1}\Lambda_d R\varphi_n-\rg\|_{H^{(q,2)}(\Pd)}\lesssim_d \|\varphi_n-R^\ast \rg\|_{H_{(d-1)/2}^{(q,2+(d-1)/2)}(\R^d)}\xrightarrow[]{n\rightarrow\infty} 0.
\end{equation}
Indeed, the first inequality is a consequence of Theorem~\ref{thm:HtrsPd;Ck} applied to $s=2,t=0$ and $q>(d+1)/2$. The second inequality is in fact an equality up to a constant; from definition of the $H_t^{(q,s)}(\R^d)$ norms, the isometry \eqref{eq:Radon_isometry;Hts}, and the intertwining property $R(-\Delta)^{(d-1)/2}=\tilde c\Lambda_d R$, we have
\begin{align*}
    &\|\varphi_n-R^\ast \rg\|_{H_{(d-1)/2}^{(q,2+(d-1)/2)}(\R^d)}
    =\|(-\Delta)^{(d-1)/2}(\varphi_n-R^\ast \rg)\|_{H_{-(d-1)/2}^{(q,2-(d-1)/2)}(\R^d)}    \\
    &= \|R(-\Delta)^{(d-1)/2}(\varphi_n-R^\ast \rg)\|_{H^{(q,2)}(\Pd)}
    = \tilde c\|\Lambda_d R\varphi_n-\Lambda_d R R^\ast \rg\|_{H^{(q,2)}(\Pd)}  \\
    &= \tilde c c_d \|c_d^{-1}\Lambda_d R\varphi_n-c_d^{-1}\Lambda_d R R^\ast \rg\|_{H^{(q,2)}(\Pd)}
    = \tilde c c_d \|c_d^{-1}\Lambda_d R\varphi_n-\rg\|_{H^{(q,2)}(\Pd)}.
\end{align*}
In the last line we have used the inversion formula \eqref{eq:dual_inversion}.

Without loss of generality, we can choose $\varphi_n\in C_c^\infty(\R^d)$, for instance by noting that $C_c^\infty(\R^d)$ is dense in $\Sch(\R^d)$ in the Schwartz topology and that $\Sch(\R^d)\subset H_{(d-1)/2}^{(q,2+(d-1)/2)}(\R^d)$ is a continuous embedding. Thus
\[\int_I \alpha'(t)\langle \mu_t,\varphi_n\rangle_{\R^d} + \alpha(t)\langle J_t,\nabla\varphi_n\rangle_{\R^d}\,dt=0.\]
Recalling $R\nabla\varphi_n=\theta\partial_r R\varphi_n$, and applying the duality formulae for bounded measures \eqref{eq:duality;radon;measure} and distributions \eqref{eq:duality;radon;dist},
\begin{align*}
    c_d^{-1}\int_I \alpha'(t)\langle \hat\mu_t,\Lambda_d R\varphi_n\rangle_{\Pd} + \alpha(t)\langle \theta\cdot\widehat J_t,\partial_r \Lambda_d R\varphi_n\rangle_{\Pd}\,dt =0
\end{align*}
Let $\tilde I$ be the compact interval containing the support of $\alpha$. Then
\begin{align*}
    &\left|\int_I \alpha'(t)\langle\hat\mu_t,\rg\rangle_{\Pd}+\alpha(t)\langle \theta\cdot\widehat J_t,\partial_r\rg\rangle_{\Pd}\,dt\right|\\
    &\qquad=\left|\int_{\tilde I} \alpha'(t)\langle\hat\mu_t,c_d^{-1}\Lambda_d R\varphi_n-\rg\rangle_{\Pd}+\alpha(t)\langle\theta\cdot\widehat J_t,\partial_r(c_d^{-1}\Lambda_d R\varphi_n - \rg)\rangle_{\Pd}\,dt\right|  \\
    &\qquad\leq \|\alpha\|_{C^1(\tilde I)}(|\hat\mu|_{TV(\tilde I\times\Pd)}+|\widehat J|_{TV(\tilde I\times\Pd)})\|c_d^{-1}\Lambda_d R\varphi_n - \rg\|_{C^1(\Pd)}.
\end{align*}
By Definition~\ref{def:CE} (ii), $|\hat\mu|_{TV(\tilde I\times\Pd)}+|\widehat J|_{TV(\tilde I\times\Pd)}<+\infty$. Thus from \eqref{eq:rg;radon_approx} we obtain \eqref{eq:CE_radon}.

\vspace{3mm}
\noindent\emph{Step 2$^o$.} 
Let $B_{\S^{d-1}}(\theta_0,\eps)=\{\omega\in\S^{d-1}:|\omega-\theta_0|\leq\eps\}$, and let $B_{\S^{d-1}}^e(\theta_0,\eps)=B_{\S^{d-1}}(\theta,\eps)\cap B_{\S^{d-1}}(-\theta,\eps)$. Then $\one_{B_{\S^{d-1}}^e(\theta_0,\eps)}\in L^1(\S^{d-1})$ is an even function. A standard approximation argument using smooth cutoff function in the $\theta$-variable using the continuity equation in the Radon space \eqref{eq:CE_radon} we deduce that for any $\theta_0\in\S^{d-1}$, $\eps>0$, and $\rg\in C_c^\infty(\Pd)$
\begin{equation}\label{eq:CE_radon_anglecutoff}
        \int_{B_{\S^{d-1}}^e(\theta_0,\eps)}\int_I  \alpha'(t)\langle\hat\mu_t^\theta,\rg^\theta\,\rangle_{\R} + \alpha(t)\langle\theta\cdot\widehat J_t^\theta,\partial_r\rg^\theta\rangle_{\R}\,dt\,d\theta=0.
\end{equation}
Indeed, as no derivatives in $\theta$ appear in the continuity equation, the passage to the limit is justified by the dominated convergence theorem.

As the integrand in \eqref{eq:CE_radon_anglecutoff} with respect to $d\theta$ is clearly $L^1$, for each $\rg\in C_c^\infty(\Pd)$ and $\alpha\in C_c^\infty(I)$ the Lebesgue differentiation theorem yields that there exists a null set $\mathcal{N}_{\rg,\alpha}\subset\S^{d-1}$ such that
\begin{align*}
    \int_I  \alpha'(t)\langle\hat\mu_t^\theta,\rg^\theta\,\rangle_{\R} + \alpha(t)\langle\theta\cdot\widehat J_t^\theta,\partial_r\rg^\theta\rangle_{\R}\,dt=0 \text{ for all } \theta\not\in\mathcal{N}_{\rg}.
\end{align*}
By separability of $C_c^\infty(\Omega)$ for $\Omega=I,\Pd$ we can find a null set $\mathcal{N}\subset\S^{d-1}$ such that the above holds for all $\rg\in C_c^\infty(\Pd)$ and $\alpha\in C_c^\infty(I)$.

As for every $\psi\in C_c^\infty(\R)$ one can find $\rg\in C_c^\infty(\Pd)$ with $\rg^\theta=\psi$ we conclude that for all $\theta\not\in\mathcal{N}$
\[\int_I  \alpha'(t)\langle\hat\mu_t^\theta,\psi\,\rangle_{\R} + \alpha(t)\langle \theta\cdot\widehat J_t^\theta,\partial_r\psi\rangle_{\R}\,dt=0 \text{ for all } \alpha\in C_c^\infty(I), \psi\in C_c^\infty(\R).\]
\end{proof}

\nc

\end{appendix} 

\end{document}